\documentclass[a4paper,11pt]{amsart}

\protected\def\ignorethis#1\endignorethis{}
\let\endignorethis\relax

\def\TOCstop{\addtocontents{toc}{\ignorethis}}
\def\TOCstart{\addtocontents{toc}{\endignorethis}}

\usepackage[top=3.0cm, bottom=3.0cm, inner=3.0cm, outer=3.0cm, includefoot]{geometry}

\usepackage[english]{babel}
\usepackage{amsmath}
\DeclareMathOperator*{\argmin}{arg\,min}
\usepackage{amsfonts}
\usepackage{mathtools}
\usepackage{amsthm}
\usepackage{csquotes}
\usepackage{tikz}
\usepackage{theoremref}
\usepackage{oplotsymbl}
\usepackage{relsize}
\usepackage{diagbox}
\usepackage{subcaption}
\usepackage{interval}

\intervalconfig{
soft open fences,
}

\usepackage[backend=biber]{biblatex}
\newtheorem{theorem}{Theorem}[section]
\newtheorem{corollary}[theorem]{Corollary}
\newtheorem{lemma}[theorem]{Lemma}
\newtheorem{proposition}[theorem]{Proposition}
\newtheorem{conjecture}[theorem]{Conjecture}

\theoremstyle{definition}
\newtheorem{definition}{Definition}[section]

\theoremstyle{remark}
\newtheorem*{remark}{Remark}

\numberwithin{equation}{section}
\title
{
Ollivier-Ricci curvature of regular graphs
}

\author{Moritz Hehl}

\usepackage{authblk}

\date{\today}

\begin{document}

\maketitle

\pagestyle{plain}
\pagenumbering{arabic}

\begin{abstract}
    We derive explicit formulas for the Lin-Lu-Yau curvature 
    and the $\alpha$-Ollivier-Ricci curvature in terms of graph parameters and an optimal assignment. Utilizing these precise expressions, we examine the relationship between the Lin-Lu-Yau curvature $\kappa$ and the $0$-Ollivier-Ricci curvature $\kappa_{0}$, resulting in an equality condition on regular graphs. This condition allows us to characterize edges that are bone idle in regular graphs of girth four and to construct a family of bone idle graphs with this girth. We then use our formulas to provide an efficient implementation of the Ollivier-Ricci curvature on regular graphs, enabling us to identify all bone idle, regular graphs with fewer than 15 vertices.  Moreover, we establish a rigidity theorem for cocktail party graphs, proving that a regular graph is a cocktail party graph if and only if its Lin-Lu-Yau curvature is equal to one. Furthermore, we present a condition on the degree of a regular graph that guarantees positive Ricci curvature. We conclude this work by discussing the maximal number of vertices that a $d$-regular graph with positive Lin-Lu-Yau curvature can have.
\end{abstract}

\tableofcontents

\section{Introduction}

Since F. Gauss and B. Riemann introduced the geometric notion of curvature over 150 years ago, it played a central role in differential geometry. The curvature of a space refers to the degree to which it deviates from being flat. Ricci curvature is one of the various notions that are of importance. 
It is a fundamental concept in differential geometry, particularly
in the study of Riemannian manifolds. Broadly speaking, Ricci curvature measures the extent to which the geometry of a Riemannian manifold deviates locally from that of a Euclidean space. It plays a crucial role in various areas of mathematics and theoretical physics. In the general theory of relativity, it appears in Einstein's field equations, where the Ricci curvature tensor is part of the Einstein tensor, which describes the curvature of spacetime. It is also a crucial ingredient for the Ricci flow, introduced in Richard Hamilton's groundbreaking work from 1982 \cite{hamilton1982three}. Recently, the Ricci flow has led to the proof of the Poincaré conjecture and the more general Geometrization conjecture. Gromov's work \cite{gromov1981structures} demonstrates the important role of the Ricci curvature tensor in Riemannian geometry. 

Given the importance of Ricci curvature in differential 
geometry, there is a clear interest in extending this concept to metric spaces more general than Riemannian manifolds. This motivation has led to the proposal of various generalized notions of curvature for non-smooth or discrete structures. This contrasts with Riemannian geometry, where we have a unique notion of Ricci curvature. Every approach captures a different aspect of the Riemannian notion, and each is thus suited for a different purpose.

In their 1985 paper \cite{Bakry1985}, Bakry and Émery proposed a generalized notion within 
the broad framework of a Markov semigroup, the \textit{Bakry-Émery curvature}. It is motivated
by the curvature dimension inequality and the Bakry-Émery $\Gamma$ calculus. The Bakry-Émery theory
applied to graphs has received considerable interest (see, e.g., \cite{Lin2010Ricci} and \cite{Liu2018Bakry}).

Another approach, taken by Erbar and Maas \cite{erbar2012ricci} and by Mielke \cite{Mielke2013GeodesicCO}, 
uses the convexity of the entropy to introduce a discrete Ricci curvature notion. 

Forman's work \cite{Forman2003Bochner}, introduces a discretization of the classical Ricci curvature on CW complexes, the \textit{Forman-Ricci curvature}. This definition is based on the Bochner-Weitzenböck formula. Sreejith et al. \cite{sreejith2016forman} adapted Forman's notion to the case of graphs.

The Forman-Ricci curvature is highly correlated \cite{samal2018comparative} with, and in certain cases even coincides \cite{jost2021characterizations} with, the curvature notion we consider in this work, the \textit{Ollivier-Ricci curvature}. In 2009, Ollivier \cite{ollivier2009ricci} developed this discrete Ricci curvature notion on metric spaces equipped with Markov chains or with a measure.

In the work by von Renesse and Sturm \cite{vonRenesseMax-K.2005Tige}, it was shown that Riemannian manifolds with positive Ricci curvature are characterized by the property that small spheres are closer to one another than their centers are. Namely, let $M$ be a Riemannian manifold. Consider a tangent vector $\omega=(xy)$. Let $\nu_{x}$ be another tangent vector at $x$ and let $\nu_{y}$ be the tangent vector at $y$ obtained by parallel transport of $\nu_{x}$ along $\omega$ from $x$ to $y$. If the geodesic starting at $x$ in direction $\nu_{x}$ and the 
geodesic starting at $y$ in direction $\nu_{y}$ converge, then the sectional curvature is positive. If the geodesics are parallel, then sectional curvature is zero and if they diverge it is negative. The Ricci curvature at $\omega=(xy)$ is obtained by averaging over all directions $\nu_{x}$ at $x$. Now, a point on a small sphere $S_{x}$ around $x$ can be thought of as a direction $\nu_{x}$. Thus, on average, Ricci curvature determines if a point on $S_{x}$ and the corresponding point on $S_{y}$ are closer than $x$ and $y$ or further away. Figure \ref{sectional_curvature} \cite{ollivier2009ricci} illustrates this key observation.
       
\begin{figure}
    \centering
        
    \begin{tikzpicture}[x=0.75pt,y=0.75pt,yscale=-1,xscale=1]
        
        \draw    (378.5,98) .. controls (418.61,105.54) and (455.07,104.62) .. (485.67,100.2) .. controls (544.35,91.71) and (581.5,70.34) .. (581.5,71) ;
        \draw    (382.5,176) .. controls (421.06,170.42) and (455.64,173.14) .. (484.66,179.23) .. controls (544.06,191.71) and (580.17,218.34) .. (579.5,217) ;
        \draw   (443.75,101.25) .. controls (443.75,80.4) and (460.65,63.5) .. (481.5,63.5) .. controls (502.35,63.5) and (519.25,80.4) .. (519.25,101.25) .. controls (519.25,122.1) and (502.35,139) .. (481.5,139) .. controls (460.65,139) and (443.75,122.1) .. (443.75,101.25) -- cycle ;
        \draw   (447.75,180.25) .. controls (447.75,159.4) and (464.65,142.5) .. (485.5,142.5) .. controls (506.35,142.5) and (523.25,159.4) .. (523.25,180.25) .. controls (523.25,201.1) and (506.35,218) .. (485.5,218) .. controls (464.65,218) and (447.75,201.1) .. (447.75,180.25) -- cycle ;
        \draw  [dash pattern={on 0.84pt off 2.51pt}]  (519.5,94) -- (523.5,190) ;
        \draw [shift={(523.5,190)}, rotate = 87.61] [color={rgb, 255:red, 0; green, 0; blue, 0 }  ][fill={rgb, 255:red, 0; green, 0; blue, 0 }  ][line width=0.75]      (0, 0) circle [x radius= 3.35, y radius= 3.35]   ;
        \draw [shift={(519.5,94)}, rotate = 87.61] [color={rgb, 255:red, 0; green, 0; blue, 0 }  ][fill={rgb, 255:red, 0; green, 0; blue, 0 }  ][line width=0.75]      (0, 0) circle [x radius= 3.35, y radius= 3.35]   ;
        \draw  [dash pattern={on 0.84pt off 2.51pt}]  (482.01,99.48) -- (484.66,179.23) ;
        \draw [shift={(484.66,179.23)}, rotate = 88.09] [color={rgb, 255:red, 0; green, 0; blue, 0 }  ][fill={rgb, 255:red, 0; green, 0; blue, 0 }  ][line width=0.75]      (0, 0) circle [x radius= 3.35, y radius= 3.35]   ;
        \draw [shift={(482.01,99.48)}, rotate = 88.09] [color={rgb, 255:red, 0; green, 0; blue, 0 }  ][fill={rgb, 255:red, 0; green, 0; blue, 0 }  ][line width=0.75]      (0, 0) circle [x radius= 3.35, y radius= 3.35]   ;
        \draw    (96.5,70) .. controls (190.5,54) and (289.5,105) .. (289.5,103) ;
        \draw    (94.5,186) .. controls (184.5,207) and (285.5,161) .. (285.5,160) ;
        \draw  [dash pattern={on 0.84pt off 2.51pt}]  (164.5,69) -- (165.5,191) ;
        \draw [shift={(165.5,191)}, rotate = 89.53] [color={rgb, 255:red, 0; green, 0; blue, 0 }  ][fill={rgb, 255:red, 0; green, 0; blue, 0 }  ][line width=0.75]      (0, 0) circle [x radius= 3.35, y radius= 3.35]   ;
        \draw [shift={(164.5,69)}, rotate = 89.53] [color={rgb, 255:red, 0; green, 0; blue, 0 }  ][fill={rgb, 255:red, 0; green, 0; blue, 0 }  ][line width=0.75]      (0, 0) circle [x radius= 3.35, y radius= 3.35]   ;
        \draw   (126.57,69) .. controls (126.57,48.05) and (143.55,31.07) .. (164.5,31.07) .. controls (185.45,31.07) and (202.43,48.05) .. (202.43,69) .. controls (202.43,89.95) and (185.45,106.93) .. (164.5,106.93) .. controls (143.55,106.93) and (126.57,89.95) .. (126.57,69) -- cycle ;
        \draw   (127.57,191) .. controls (127.57,170.05) and (144.55,153.07) .. (165.5,153.07) .. controls (186.45,153.07) and (203.43,170.05) .. (203.43,191) .. controls (203.43,211.95) and (186.45,228.93) .. (165.5,228.93) .. controls (144.55,228.93) and (127.57,211.95) .. (127.57,191) -- cycle ;
        \draw  [dash pattern={on 0.84pt off 2.51pt}]  (202.5,74) -- (203.5,186) ;
        \draw [shift={(203.5,186)}, rotate = 89.49] [color={rgb, 255:red, 0; green, 0; blue, 0 }  ][fill={rgb, 255:red, 0; green, 0; blue, 0 }  ][line width=0.75]      (0, 0) circle [x radius= 3.35, y radius= 3.35]   ;
        \draw [shift={(202.5,74)}, rotate = 89.49] [color={rgb, 255:red, 0; green, 0; blue, 0 }  ][fill={rgb, 255:red, 0; green, 0; blue, 0 }  ][line width=0.75]      (0, 0) circle [x radius= 3.35, y radius= 3.35]   ;
        
        \draw (469,85.4) node [anchor=north west][inner sep=0.75pt]    {$x$};
        \draw (470,180.4) node [anchor=north west][inner sep=0.75pt]    {$y$};
        \draw (426,68.4) node [anchor=north west][inner sep=0.75pt]    {$S_{x}$};
        \draw (428,190.4) node [anchor=north west][inner sep=0.75pt]    {$S_{y}$};
        \draw (525,126.4) node [anchor=north west][inner sep=0.75pt]    {$ >d( x,y)$};
        \draw (489,83.4) node [anchor=north west][inner sep=0.75pt]    {$\nu _{x}$};
        \draw (494,184.4) node [anchor=north west][inner sep=0.75pt]    {$\nu _{y}$};
        \draw (154,55.4) node [anchor=north west][inner sep=0.75pt]    {$x$};
        \draw (150,193.4) node [anchor=north west][inner sep=0.75pt]    {$y$};
        \draw (112,208.4) node [anchor=north west][inner sep=0.75pt]    {$S_{y}$};
        \draw (110,35.4) node [anchor=north west][inner sep=0.75pt]    {$S_{x}$};
        \draw (175,54.4) node [anchor=north west][inner sep=0.75pt]    {$\nu _{x}$};
        \draw (174,190.4) node [anchor=north west][inner sep=0.75pt]    {$\nu _{y}$};
        \draw (205,114.4) node [anchor=north west][inner sep=0.75pt]    {$< d( x,y)$};     
    \end{tikzpicture}              
    \caption{Positive vs. negative sectional curvature.}
    \label{sectional_curvature}
\end{figure}
With this characterization in mind, we now turn to Ollivier's generalization to the framework of metric spaces.
Instead of the sphere $S_{x}$ centered at a point $x$, Ollivier uses probability measures $m_{x}$ supported in the neighborhood of $x$. More precisely,
let $(X,d)$ be a Polish metric space, equipped with its Borel $\sigma$-algebra. A family of probability measures
$m = \{m_{x}: x \in X\}$ is called a random walk $m$ on $X$ if the two technical assumptions are satisfied \cite{ollivier2009ricci}:
\begin{itemize}
    \item[$(i)$] The measure $m_{x}$ depends measurably on $x\in X$.
    \item[$(ii)$] Each measure $m_{x}$ has finite first moment, i.e., for some (hence any) $z \in X$ and for any $x \in X$: $\int d(z,y)dm_{x}(y) < \infty$.
\end{itemize}

As an analogue for the average distance between corresponding points on $S_{x}$ and $S_{y}$,  Ollivier uses the Wasserstein distance $W_{1}$ between the probability measures $m_{x}$ and $m_{y}$. Thus, optimal transport theory is an essential ingredient for this notion of discrete Ricci curvature. In what follows, we introduce some basic concepts of the optimal transport theory and refer the interested reader to the very comprehensive book by Villani \cite{villani2003topics}. 

Following the case of Riemannian manifolds, the comparison of the distance of $m_{x}$ and $m_{y}$ with the distance of $x$ and $y$ determines the sign of the Ollivier Ricci curvature. Namely, in the case that the probability measures $m_{x}$ and $m_{y}$ are closer than $x$ and $y$, Ricci curvature will be positive, otherwise negative. This leads to the following definition.

\begin{definition}[Ollivier-Ricci curvature, \cite{ollivier2009ricci}]
    Let $(X,d)$ be a metric space with a collection of probability measures $m=\{m_{x}: x \in X\}$. 
    For two distinct points $x,y \in X$, the \textit{Ollivier-Ricci curvature} of $(X,d,m)$ along 
    $(x,y)$ is defined as 
    \begin{equation*}
        \kappa(x,y)  = 1- \frac{W_{1}(m_{x}, m_{y})}{d(x,y)}.
    \end{equation*}
\end{definition}
 
Rather than thinking of $m_{x}$ as an analogue for the sphere centered at $x$, Ollivier proposes another, more probabilistic, way to think about it. That is, as defining a Markov chain with $n$-step transition probability from $x$ to $y$ given by:
\begin{equation*}
    p_{xy}^{(n)} = \int_{z \in X} p_{xz}^{(n-1)}dm_{z}(y)
\end{equation*}
where $p_{xy}^{(1)} = m_{x}(y)$.

In \cite{ollivier2009ricci}, Ollivier shows that up to some scaling factor, this definition captures the Riemannian Ricci curvature in the case of a Riemannian manifold equipped with the Riemannian volume measure. This refines Theorem 1.5 in the paper by von Renesse 
and Sturm \cite{vonRenesseMax-K.2005Tige}.

The Ollivier-Ricci curvature in the context of locally finite graphs has recently garnered significant attention. This has led to numerous compelling theoretical results, extending various classical theorems from Riemannian geometry to the graph setting. These include spectral gap estimates, concentration of measure, and log-Sobolev inequalities \cite{ollivier2010survey}.

On the other hand, this discrete curvature notion was found to be beneficial in the analysis of complex networks.
In the realm of network science, graph curvature holds general interest because numerous real-world networks exhibit various forms of geometry \cite{boguna2021network}. The Ollivier-Ricci curvature can offer valuable insights into the geometry of these networks. It has recently found applications in numerous practical contexts, e.g., in the analysis of the internet topology \cite{ni2015ricciinternet}, as an economic indicator for market fragility and systemic risk \cite{Sandhu2016systematic}, to differentiate cancer networks \cite{Sandhu2015cancer} or in the study of brain structural networks \cite{Farooq2019hallmark}. 

Furthermore, Ollivier-Ricci curvature finds applications in machine learning. In \cite{ni2019community} and \cite{Sia2019community}, the authors propose Ollivier-Ricci curvature based approaches to community detection. Regarding graph neural networks (GNNs), Ollivier-Ricci curvature provides a unified framework for studying over-squashing and over-smoothing, two inherent problems of these models \cite{nguyen2023revisiting}. Recently, new GNN architectures, incorporating graph curvature, were proposed \cite{Ye2020Curvature,li2022curvature}. These curvature-based graph neural networks (CGNNs) outperform state-of-the-art methods on various synthetic and real-world networks.

On graphs, Ollivier's notion of Ricci curvature $\kappa_{\alpha}$ takes values on edges and depends on an idleness parameter $\alpha \in [0,1]$. Ollivier considered idleness parameters $\alpha=0$ and $\alpha=\frac{1}{2}$. In 2011, Lin, Lu, and Yau \cite{lin2011ricci} introduced a modification of the Ollivier-Ricci curvature, by computing the derivative of the curvature with respect to the idleness parameter. We will refer to this modification as \textit{Lin-Lu-Yau curvature}. The relationship between the Lin-Lu-Yau curvature and the original $\alpha$-Ollivier-Ricci curvature was studied by Bourne et al. in \cite{bourne2018ollivier}. One of the main contributions of this work will be to provide answers to open questions proposed in this paper. Namely, we will give a necessary and sufficient condition on regular graphs for $\kappa_{0}(x,y) =\kappa(x,y)$. Using this, we will provide a necessary and sufficient condition for an edge to be bone idle on a graph of girth four. That means $\kappa_{\alpha}(x,y) = 0$ for any $\alpha \in [0,1]$. We furthermore construct an infinite family of bone idle, regular graphs with a girth of four. Moreover, we prove that $\kappa(x,y) > 0$ implies on regular graphs that $\kappa_{0}(x,y) \geq 0$ and we construct non-regular graphs with edges $x\sim y$ satisfying $\kappa(x,y) > 0$ and $\kappa_{0}(x,y) < 0$. 

To this end, we derive explicit Ricci curvature formulas for regular graphs. Bonini et al. \cite[Theorem 4.6]{bonini2020condensed} derive an exact formula for the Lin-Lu-Yau curvature on strongly regular graphs in terms of graph parameters and the size of a maximal matching in the core neighborhood. Münch et al. \cite[Theorem 2.6]{munch2019ollivier} show an exact formula for a modified curvature notion on combinatorial graphs. There is an overlap between some of our methods and theirs.

To my knowledge, precise expressions for the Lin-Lu-Yau curvature and the $\alpha$-Ollivier-Ricci curvature on arbitrary regular graphs have not yet been established. Using these precise expressions, we provide an efficient implementation for the Ollivier-Ricci curvature on regular graphs. Additionally, we use these results to establish a rigidity theorem for cocktail party graphs and to give a condition on the degree of a regular graph, guaranteeing that it is of positive Ricci curvature.

\subsection{Outline}

In Chapter 2, we introduce the necessary concepts of graph theory and optimal transport theory. In Chapter 3, the Ollivier-Ricci curvature on graphs, as well as the modification by Lin, Lu, and Yau, will be presented. In Chapter 4, we derive our precise expressions for the different curvature notions. Chapter 5 is devoted to the study of the relationship between the Lin-Lu-Yau curvature $\kappa$ and the $0$-Ollivier-Ricci curvature $\kappa_{0}$. In Chapter 6, we present the implementation of the Ollivier-Ricci curvature on regular graphs and the results obtained by applying it to exhaustive lists of regular graphs with a given number of vertices and degree.

\section{Preliminaries}

We begin by reviewing some fundamental concepts of graph theory and optimal transport theory. Those seeking a thorough introduction to these areas are encouraged to consult the works of Diestel \cite{diestel2017graph} for graph theory and Villani \cite{villani2003topics} for optimal transport theory.

\subsection{Graph theory}

A \textit{simple graph} $G=(V,E)$ is an unweighted, undirected graph that contains no multiple edges or self-loops. We call $V$ the \textit{vertex set} and $E$ the \textit{edge set}. For two vertices $x,y \in V$ we denote the existence of an edge between $x$ and $y$ by $x \sim y$. 

Denote by $d(x,y)$ the \textit{shortest path distance} between $x$ and $y$, i.e., the number of edges in a shortest path connecting them. If no such path exists, $d(x,y)$ is defined to be infinity. For two nonempty sets $A,B\subset V$, we define 
\begin{equation*}
    dist(A,B) = \inf\{d(x,y): x \in A, y \in B\}.
\end{equation*}
The \textit{diameter} of a graph $G$ is defined as the length of the longest shortest path, that is,
\begin{equation*}
    diam(G) = \max_{x,y\in V} d(x,y).
\end{equation*}
The \textit{girth} of a graph is the length of a shortest cycle contained in the graph. The girth
is defined to be infinity if the graph does not contain any cycles. \\
Given $x \in V$, define 
\begin{equation*}
    S_{r}(x) := \{y \in V: d(x,y) = r\}, \quad B_{r}(x) := \{y \in V: d(x,y) \leq r\}.
\end{equation*} 
For an edge $x \sim y$, we denote by $\triangle(x,y)$ the set of common neighbors of $x$ and $y$, that is,
\begin{equation*}
    \triangle(x,y) = S_{1}(x) \cap S_{1}(y).
\end{equation*}
Denote by $d_{x}$ the \textit{degree} of $x$, that is, $d_{x}= |S_{1}(x)|$. The graph $G$ is called \textit{locally finite} if every vertex has finite degree. The graph $G$ is said to be \textit{$d$-regular} if every vertex has the same degree $d$. 

A $d$-regular graph $G$ with $n$ vertices is called \textit{strongly regular} with parameters $(n,d,\lambda, \mu)$ if every two adjacent vertices have $\lambda \geq 0$ common neighbors and every two nonadjacent vertices have $\mu \geq 1$ common neighbors.

We end this section by introducing another important class of graphs. A simple graph $G=(V,E)$ is called \textit{bipartite}, if there is a partition of $V$ into two disjoint subsets $V=A \cup B$, such that every edge connects a vertex in $A$ to a vertex in $B$. If the partition $V=A \cup B$ is specified, we denote the bipartite graph by $G=(A,B,E)$.

In what follows, $G=(V,E)$ is always assumed to be a locally finite, connected and simple graph.

\subsection{Optimal transport theory}

A probability measure on the vertex set $V$ is a map $\mu: V \to [0,1]$ that satisfies $\sum_{x \in V} \mu(x)=1$. Denote by $\mathcal{P}(V)$ the set of probability measures on $V$. We now introduce a metric on $\mathcal{P}(V)$, the 1-Wasserstein distance.

\begin{definition}[1-Wasserstein distance]
    Let $G=(V,E)$ be a locally finite graph. The \textit{1-Wasserstein distance} between
    two probability measures $\mu_{1}$, $\mu_{2}$ on $V$ is defined as 
    \begin{equation} \label{eq:1}
        W_{1}^{G}(\mu_{1}, \mu_{2}) = \inf_{\pi \in \Pi(\mu_{1}, \mu_{2})} \sum_{x \in V} \sum_{y \in V} d(x,y)  \pi(x,y),
    \end{equation}
    where 
    \begin{equation*}\Pi(\mu_{1}, \mu_{2}) = \left \{ \pi: V \times V \to [0,1]: \sum_{y \in V} \pi(x,y)=\mu_{1}(x), \quad \sum_{x \in V} \pi(x,y)=\mu_{2}(y) \right \}. \end{equation*}
\end{definition}

Intuitively, imagine two distributions as piles of earth. The 1-Wasserstein distance measures the minimal effort needed to transform one pile into another. Due to this analogy, the metric is also called the \textit{earth mover's distance}. We call $\pi \in \Pi(\mu_{1}, \mu_{2})$ a transport plan and $W_{1}^{G}(\mu_{1}, \mu_{2})$ is the minimal effort which is required for such a transport. If the infimum in \refeq{eq:1} is attained, we call $\pi$ an \textit{optimal transport plan} transporting $\mu_{1}$ to $\mu_{2}$. 

Hereafter, we will omit the superscript from $W_{1}^{G}$ when the graph is evident from the context. 

The following Lemma, which is intuitively clear, allows us to impose some assumptions on an optimal transport plan. It states that it is not necessary to move any mass that is shared between the probability measures. 

\begin{lemma}\thlabel{dontmove}
    Let $G=(V,E)$ be a locally finite graph and let $\mu_{1}, \mu_{2}$ be two probability measures on V.
    Then there exists an optimal transport plan $\pi$ transporting $\mu_{1}$ to $\mu_{2}$ satisfying
    \begin{equation*}
        \pi(x,x) = \min\{\mu_{1}(x), \mu_{2}(x)\}
    \end{equation*}
    for all $x \in V$.
\end{lemma}

This is an immediate consequence of the triangle inequality. For more details, we refer to \cite[Lemma 4.1]{bourne2018ollivier}.

The Kantorovich duality is a fundamental concept in optimal transport theory (see Theorem 1.3 in \cite{villani2003topics}). Here, we introduce it in our specific graph setting. For that, we define the notion of 1-Lipschitz functions. 

\begin{definition}[1-Lipschitz function]
    Let $G=(V,E)$ be a locally finite graph. A function $f: V \to \mathbb{R}$ is called \textit{1-Lipschitz} if 
    \begin{equation*}
        |f(x)-f(y)| \leq d(x,y), \quad \forall x,y \in V. 
    \end{equation*}
    The set of all 1-Lipschitz functions on $V$ is denoted by 1-Lip.
\end{definition}

\begin{theorem}[Kantorovich duality]
    Let $G=(V,E)$ be a locally finite graph. Let $\mu_{1}$ and $\mu_{2}$ be two probability measures on $V$. Then
    \begin{equation} \label{eq:2}
        W_{1}(\mu_{1}, \mu_{2}) = \sup_{f \in \text{1-Lip}} \sum_{x \in V} f(x)(\mu_{1}(x) - \mu_{2}(x)).
    \end{equation}
\end{theorem}

We call $f \in \text{1-Lip}$ an \textit{optimal Kantorovich potential} transporting $\mu_{1}$ to $\mu_{2}$ if the supremum in \refeq{eq:2} is attained. 

When dealing with 1-Lipschitz functions, the following extension theorem of McShane \cite{McSHANE1934ExtensionOR} proves to be useful. We state it here in the graph setting.

\begin{theorem}[McShane]\thlabel{McShane}
    Let $G=(V,E)$ be a locally finite graph. Suppose that $U \subset V$ and that $f: U \to \mathbb{R}$ is a 1-Lipschitz function. Then there exists an extension of $f$, i.e., a function 
    $\bar{f}: V \to \mathbb{R}$ such that $\bar{f}\vert_{U} = f$, which is 1-Lipschitz.
\end{theorem}

In what follows, probability measures under which all points in their support carry equal mass will be of importance, i.e.,
\begin{equation*}
    \mu_{1} = \frac{1}{n}\sum_{i=1}^{n}\delta_{x_{i}}, \quad \mu_{2} = \frac{1}{n}\sum_{j=1}^{n}\delta_{y_{j}}.
\end{equation*}
Here, $\delta_{x}$ denotes the Dirac measure centered on some fixed vertex $x \in V$:
\begin{equation*}
    \delta_{x}(y) = \begin{cases}
        1, & \text{if $x = y$;}\\
        0, & \text{otherwise.}\\
    \end{cases}
\end{equation*}
In this case, any transport plan $\pi \in \Pi(\mu_{1}, \mu_{2})$ can be represented as a bistochastic $n \times n$ matrix $\pi = (\pi_{ij})_{i,j=1, \dots, n}$. Namely, bistochasticity means that $\pi_{ij} \geq 0$ for any  $i,j \in \{1, \dots, n\}$, and 
\[
    \forall i: \;\sum_{j=1}^{n}\pi_{ij}=1; \; \text{and} \; \; \forall j: \; \sum_{i=1}^{n}\pi_{ij}=1.
\]
We denote by $\mathcal{B}_{n}$ the set of all bistochastic $n \times n$ matrices. Thus, the Wasserstein distance reduces to 

\begin{equation*}
    W_{1}(\mu_{1}, \mu_{2}) = \inf_{\pi \in \mathcal{B}_{n}}\frac{1}{n}\sum_{i,j=1}^{n}\pi_{ij}d(x_{i},y_{j}).
\end{equation*}

Therefore, calculating the Wasserstein distance is a linear optimization problem on the bounded and convex set $\mathcal{B}_{n}$. \textit{Choquet's minimization theorem} states that some extremal points of $\mathcal{B}_{n}$ are optimal. \textit{Birkhoff's theorem} states that these extremal points are permutation matrices, i.e., 
\begin{equation*}
    \pi_{ij}= \delta_{\sigma(i),j} = \begin{cases}
        1, & \text{if $\sigma(i) = j$;}\\
        0, & \text{otherwise.}\\
    \end{cases}
\end{equation*}
for some permutation $\sigma$ of $\{1, \dots,n\}$. Here $\delta_{ij}$ is the Kronecker symbol. We denote by $S_{n}$ the set of all permutations of $\{1, \dots, n\}$. Thus, we obtain the following Lemma.

\begin{lemma}\thlabel{Wasserstein_distance_easy}
    Let $G=(V,E)$ be a locally finite graph. Let $\mu_{1} = \frac{1}{n}\sum_{i=1}^{n}\delta_{x_{i}}$ and $\mu_{2} = \frac{1}{n}\sum_{i=1}^{n}\delta_{y_{i}}$, where $\{x_{1}, \dots, x_{n}\} \subset V$ and  $\{y_{1}, \dots, y_{n}\} \subset V$. Then the Wasserstein distance between $\mu_{1}$ and $\mu_{2}$ reduces to
    \begin{equation*}
        W_{1}(\mu_{1}, \mu_{2}) = \inf_{\sigma \in S_{n}} \frac{1}{n} \sum_{i=1}^{n} d(x_{i}, y_{\sigma(i)}).
    \end{equation*}
\end{lemma}

That means that there exists an optimal transport plan that corresponds to a one-to-one assignment between the source vertices $x_{i}$ and the target vertices $y_{j}$. We refer to such an optimal transport plan as an \textit{optimal assignment}. In other words, the Kantorovich problem coincides with the Monge problem in this special case.

\section{Ollivier-Ricci curvature on graphs}

In this chapter, we shall introduce Ollivier's discrete notion of Ricci curvature on graphs, as well as the modification introduced by Lin, Lu, and Yau.

\subsection{Definitions}

To introduce Ollivier's notion of Ricci curvature on graphs, we define the probability measures $\mu_{x}^{\alpha}$ for $x \in V$ and $\alpha \in [0,1]$ by:

\begin{equation*}
    \mu_{x}^{\alpha}(y) =
    \begin{cases}
        \alpha, & \text{if $y = x$;}\\
        \frac{1-\alpha}{d_{x}}, & \text{if $y \sim x$;}\\
        0, & \text{otherwise.}
    \end{cases}
\end{equation*}

\begin{definition}[Ollivier-Ricci curvature]
    Let $G = (V,E)$ be a locally finite graph. For two distinct vertices $x,y \in V$, we define the \textit{$\alpha$-Ollivier-Ricci curvature} by
    \begin{equation*}
        \kappa_{\alpha}^{G}(x,y) = 1 - \frac{W_{1}^{G}(\mu_{x}^{\alpha}, \mu_{y}^{\alpha})}{d(x,y)}.
    \end{equation*}
    The parameter $\alpha$ is called the \textit{idleness}.
    If $x \sim y$ we call $\kappa_{\alpha}^{G}(x,y)$ \textit{short scale}, and \textit{long scale}
    otherwise.
\end{definition}

\begin{remark}
    Note that $W(\mu_{x}^{1}, \mu_{y}^{1})=d(x,y)$ for any $x,y\in V$. Therefore, we obtain $\kappa_{1}^{G}(x,y)= 0$ for any $x,y \in V$.
\end{remark}

Ollivier considered idleness parameters $\alpha=0$ and $\alpha=\frac{1}{2}$. In \cite{bourne2018ollivier}, the authors study the short scale Ollivier-Ricci curvature as a function of the idleness parameter. In \cite{cushing2019long}, this investigation is extended to the long scale Ollivier-Ricci curvature. To this end, they introduce the following function.

\begin{definition}[Ollivier-Rici idleness function]
    Let $G=(V,E)$ be a locally finite graph. For two distinct vertices $x,y \in V$, 
    the function $\alpha \mapsto \kappa_{\alpha}^{G}(x,y)$
    is called the \textit{idleness function}.
\end{definition}

It was first shown by Lin, Lu, and Yau in \cite{lin2011ricci}, that the idleness function is concave. Using that $\kappa_{1}^{G}(x,y) = 0$, this implies that the function $h(\alpha) = \kappa_{\alpha}^{G}(x,y)/(1-\alpha)$ is increasing over the interval $[0,1)$. They further showed, that $h(\alpha)$ is bounded and thus, 
the limit $\lim_{\alpha \to 1} \kappa_{\alpha}^{G}(x,y)/(1-\alpha)$ exists. Lin, Lu, and Yau used this result, to introduce a modification of the Ollivier-Ricci curvature.

\begin{definition}[Lin-Lu-Yau curvature]
    Let $G=(V,E)$ be a locally finite graph. For two distinct vertices $x,y \in V$, the
    \textit{Lin-Lu-Yau curvature} is defined as
    \begin{equation*}
        \kappa^{G}(x,y) = \lim_{\alpha \to 1} \frac{\kappa_{\alpha}^{G}(x,y)}{1-\alpha}.
    \end{equation*}
\end{definition}

Recall that $\kappa_{1}^{G}(x,y)=0$ for any $x,y \in V$. Thus, $\kappa^{G}(x,y)$ is the negative of the derivative of $\kappa_{\alpha}^{G}(x,y)$ with respect to the idleness parameter at $\alpha=1$. 

Hereafter, we will omit the superscripts from $\kappa_{\alpha}^{G}$ and $\kappa^{G}$ when the graph is evident from the context. 

We write $\textit{Ric}(G) \geq k$ if $\kappa(x,y) \geq k$ for all $x,y \in V$. The subsequent Lemma states that proving $\kappa(x,y) \geq k$ for adjacent vertices $x \sim y$ is sufficient to show that $\textit{Ric}(G) \geq k$.

\begin{lemma}[\cite{lin2011ricci}, Lemma 2.3]
    Let $G=(V,E)$ be a locally finite graph. If $\kappa(x,y) \geq k$ for any  edge $(x,y) \in E$,  then $\kappa(x,y) \geq k$ for 
    any two vertices $x,y \in V$.
\end{lemma}

Therefore, in what follows, we will only consider $\kappa(x,y)$ for adjacent vertices $x \sim y$. We write $\textit{Ric}(G) = k$ if $\kappa(x,y) = k$ for all edges $x \sim y$.

\subsection{Comparing the Ollivier-Ricci curvature and its modification}

In \cite{bourne2018ollivier}, Bourne et al. studied the relationship between the $\alpha$-Ollivier-Ricci curvature and its modification by Lin, Lu, and Yau. The crucial element in establishing the relation is the following theorem concerning the structure of the idleness function $\alpha \mapsto \kappa_{\alpha}(x,y)$.

\begin{theorem}[\cite{bourne2018ollivier}, Theorem 4.4]
    Let $G=(V,E)$ be a locally finite graph and let $x,y \in V$ with $x \sim y$. Then $\alpha \mapsto \kappa_{\alpha}(x,y)$ is linear over $\left[\frac{1}{max\{d_{x},d_{x}\} +1},1\right]$. 
\end{theorem}

Therefore, $\kappa_{\alpha}(x,y)$ does not change its slope on the interval $\left[\frac{1}{max\{d_{x},d_{x}\} +1},1\right]$. Thus, the subsequent theorem is an immediate consequence of the mean value theorem, using that $\kappa_{1} = 0$ and that $\frac{\partial}{\partial \alpha} \kappa_{\alpha}\big\vert_{\alpha=1} = -\kappa$.

\begin{theorem}\thlabel{relation_curvature}
    Let $G=(V,E)$ be a locally finite graph and let $x,y \in V$ with $x \sim y$. Let  $\alpha \in \left[\frac{1}{max\{d_{x},d_{y}\} +1},1\right]$ be arbitrary. Then
    \begin{equation*}
        \kappa_{\alpha}(x,y) = (1-\alpha)\kappa(x,y). 
    \end{equation*}
\end{theorem}

Hence, the Lin-Lu-Yau curvature coincides up to a scaling factor with the $\alpha$-Ollivier-Ricci curvature for $\alpha$ large enough.

\subsection{The core neighborhood}

For two adjacent vertices $x \sim y$, the Ollivier-Ricci curvature is fully determined by the 1-Wasserstein distance between the probability measures $\mu_{x}^{\alpha}$ and $\mu_{y}^{\alpha}$. These measures have support $B_{1}(x)$ and $B_{1}(y)$, respectively.
Observe that for every $i \in S_{1}(x)$ and $j \in S_{1}(y)$ we have $d(i,j) \leq 3$, as $i \sim x \sim y \sim j$. Thus, only vertices with a distance less than or equal to two from $x$ and $y$ are relevant for the optimal transport problem. Therefore, the Wasserstein distance, and hence the Ollivier-Ricci curvature, is a local property. We define the \textit{core neighborhood} of an edge $x \sim y$ as 
\begin{equation*}
    N_{xy} = S_{1}(x) \cup S_{1}(y) \cup P_{xy},
\end{equation*}
where $P_{xy}= \{i \in V: d(x,i) = d(y,i) = 2\}$. According to the above reasoning, the Ollivier-Ricci curvature should only depend on the core neighborhood of an edge. The following Lemma makes this statement precise.

\begin{lemma}[Reduction Lemma]
    Let $G=(V,E)$ be a locally finite graph. Given two adjacent vertices $x,y \in V$, denote by $G_{xy}$ the subgraph of $G$ induced by $N_{xy}$. Let $\alpha \in [0,1]$ be arbitrary, then 
    \begin{equation*}
        \kappa_{\alpha}^{G}(x,y) = \kappa_{\alpha}^{G_{xy}}(x,y).
    \end{equation*}
    Furthermore,
    \begin{equation*}
        \kappa^{G}(x,y) = \kappa^{G_{xy}}(x,y).
    \end{equation*}
\end{lemma}

\begin{figure}
    \center
    \begin{tikzpicture}[x=0.75pt,y=0.75pt,yscale=-1,xscale=1]
        
        \draw   (51.7,87.68) .. controls (51.69,83.96) and (54.83,80.94) .. (58.71,80.94) .. controls (62.58,80.94) and (65.72,83.96) .. (65.72,87.68) .. controls (65.72,91.4) and (62.58,94.42) .. (58.71,94.42) .. controls (54.84,94.42) and (51.7,91.41) .. (51.7,87.68) -- cycle ;
        \draw   (51.69,197.9) .. controls (51.69,194.18) and (54.83,191.16) .. (58.7,191.16) .. controls (62.58,191.16) and (65.72,194.18) .. (65.72,197.9) .. controls (65.72,201.62) and (62.58,204.64) .. (58.7,204.64) .. controls (54.83,204.64) and (51.69,201.62) .. (51.69,197.9) -- cycle ;
        \draw    (58.71,94.42) -- (58.7,191.16) ;
        \draw   (152.32,138.03) .. controls (152.32,127.08) and (161.55,118.21) .. (172.94,118.21) .. controls (184.33,118.21) and (193.56,127.08) .. (193.56,138.03) .. controls (193.56,148.98) and (184.33,157.85) .. (172.94,157.85) .. controls (161.55,157.85) and (152.32,148.98) .. (152.32,138.03) -- cycle ;
        \draw   (150.67,50.02) .. controls (150.67,39.07) and (159.9,30.19) .. (171.29,30.19) .. controls (182.68,30.19) and (191.91,39.07) .. (191.91,50.02) .. controls (191.91,60.97) and (182.68,69.84) .. (171.29,69.84) .. controls (159.9,69.84) and (150.67,60.97) .. (150.67,50.02) -- cycle ;
        \draw   (153.97,225.25) .. controls (153.97,214.31) and (163.2,205.43) .. (174.59,205.43) .. controls (185.98,205.43) and (195.21,214.31) .. (195.21,225.25) .. controls (195.21,236.2) and (185.98,245.08) .. (174.59,245.08) .. controls (163.2,245.08) and (153.97,236.2) .. (153.97,225.25) -- cycle ;
        \draw    (64.48,83.32) -- (150.67,50.02) ;
        \draw    (65.72,197.9) -- (153.97,225.25) ;
        \draw    (65.72,87.68) -- (151.91,132.48) ;
        \draw    (65.72,197.9) -- (152.73,145.96) ;
        \draw    (171.29,69.84) -- (172.94,118.21) ;
        \draw    (172.94,157.85) -- (174.59,205.43) ;
        \draw    (179.54,69.05) .. controls (240.99,132.48) and (240.99,133.27) .. (181.6,206.22) ;
        \draw   (255.01,137.24) .. controls (255.01,126.29) and (264.24,117.42) .. (275.63,117.42) .. controls (287.02,117.42) and (296.25,126.29) .. (296.25,137.24) .. controls (296.25,148.19) and (287.02,157.06) .. (275.63,157.06) .. controls (264.24,157.06) and (255.01,148.19) .. (255.01,137.24) -- cycle ;
        \draw    (189.44,59.53) -- (257.48,126.14) ;
        \draw    (193.56,138.03) -- (255.01,137.24) ;
        \draw    (261.61,151.51) -- (189.02,212.57) ;
        \draw  [dash pattern={on 4.5pt off 4.5pt}] (22,9.5) -- (313.57,9.5) -- (313.57,269.5) -- (22,269.5) -- cycle ;
        \draw  [dash pattern={on 0.84pt off 2.51pt}] (354.4,138.82) .. controls (354.4,127.88) and (363.63,119) .. (375.02,119) .. controls (386.41,119) and (395.64,127.88) .. (395.64,138.82) .. controls (395.64,149.77) and (386.41,158.65) .. (375.02,158.65) .. controls (363.63,158.65) and (354.4,149.77) .. (354.4,138.82) -- cycle ;
        \draw  [dash pattern={on 0.84pt off 2.51pt}]  (191.91,50.02) -- (360.59,122.97) ;
        \draw  [dash pattern={on 0.84pt off 2.51pt}]  (195.21,225.25) -- (362.24,155.48) ;
        \draw  [dash pattern={on 0.84pt off 2.51pt}]  (296.25,137.24) -- (354.4,138.82) ;
        
        \draw (39.09,78.89) node [anchor=north west][inner sep=0.75pt]    {$x$};
        \draw (38.27,189.11) node [anchor=north west][inner sep=0.75pt]    {$y$};
        \draw (133.5,16.56) node [anchor=north west][inner sep=0.75pt]  [font=\footnotesize]  {$S_{1}( x) \backslash \triangle ( x,y)$};
        \draw (134.33,246.51) node [anchor=north west][inner sep=0.75pt]  [font=\footnotesize]  {$S_{1}( y) \backslash \triangle ( x,y)$};
        \draw (264.69,130.12) node [anchor=north west][inner sep=0.75pt]  [font=\footnotesize]  {$P_{x}{}_{y}$};
        \draw (100.91,132.88) node [anchor=north west][inner sep=0.75pt]  [font=\footnotesize]  {$\triangle ( x,y)$};
        \draw (315.57,22.39) node [anchor=north west][inner sep=0.75pt]    {$G_{x}{}_{y}$};
        \draw (396.77,129.43) node [anchor=north west][inner sep=0.75pt]    {$G\backslash G_{x}{}_{y}$};
    \end{tikzpicture}
    \caption{The core neighborhood of an edge $x \sim y$.}
    \label{core_neighborhood}          
\end{figure}

The illustration in figure \ref{core_neighborhood} of the core neighborhood of an edge $x \sim y$ is due to Bhattacharya et al. \cite{bhattacharya2015exact}.

In their paper, they proved the reduction Lemma for $\alpha = 0$. We extend this result to the case of an arbitrary $\alpha \in [0,1]$ and to the Lin-Lu-Yau curvature. To do so, we follow the same line of reasoning.

\begin{proof}
    Let $\alpha \in [0,1]$ be arbitrary. By the Kantorovich duality
    \begin{equation*}
        W_{1}^{G}(\mu_{x}^{\alpha}, \mu_{y}^{\alpha}) = \sup_{f \in \text{1-Lip}} \sum_{z \in V} f(z)(\mu_{x}^{\alpha}(z) - \mu_{y}^{\alpha}(z)).
    \end{equation*}
    Observe that $d_{G}(i,j) \leq d_{G_{xy}}(i,j)$ for any $i,j \in V(G_{xy})$. Hence, any function which is 1-Lipschitz with respect to $d_{G}$ is also 1-Lipschitz with respect to $d_{G_{xy}}$. Therefore,
    \begin{equation*}
        W_{1}^{G_{xy}}(\mu_{x}^{\alpha}, \mu_{y}^{\alpha}) \geq W_{1}^{G}(\mu_{x}^{\alpha}, \mu_{y}^{\alpha}).
    \end{equation*}
    To show the other inequality, let $f: V(G_{xy}) \to \mathbb{R}$ be an optimal Kantorovich potential on $G_{xy}$, namely
    \begin{equation*}
        W_{1}^{G_{xy}}(\mu_{x}^{\alpha}, \mu_{y}^{\alpha}) = \sum_{z \in V} f(z)(\mu_{x}^{\alpha}(z) - \mu_{y}^{\alpha}(z)).
    \end{equation*}
    We define a function $g: V(G) \to \mathbb{R}$ which is 1-Lipschitz with respect to $d_{G}$ and coincides with $f$ on $S_{1}(x) \cup S_{1}(y)$. To this end, define $g=f$ on $S_{1}(x) \cup S_{1}(y)$. By the construction of $G_{xy}$, we obtain $d_{G}(i,j)=d_{G_{xy}}(i,j)$, for any $i,j \in S_{1}(x) \cup S_{1}(y)$. As $f$ is an optimal Kantorovich potential on $G_{xy}$, it is 1-Lipschitz with respect to $d_{G_{xy}}$. Therefore,
    \begin{equation*}
        |g(i) -g(j)| = |f(i)-f(j)| \leq d_{G_{xy}}(i,j) = d_{G}(i,j)
    \end{equation*}
    for any $i,j \in S_{1}(x) \cup S_{1}(y)$. According to \thref{McShane}, $g$ can be extended to a 1-Lipschitz function $\bar{g}$ on $V$. We conclude by
    \begin{align*}
        W_{1}^{G}(\mu_{x}^{\alpha}, \mu_{y}^{\alpha}) &\geq \sum_{z \in V} \bar{g}(z)(\mu_{x}^{\alpha}(z) - \mu_{y}^{\alpha}(z)) \\
        & = \sum_{z \in S_{1}(x) \cup S_{1}(y)} g(z)(\mu_{x}^{\alpha}(z) - \mu_{y}^{\alpha}(z)) \\
        & = \sum_{z \in S_{1}(x) \cup S_{1}(y)} f(z)(\mu_{x}^{\alpha}(z) - \mu_{y}^{\alpha}(z)) \\
        & = W_{1}^{G_{xy}}(\mu_{x}^{\alpha}, \mu_{y}^{\alpha}),
    \end{align*}
    where we use that $\text{supp}(\mu_{x}^{\alpha}) \cup \text{supp}(\mu_{y}^{\alpha}) = S_{1}(x) \cup S_{1}(y)$. Therefore,
    \begin{equation*}
        \kappa_{\alpha}^{G}(x,y) = 1 - W_{1}^{G}(\mu_{x}^{\alpha}, \mu_{y}^{\alpha}) = 1 - W_{1}^{G_{xy}}(\mu_{x}^{\alpha}, \mu_{y}^{\alpha}) = \kappa_{\alpha}^{G_{xy}}(x,y).
    \end{equation*}
    As this equality holds for any $\alpha \in [0,1]$, we obtain $\kappa^{G}(x,y) = \kappa^{G_{xy}}(x,y)$.
\end{proof}

Focusing solely on the core neighborhood when calculating the Ricci curvature of an edge considerably simplifies the computation.

\section{Exact expressions of Ollivier-Ricci curvature}

The aim of this chapter is to derive precise expressions of the Lin-Lu-Yau curvature and the $\alpha$-Ollivier-Ricci curvature of an edge $x\sim y$ with $d_{x} = d_{y}$. These expressions are formulated in terms of the graph parameters and an optimal assignment of vertices in the core neighborhood. 

Bourne et al. \cite{bourne2018ollivier} showed that the idleness function $\alpha \mapsto \kappa_{\alpha}(x,y)$ is piecewise linear. If $x$ and $y$ are of equal degree $d$, $\kappa_{\alpha}(x,y)$ has at most two linear parts. More precisely, the idleness function is linear on the intervals $\interval[scaled]{0}{1/(d+1)}$ and $\interval[scaled]{1/(d+1)}{1}$. We will provide exact formulas for $\kappa_{\alpha}$ on the two linear parts.

\subsection{Decomposing the neighborhood}

In the previous chapter, we have already seen that the Ricci-curvature is a local property dependent solely on the core neighborhood of an edge.

In what follows, we will see that the Ricci curvature of an edge $x \sim y$ with $d_{x} = d_{y}$ can be further simplified. To this end, we decompose the 1-step neighborhood of the vertices $x$ and $y$ into disjoint sets. We denote by $\triangle(x,y)$ the set of common neighbors of both $x$ and $y$, namely, 
\begin{equation*}
    \triangle(x,y) = S_{1}(x) \cap S_{1}(y).
\end{equation*}
Thus, $\vert \triangle(x,y)\vert$ is the number of $3$-cycles in $G$
supported on the edge $x\sim y$.  \\
Furthermore, we define 
\[
    R_{x}(x,y) = S_{1}(x) \setminus (\triangle(x,y) \cup \{y\}) \quad \text{and} \quad R_{y}(x,y) = S_{1}(y) \setminus (\triangle(x,y) \cup \{x\}),
\]
giving rise to the following disjoint decompositions of $S_{1}(x)$ and $S_{1}(y)$. 
\[
    S_{1}(x) = \{y\} \cup \triangle(x,y) \cup R_{x}(x,y) \quad \text{and} \quad S_{1}(y) = \{x\} \cup \triangle(x,y) \cup R_{y}(x,y).
\]

Using these decompositions, we argue that the computation of the 1-Wasserstein distance reduces to an assignment problem between $R_{x}(x,y)$ and $R_{y}(x,y)$. Therefore, we introduce the concept of optimal assignments.

\begin{definition}[Optimal assignment]
    Let $G=(V,E)$ be a locally finite graph. Let $U, W \subset V$ with $\vert U \vert = \vert W \vert < \infty$. We call a bijection $\phi: U \to W$ an \textit{assignment} between $U$ and $W$. Denote by $\mathcal{A}_{UW}$ the set of all assignments between $U$ and $W$. We call $\phi \in \mathcal{A}_{UW}$ an \textit{optimal assignment} between $U$ and $W$ if
    \begin{equation*}
        \sum_{z \in U} d(z,\phi(z)) = \inf_{\psi \in \mathcal{A}_{UW}} \sum_{z \in U} d(z,\psi(z)).
    \end{equation*}
    Denote the set of all optimal assignments between $U$ and $W$ by $\mathcal{O}_{UW}$. 
\end{definition}

\begin{remark}
    Observe that for $U=W=\emptyset$, $\mathcal{A}_{UW}$ contains the empty function. Hence, $\mathcal{A}_{UW} \neq \emptyset$ and $\mathcal{O}_{UW} \neq \emptyset$ always holds true.
\end{remark}

\begin{definition}[Perfect matching]
    Let $G=(V,E)$ be a locally finite graph. A \textit{perfect matching} between two disjoint subsets $U,W \subset V$ is a bijective map $\phi: U \to W$, with $d(z, \phi(z)) = 1$ for all $z \in U$.
\end{definition}

Note that for an edge $x \sim y$ with $d_{x}=d_{y}$, we have $\vert R_{x}(x,y) \vert = \vert R_{y}(x,y) \vert$. For the sake of simplicity, we will denote by $\mathcal{A}_{xy}$ and $\mathcal{O}_{xy}$ the sets $\mathcal{A}_{R_{x}(x,y)R_{y}(x,y)}$ and $\mathcal{O}_{R_{x}(x,y)R_{y}(x,y)}$, respectively. 
\\
For an assignment $\phi \in \mathcal{A}_{xy}$, we associate the following sets:

\begin{itemize}
    \item $\square(\phi)$: The neighbors of $x$, forming a 4-cycle based at the edge $x \sim y$, with their image under $\phi$. That is, $\square(\phi) = \{ z \in R_{x}(x,y): d(z, \phi(z)) = 1\}.$ 
    \item$\pentago(\phi)$: The neighbors of $x$, forming a 5-cycle based at the edge $x \sim y$, with their image under $\phi$. That is,
    $\pentago(\phi) = \{ z \in R_{x}(x,y): d(z, \phi(z)) = 2\}.$ 
\end{itemize}

\subsection{Ollivier-Ricci curvature on the last linear part}

The objective of this section is to establish our first precise expression for the $\alpha$-Ollivier-Ricci curvature. More precisely, we aim to give an exact formula for $\kappa_{\alpha}(x,y)$, where $\alpha \in \interval[scaled]{\frac{1}{d+1}}{1}$ and $x, y$ are two adjacent
vertices of equal degree, denoted by $d$. 

To this end, we fix an arbitrary $\alpha$ in this interval. For the convenience of the reader, we state again the definition of the probability measure $\mu_{x}^{\alpha}$:
\begin{equation*}
    \mu_{x}^{\alpha}(y) =
    \begin{cases}
        \alpha, & \text{if $y = x$;}\\
        \frac{1-\alpha}{d}, & \text{if $y \sim x$;}\\
        0, & \text{otherwise.}
    \end{cases}
\end{equation*}

The following Lemma states some of the assumptions we can impose on an optimal transport plan transporting $\mu_{x}^{\alpha}$ to $\mu_{y}^{\alpha}$.

\begin{lemma}\thlabel{assumptions_transport_plan}
    Let $\alpha \in \interval[scaled]{\frac{1}{d+1}}{1}$. Then there exists an optimal transport plan $\pi$ transporting $\mu_{x}^{\alpha}$ to $\mu_{y}^{\alpha}$ with the following properties:
    \begin{enumerate}
        \item[$(i)$] $\forall i \in \triangle(x,y): \pi(i,i) = \mu_{x}^{\alpha}(i) = \mu_{y}^{\alpha}(i);$ 
        \item[$(ii)$] $\pi(x,x) = \pi(y,y) = \mu_{y}^{\alpha}(x);$
        \item[$(iii)$] $\pi(x,y) = \mu_{y}^{\alpha}(y) - \mu_{y}^{\alpha}(x).$
    \end{enumerate}
\end{lemma}

\begin{remark}
    Observe that $\mu_{y}^{\alpha}(y) - \mu_{y}^{\alpha}(x) \geq 0$ if and only if $\alpha \geq \frac{1}{d+1}$. Hence, the subsequent discussion pertains only to the last linear part of $\kappa_{\alpha}$. 
\end{remark}

\begin{proof}
    Properties $(i)$ and $(ii)$ are immediate consequence of \thref{dontmove}.

    For the proof of $(iii)$, let $\pi$ be an optimal transport plan transporting $\mu_{x}^{\alpha}$ to $\mu_{y}^{\alpha}$ while satisfying properties $(i)$ and $(ii)$. Since $\pi$ satisfies property $(ii)$, we must have
    \begin{align*}
        \pi(x,y) & = \mu_{y}^{\alpha}(y) - \sum_{j \in V\setminus\{x\}}\pi(j,y) \\
                 & \leq \mu_{y}^{\alpha}(y) - \pi(y,y) \\
                 & = \mu_{y}^{\alpha}(y) - \mu_{y}^{\alpha}(x).
    \end{align*}
    Assume that $\pi(x,y) < \mu_{y}^{\alpha}(y) - \mu_{y}^{\alpha}(x)$. Thus, the sets 
    \[
        I = \{z \in S_{1}(x)\setminus\{y\}: \pi(z,y) > 0\} \quad \text{and} \quad O = \{z \in S_{1}(y)\setminus\{x\}: \pi(x,z) > 0\}
    \]
    are not empty. Due to property $(i)$, we must have $I \cap O = \emptyset$. Since $\pi$ satisfies property $(ii)$, we obtain:
    \begin{align*}
        \sum_{z \in I} \pi(z,y) & = \mu_{y}^{\alpha}(y) - \pi(y,y) - \pi(x,y) \\
                                & = \mu_{x}^{\alpha}(x) - \pi(x,x) - \pi(x,y) \\
                                & = \sum_{z \in O} \pi(x,z).
    \end{align*}
    In other words, the mass transported from $I$ to $y$ equals the mass transported from $x$ to $O$ under $\pi$. 
    We denote this mass by $m$, namely $m = \sum_{z \in I} \pi(z,y)$.\\
    We now obtain a new transport plan $\pi'$ from modifying $\pi$ as follows:
    \begin{itemize}
        \item Set $\pi'(x,y) = \mu_{y}^{\alpha}(y) - \mu_{y}^{\alpha}(x)$, $\pi'(x,z) = 0 \; \forall z\in O$ and $\pi'(z,y) = 0 \; \forall z\in I$
        \item Move the mass $m$ directly from $I$ to $O$.
    \end{itemize}
    See Figure \ref{optimal_transport_plan} for an illustration of this modification.

    \begin{figure}
        \center
    
        \begin{tikzpicture}[x=0.75pt,y=0.75pt,yscale=-1,xscale=1]
            
            \draw   (234.48,72.76) .. controls (234.48,69.03) and (237.51,66) .. (241.24,66) .. controls (244.97,66) and (248,69.03) .. (248,72.76) .. controls (248,76.49) and (244.97,79.52) .. (241.24,79.52) .. controls (237.51,79.52) and (234.48,76.49) .. (234.48,72.76) -- cycle ;
            \draw   (399.46,72.76) .. controls (399.46,69.03) and (396.43,66) .. (392.7,66) .. controls (388.96,66) and (385.94,69.03) .. (385.94,72.76) .. controls (385.94,76.49) and (388.96,79.52) .. (392.7,79.52) .. controls (396.43,79.52) and (399.46,76.49) .. (399.46,72.76) -- cycle ;
            \draw    (248,72.76) -- (383.94,72.76) ;
            \draw [shift={(385.94,72.76)}, rotate = 180] [color={rgb, 255:red, 0; green, 0; blue, 0 }  ][line width=0.75]    (10.93,-3.29) .. controls (6.95,-1.4) and (3.31,-0.3) .. (0,0) .. controls (3.31,0.3) and (6.95,1.4) .. (10.93,3.29)   ;
            \draw  [fill={rgb, 255:red, 155; green, 155; blue, 155 }  ,fill opacity=1 ] (167.55,133.25) .. controls (152.64,133.07) and (140.86,106.4) .. (141.25,73.68) .. controls (141.64,40.96) and (154.04,14.58) .. (168.95,14.75) .. controls (183.86,14.93) and (195.64,41.6) .. (195.25,74.32) .. controls (194.86,107.04) and (182.46,133.42) .. (167.55,133.25) -- cycle ;
            \draw  [fill={rgb, 255:red, 155; green, 155; blue, 155 }  ,fill opacity=1 ] (470.08,132.24) .. controls (455.17,132.45) and (442.71,106.1) .. (442.25,73.38) .. controls (441.8,40.66) and (453.51,13.96) .. (468.42,13.76) .. controls (483.33,13.55) and (495.79,39.9) .. (496.25,72.62) .. controls (496.7,105.34) and (484.99,132.04) .. (470.08,132.24) -- cycle ;
            \draw    (195.25,74.32) .. controls (293.01,16.29) and (292.51,15.01) .. (391.2,65.24) ;
            \draw [shift={(392.7,66)}, rotate = 206.98] [color={rgb, 255:red, 0; green, 0; blue, 0 }  ][line width=0.75]    (10.93,-3.29) .. controls (6.95,-1.4) and (3.31,-0.3) .. (0,0) .. controls (3.31,0.3) and (6.95,1.4) .. (10.93,3.29)   ;
            \draw    (241.24,79.52) .. controls (355.92,135.72) and (356.5,136.99) .. (440.97,74.33) ;
            \draw [shift={(442.25,73.38)}, rotate = 143.43] [color={rgb, 255:red, 0; green, 0; blue, 0 }  ][line width=0.75]    (10.93,-3.29) .. controls (6.95,-1.4) and (3.31,-0.3) .. (0,0) .. controls (3.31,0.3) and (6.95,1.4) .. (10.93,3.29)   ;
            \draw   (235.48,220.76) .. controls (235.48,217.03) and (238.51,214) .. (242.24,214) .. controls (245.97,214) and (249,217.03) .. (249,220.76) .. controls (249,224.49) and (245.97,227.52) .. (242.24,227.52) .. controls (238.51,227.52) and (235.48,224.49) .. (235.48,220.76) -- cycle ;
            \draw   (400.46,220.76) .. controls (400.46,217.03) and (397.43,214) .. (393.7,214) .. controls (389.96,214) and (386.94,217.03) .. (386.94,220.76) .. controls (386.94,224.49) and (389.96,227.52) .. (393.7,227.52) .. controls (397.43,227.52) and (400.46,224.49) .. (400.46,220.76) -- cycle ;
            \draw    (249,220.76) -- (384.94,220.76) ;
            \draw [shift={(386.94,220.76)}, rotate = 180] [color={rgb, 255:red, 0; green, 0; blue, 0 }  ][line width=0.75]    (10.93,-3.29) .. controls (6.95,-1.4) and (3.31,-0.3) .. (0,0) .. controls (3.31,0.3) and (6.95,1.4) .. (10.93,3.29)   ;
            \draw  [fill={rgb, 255:red, 155; green, 155; blue, 155 }  ,fill opacity=1 ] (168.55,281.25) .. controls (153.64,281.07) and (141.86,254.4) .. (142.25,221.68) .. controls (142.64,188.96) and (155.04,162.58) .. (169.95,162.75) .. controls (184.86,162.93) and (196.64,189.6) .. (196.25,222.32) .. controls (195.86,255.04) and (183.46,281.42) .. (168.55,281.25) -- cycle ;
            \draw  [fill={rgb, 255:red, 155; green, 155; blue, 155 }  ,fill opacity=1 ] (471.08,280.24) .. controls (456.17,280.45) and (443.71,254.1) .. (443.25,221.38) .. controls (442.8,188.66) and (454.51,161.96) .. (469.42,161.76) .. controls (484.33,161.55) and (496.79,187.9) .. (497.25,220.62) .. controls (497.7,253.34) and (485.99,280.04) .. (471.08,280.24) -- cycle ;
            \draw    (196.25,222.32) .. controls (310.5,152) and (313.5,152) .. (443.25,221.38) ;
            \draw [shift={(443.25,221.38)}, rotate = 208.13] [color={rgb, 255:red, 0; green, 0; blue, 0 }  ][line width=0.75]    (10.93,-3.29) .. controls (6.95,-1.4) and (3.31,-0.3) .. (0,0) .. controls (3.31,0.3) and (6.95,1.4) .. (10.93,3.29)   ;
            
            \draw (297,52.4) node [anchor=north west][inner sep=0.75pt]    {$\pi ( x,y)$};
            \draw (236.48,83.16) node [anchor=north west][inner sep=0.75pt]    {$x$};
            \draw (387.94,83.16) node [anchor=north west][inner sep=0.75pt]    {$y$};
            \draw (125,25.4) node [anchor=north west][inner sep=0.75pt]    {$I$};
            \draw (497,23.4) node [anchor=north west][inner sep=0.75pt]    {$O$};
            \draw (288,8.4) node [anchor=north west][inner sep=0.75pt]    {$m$};
            \draw (346,124.4) node [anchor=north west][inner sep=0.75pt]    {$m$};
            \draw (237.48,231.16) node [anchor=north west][inner sep=0.75pt]    {$x$};
            \draw (388.94,231.16) node [anchor=north west][inner sep=0.75pt]    {$y$};
            \draw (126,173.4) node [anchor=north west][inner sep=0.75pt]    {$I$};
            \draw (498,171.4) node [anchor=north west][inner sep=0.75pt]    {$O$};
            \draw (262,196.4) node [anchor=north west][inner sep=0.75pt]    {$\mu _{y}^{\alpha }( y) \ -\mu _{y}^{\alpha }( x) \ $};
            \draw (305,150.4) node [anchor=north west][inner sep=0.75pt]    {$m$};
            \draw (73,57.4) node [anchor=north west][inner sep=0.75pt]    {$\pi :$};
            \draw (72,204.4) node [anchor=north west][inner sep=0.75pt]    {$\pi ':$};

        \end{tikzpicture}
        \caption{Illustration of the optimal transport plan $\pi$ and its modification $\pi'$ for the proof of \thref{assumptions_transport_plan}.}
        \label{optimal_transport_plan}      
            
    \end{figure}

    Due to property $(i)$, we must have $O \cap \triangle(x,y) = I \cap \triangle(x,y) = \emptyset$. 
    Therefore, $d(i,y) = 2$ for any $i \in I$ and $d(x,j) = 2$ for any $j \in O$. 
    Thus, $\pi$ moves mass 
    \begin{itemize}
        \item $m$ at distance two from $I$ to $y$,
        \item $m$ at distance two from $x$ to $O$,
        \item $\pi(x,y) = \mu_{y}^{\alpha}(y) - \mu^{\alpha}_{y}(x) - m$ at distance one from $x$ to $y$,
    \end{itemize}
    resulting in a cost of:
    \[
        3 \cdot m + \mu_{y}^{\alpha}(y) - \mu^{\alpha}_{y}(x).
    \]
    On the other hand, observe that $d(i,j) \leq 3$ for any $i\in I$ and $j \in O$. Thus, $\pi'$ moves the mass
    \begin{itemize}
        \item $\mu_{y}^{\alpha}(y) - \mu^{\alpha}_{y}(x)$ at distance one from $x$ to $y$,
        \item $m$ from $I$ to $O$ at distance less than or equal to three,
    \end{itemize}
    resulting in a cost less than or equal to:
    \[
        3 \cdot m + \mu_{y}^{\alpha}(y) - \mu^{\alpha}_{y}(x).
    \]
    This is exactly the cost of moving this part of the mass under the optimal transport plan $\pi$. As $\pi'$ coincides with $\pi$ everywhere else, $\pi'$ is an optimal transport plan as well. Finally, observe that by construction,
    \[
        \pi'(j,j) = \pi(j,j) \; \forall j \in \triangle(x,y), \quad \pi'(x,x) = \pi(x,x) \quad \text{and} \quad \pi'(y,y) = \pi(y,y).
    \] 
    Hence, we have constructed an optimal transport plan satisfying properties $(i) -  (iii)$.
\end{proof}

We are now in a position to derive a precise expression for the $\alpha$-Ollivier-Ricci curvature of an edge on the last linear part, in terms of graph parameters and an optimal assignment between $R_{x}(x,y)$ and $R_{y}(x,y)$.

\begin{theorem}\thlabel{precise_formula_1}
    Let $G=(V,E)$ be a locally finite graph. Let $x,y \in V$ be of equal degree $d$ with $x\sim y$. Furthermore, let $\alpha \in \interval[scaled]{\frac{1}{d+1}}{1}$.
    Then the $\alpha$-Ollivier-Ricci curvature
    \begin{equation*}
        \kappa_{\alpha}(x,y) = \frac{1 - \alpha}{d}\Biggl(d+1 - \inf_{\phi \in \mathcal{A}_{xy}} \mathlarger{\sum}_{z \in R_{x}(x,y)}d(z, \phi(z))\Biggr).
    \end{equation*}
\end{theorem}

\begin{remark}
    For a similar curvature formula in the case of combinatorial graphs, refer to \cite[Theorem 2.6]{munch2019ollivier}.
\end{remark}
\begin{proof}
    To demonstrate this theorem, we proceed with a proof by cases.

    \emph{Case 1: $R_{x}(x,y) = \emptyset$.} In this case, $\mathcal{A}_{xy}$ only contains the empty function. Thus, 
    \begin{equation*}
        \frac{1 - \alpha}{d}\Biggl(d+1 - \inf_{\phi \in \mathcal{A}_{xy}} \mathlarger{\sum}_{z \in R_{x}(x,y)}d(z, \phi(z))\Biggr) = \frac{1 - \alpha}{d}(d+1).
    \end{equation*}
    On the other hand, the transport plan $\pi$ transporting $\mu_{x}^{\alpha}$ to $\mu_{y}^{\alpha}$, defined by
    \begin{equation*}
        \pi(i,j) = \begin{cases}
            \frac{1-\alpha}{d}, & \text{if $i = j \in \{x,y\} \cup \triangle(x,y)$;}\\
            \alpha - \frac{1-\alpha}{d}, & \text{if $i = x$ and $j = y$;}\\
            0, & \text{otherwise.}
        \end{cases}
    \end{equation*}
    is optimal, according to \thref{assumptions_transport_plan}.
    Therefore, 
    \begin{equation*}
        W_{1}(\mu_{x}^{\alpha}, \mu_{y}^{\alpha}) = \sum_{i \in B_{1}(x)} \sum_{j \in B_{1}(y)} d(i,j)\pi(i,j) = \alpha - \frac{1-\alpha}{d},
    \end{equation*}
    and hence, 
    \begin{equation*}
        \kappa_{\alpha}(x,y) = \frac{1 - \alpha}{d}(d+1).
    \end{equation*}

    \emph{Case 2: $R_{x}(x,y) \neq \emptyset$.} Let $\phi \in \mathcal{A}_{xy}$ be an optimal assignment between $R_{x}(x,y)$ and $R_{y}(x,y)$. Define the transport plan $\pi$ transporting $\mu_{x}^{\alpha}$ to $\mu_{y}^{\alpha}$ by
    \begin{equation*}
        \pi(i,j) = \begin{cases}
            \frac{1-\alpha}{d}, & \text{if $i = j \in \{x,y\} \cup \triangle(x,y)$;}\\
            \frac{1-\alpha}{d}, & \text{if $i \in R_{x}(x,y)$ and $j = \phi(i)$;}\\
            \alpha - \frac{1-\alpha}{d}, & \text{if $i = x$ and $j = y$;}\\
            0, & \text{otherwise.}
        \end{cases}
    \end{equation*}
    We claim that this transport plan is optimal. Assume this is not true. Then let $\pi'$ be an optimal transport plan transporting $\mu_{x}^{\alpha}$ to $\mu_{y}^{\alpha}$, satisfying properties $(i)-(iii)$ stated in \thref{assumptions_transport_plan}. Then,
    \begin{equation}\label{eq:3}
        \sum_{i \in B_{1}(x)} \sum_{j \in B_{1}(y)} \pi'(i,j) d(i,j) = \alpha - \frac{1-\alpha}{d} + \sum_{i \in R_{x}(x,y)} \sum_{j \in R_{x}(x,y)} \pi'(i,j) d(i,j). 
    \end{equation}
    On the other hand, we obtain by construction of $\pi$
    \begin{equation}\label{eq:4}
        \sum_{i \in B_{1}(x)} \sum_{j \in B_{1}(y)} \pi(i,j) d(i,j) = \alpha - \frac{1-\alpha}{d} + \sum_{i \in R_{x}(x,y)} \sum_{j \in R_{x}(x,y)} \pi(i,j) d(i,j). 
    \end{equation}
    Due to the non-optimality of $\pi$, we must have: 
    \begin{equation}\label{eq: 1010}
        \sum_{i \in B_{1}(x)} \sum_{j \in B_{1}(y)} \pi'(i,j) d(i,j) < \sum_{i \in B_{1}(x)} \sum_{j \in B_{1}(y)} \pi(i,j) d(i,j).
    \end{equation}
    Combining equations \ref{eq:3}, \ref{eq:4} and \ref{eq: 1010} leads to
    \begin{equation} \label{eq:5}
        \sum_{i \in R_{x}(x,y)} \sum_{j \in R_{x}(x,y)} \pi'(i,j) d(i,j) < \sum_{i \in R_{x}(x,y)} \sum_{j \in R_{x}(x,y)} \pi(i,j) d(i,j). 
    \end{equation}
    Now, observe that $\pi{\big|}_{R_{x}(x,y) \times R_{y}(x,y)}$ and $\pi'{\big|}_{R_{x}(x,y) \times R_{y}(x,y)}$ are transport plans transporting $\mu_{x}^{\alpha}{\big|}_{R_{x}(x,y)}$ to $\mu_{y}^{\alpha}{\big|}_{R_{y}(x,y)}$. Furthermore, 
    $\mu_{x}^{\alpha}{\big|}_{R_{x}(x,y)}$ and $\mu_{y}^{\alpha}{\big|}_{R_{y}(x,y)}$ are measures under which all points in their support carry equal mass $\frac{1-\alpha}{d}$. \\
    Thus, we can apply \thref{Wasserstein_distance_easy} (after multiplication with a scaling factor) and conclude that
    $\pi{\big|}_{R_{x}(x,y) \times R_{y}(x,y)}$ is an optimal transport plan. This contradicts equation \ref{eq:5}. \\
    Therefore, $\pi$ is an optimal transport plan. The rest of the proof is just a straightforward calculation. Namely,
    \begin{align*}
        W_{1}(\mu_{x}^{\alpha}, \mu_{y}^{\alpha}) & = \sum_{i \in B_{1}(x)} \sum_{j \in B_{1}(y)} d(i,j)\pi(i,j) \\
        & = \alpha - \frac{1-\alpha}{d} + \frac{1-\alpha}{d} \sum_{z \in R_{x}(x,y)}d(z, \phi(z)),
    \end{align*}
    which leads to 
    \begin{equation*}
        \kappa_{\alpha}(x,y) = \frac{1 - \alpha}{d}\Biggl(d + 1 - \mathlarger{\sum}_{z \in R_{x}(x,y)}d(z, \phi(z))\Biggr).
    \end{equation*}
    The fact that $\phi$ is an optimal assignment concludes the proof.
\end{proof}

We can modify the formula for $\kappa_{\alpha}(x,y)$, given in the previous theorem, using that 
\begin{equation*}
    \vert R_{x}(x,y) \vert = d - 1 -\vert \triangle(x,y) \vert.
\end{equation*} 
Namely,
\begin{align*}
    \kappa_{\alpha}(x,y) &= \frac{1 - \alpha}{d}\Biggl(d + 1 - \inf_{\phi \in \mathcal{A}_{xy}} \mathlarger{\sum}_{z \in R_{x}(x,y)}d(z, \phi(z))\Biggr) \\
    &=\frac{1 - \alpha}{d}\Biggl(d + 1 - \inf_{\phi \in \mathcal{A}_{xy}} \mathlarger{\sum}_{k=1}^{3} k \cdot \Big\vert \Big\{z\in R_{x}(x,y): d(z, \phi(z)) = k\Big\} \Big\vert\Biggr) \\
    &=\frac{1-\alpha}{d}\Biggl(-2d + 4+ 3\vert \triangle(x,y)\vert+ \sup_{\phi \in \mathcal{A}_{xy}}\Big\{2\vert\square(\phi)\vert+ \vert\pentago(\phi)\vert\Big\}\Biggr).
\end{align*}
The advantage of this expression is, that it provides a nice intuition of the $\alpha$-Ollivier-Ricci curvature.
\begin{itemize}
    \item Given an edge $x\sim y$ that satisfies $dist(S_{1}(x)\setminus\{y\}, S_{1}(y)\setminus\{x\})=3$, we have initial  $\alpha$-Ollivier-Ricci curvature
    \begin{equation*}
        \kappa_{\alpha}(x,y) = \frac{1-\alpha}{d}(-2d + 4).
    \end{equation*}

    \item Now, assume $x\sim y$ is contained in some 3-cycles. In this case, the mass at the common neighbors remains in place instead of being moved by a distance of three. Therefore increasing $\kappa_{\alpha}$ by $3\cdot\vert \triangle(x,y) \vert \cdot \frac{1-\alpha}{d}$.
    
    \item Next, adding an edge between a vertex from $R_{x}(x,y)$ and a vertex from $R_{y}(x,y)$ allows to move the mass $(1-\alpha)/d$ at distance one instead of distance three. Therefore increasing $\kappa_{\alpha}$ by $2\cdot\frac{1-\alpha}{d}$.
    
    \item Similarly, connecting a vertex in $R_{x}(x,y)$ with a vertex in $R_{y}(x,y)$ by a path of distance two allows to move the mass $(1-\alpha)/d$ at distance two instead of distance three. Therefore increasing $\kappa_{\alpha}$ by $\frac{1-\alpha}{d}$.
\end{itemize}

\subsection{Lin-Lu-Yau curvature} 

In this section, we derive exact formulas for the Lin-Lu-Yau curvature of an edge $x \sim y$ with $d_{x}=d_{y}$. We then proceed to state some immediate consequences of our expression. 

According to \thref{relation_curvature}, the identity
\begin{equation*}
    \kappa(x,y) = \frac{\kappa_{\alpha}(x,y)}{1-\alpha}
\end{equation*}
holds true for all idleness parameters $\alpha \in \interval[open right,scaled]{\frac{1}{max\{d_{x}, d_{y}\}+1}}{1}$. Therefore, the formula for the Lin-Lu-Yau curvature is an immediate consequence of \thref{precise_formula_1}.

\begin{theorem}\thlabel{Lin_Lu_Yau_curvature}
    Let $G=(V,E)$ be a locally finite graph. Let $x,y \in V$ be of equal degree $d$ with $x\sim y$.
    Then the Lin-Lu-Yau curvature
    \begin{equation}\label{formulae: Lin-Lu_Yau}
        \kappa(x,y) = \frac{1}{d}\Biggl(d+1 - \inf_{\phi \in \mathcal{A}_{xy}} \mathlarger{\sum}_{z \in R_{x}(x,y)}d(z, \phi(z))\Biggr).
    \end{equation}
\end{theorem}

By employing the adjusted expression for the $\alpha$-Ollivier-Ricci curvature, we deduce a second precise expression for the Lin-Lu-Yau curvature:
\begin{equation}\label{formulae: Lin-Lu_Yau_modified}
    \kappa(x,y) = \frac{1}{d}\Biggl(-2d + 4+ 3\vert \triangle(x,y)\vert+ \sup_{\phi \in \mathcal{A}_{xy}}\Big\{2\vert\square(\phi)\vert+ \vert\pentago(\phi)\vert\Big\}\Biggr).
\end{equation}

We now state some immediate consequences of our formulas.

\begin{corollary}\thlabel{Lin_Lu_Yau_in_Z}
    Let $G=(V,E)$ be a locally finite graph. Let $x,y \in V$ be of equal degree $d$ with $x\sim y$. Then $\kappa(x,y) \in \mathbb{Z}/d$. Furthermore,
    \begin{equation*}
        -2+\frac{4}{d} + 3\frac{\vert \triangle(x,y) \vert}{d} \leq \kappa(x,y) \leq \frac{2+\vert \triangle(x,y)\vert}{d}.
    \end{equation*}
\end{corollary}
\begin{proof}
    The fact that $\kappa(x,y) \in \mathbb{Z}/d$ follows directly from formula \ref{formulae: Lin-Lu_Yau}. 

    Using that $d(z, \phi(z)) \geq 1$ for all $z\in R_{x}(x,y)$, we obtain
    \begin{align*}
        \kappa(x,y) &= \frac{1}{d}\Biggl(d+1 - \inf_{\phi \in \mathcal{A}_{xy}} \mathlarger{\sum}_{z \in R_{x}(x,y)}d(z, \phi(z))\Biggr) \\
       &\leq \frac{1}{d}\Biggl(d+1 - \vert R_{x}(x,y)\vert\Biggr) \\
       &= \frac{1}{d}\Biggl(d+1 - (d - 1 - \vert \triangle(x,y) \vert) \Biggr) \\
       &= \frac{2+\vert \triangle(x,y)\vert}{d}.
    \end{align*}
    Similarly, using that $d(z, \phi(z)) \leq 3$, we obtain
    \begin{align*}
        \kappa(x,y) &= \frac{1}{d}\Biggl(d+1 - \inf_{\phi \in \mathcal{A}_{xy}} \mathlarger{\sum}_{z \in R_{x}(x,y)}d(z, \phi(z))\Biggr) \\
        &\geq \frac{1}{d}\Biggl(d+1 - 3\vert R_{x}(x,y)\vert\Biggr) \\
        &= \frac{1}{d}\Biggl(-2d + 4 + 3 \vert \triangle(x,y) \vert\Biggr).
    \end{align*}
\end{proof}

\begin{remark}
    Note that the upper bound is attained if and only if there exists a perfect matching between $R_{x}(x,y)$ and $R_{y}(x,y)$. On the other hand, the lower bound is an equality if and only if $dist(R_{x}(x,y), R_{y}(x,y))=3$ or $R_{x}(x,y) = \emptyset$.
\end{remark}

Utilizing the lower bound derived in the previous corollary, we attain the following condition for the Lin-Lu-Yau curvature to be positive.

\begin{corollary}
    Let $G=(V,E)$ be a locally finite graph. Let $x,y \in V$ be of equal degree $d$ with $x\sim y$. If
    \begin{equation*}
        \vert \triangle(x,y) \vert > \frac{2d-4}{3},
    \end{equation*}
    then $\kappa(x,y) > 0$.
\end{corollary}

\subsection{Strongly regular graphs}

Bonini et al. \cite{bonini2020condensed} derive an exact formula for the Lin-Lu-Yau curvature on strongly regular graphs. In this section, we show that our formula for arbitrary regular graphs reduces to the formula provided by Bonini et al. in the case of strongly regular graphs.

First, we recall the definition of a strongly regular graph.

\begin{definition}[Strongly regular graph]
    A $d$-regular graph $G=(V,E)$ with $n$ vertices is said to be \textit{strongly regular} with parameters $(n,d,\lambda, \mu)$, if every two adjacent vertices have $\lambda\geq 0$ common neighbors and every two nonadjacent vertices have $\mu \geq 1$ common neighbors.
\end{definition}

Note that the diameter of a strongly regular graph is less than or equal to two. In particular, every node $i\in R_{x}(x,y)$ is at a distance less than or equal to two from a node $j\in R_{x}(x,y)$. Using this property, we obtain the following precise formula for the Lin-Lu-Yau curvature on strongly regular graphs.

\begin{corollary}
    Let $G=(V,E)$ be a strongly regular graph with parameters $(n,d,\lambda, \mu)$ and let $x,y\in V$ with $x\sim y$. Then, the Lin-Lu-Yau curvature
    \begin{equation*}
        \kappa(x,y) = \frac{\lambda+2}{d} - \frac{\vert R_{x}(x,y) \vert - \vert \square(\phi) \vert}{d},
    \end{equation*}
    for an optimal assignment $\phi \in \mathcal{O}_{xy}$.
\end{corollary}

This formula coincides with the formula provided by Bonini et al. in \cite[Theorem 4.6]{bonini2020condensed}.

\begin{proof}
    Let $\phi$ be an optimal assignment between $R_{x}(x,y)$ and $R_{y}(x,y)$. As $d(i,j) \leq 2$ for any $i\in R_{x}(x,y)$ and $j \in R_{y}(x,y)$, we have 
    \begin{equation*}
         \vert \pentago(\phi) \vert = \vert R_{x}(x,y) \vert - \vert \square(\phi) \vert.
    \end{equation*}
    Using that $\vert R_{x}(x,y) \vert = d-1-\lambda$ and the formula \ref{formulae: Lin-Lu_Yau_modified}, we obtain
    \begin{equation*}
        \kappa(x,y) = \frac{\lambda+2}{d} - \frac{\vert R_{x}(x,y) \vert - \vert \square(\phi)\vert}{d}.
    \end{equation*}
\end{proof}

\subsection{Ollivier-Ricci curvature for vanishing idleness}

This section aims to derive an exact expression for the $\alpha$-Ollivier-Ricci curvature of an edge $x \sim y$ with $d_{x} = d_{y}$, in the case of vanishing idleness, i.e., $\alpha=0$. 

Recall that for $\alpha=0$, $\mu_{x}^{\alpha}$ simplifies to
\begin{equation*}
    \mu_{x}^{\alpha}(y) = \begin{cases}
        \frac{1}{d_{x}}, & \text{if $y \in S_{1}(x)$;}\\
        0, & \text{otherwise.}
    \end{cases}
\end{equation*}

As per \thref{Wasserstein_distance_easy}, there exists an optimal transport plan transporting $\mu_{x}^{\alpha}$ to $\mu_{y}^{\alpha}$ corresponding to an optimal assignment between $S_{1}(x)$ and $S_{1}(y)$.
The following Lemma states some of the assumptions we can impose on such an optimal assignment.

\begin{lemma}\thlabel{assumptions_optimal_matching}
    Let $G=(V,E)$ be a locally finite graph. Let $x,y \in V$ be of equal degree d with $x \sim y$.
    Then there exists an optimal assignment $\phi$ between $S_{1}(x)$ and $S_{1}(y)$, such that $\phi(i) = i$ for all $i \in \triangle(x,y)$.

    Furthermore, if $\vert \triangle(x,y) \vert < d-1$, there exists an optimal assignment $\phi$, satisfying the aforementioned property and $\phi(y) \neq x$.
\end{lemma}

\begin{proof}
    Let $\phi$ be an optimal assignment between $S_{1}(x)$ and $S_{1}(y)$. Assume there exists an $i \in \triangle(x,y)$ with $\phi(i) \neq i$. As an optimal assignment, $\phi$ is a bijection from $S_{1}(x)$ to $S_{1}(y)$. Hence, there exists a unique $j \in S_{1}(y)$, such that $\phi(i) = j$ and a unique $k \in S_{1}(x)$, such that $\phi(k) = i$. 
    Define a new assignment $\phi'$ between $S_{1}(x)$ and $S_{1}(y)$ by
    \begin{equation*}
        \phi'(z) = \begin{cases}
            i, & \text{if $z = i$;}\\
            j, & \text{if $z = k$;}\\
            \phi(z), & \text{otherwise.}
        \end{cases}
    \end{equation*}
    By the triangle inequality, we have $d(k,j) \leq d(k,i) + d(i,j)$. Hence, the new assignment $\phi'$ is still optimal.
    Note that $\phi'(z) = z$ for all $z \in S_{1}(x)$ with $\phi(z) = z$. Repeating this modification at all other vertices that violate the lemma's condition results in an optimal assignment $\psi$ satisfying
    $\psi(i) = i$ for all $i \in \triangle(x,y)$.
    
    For the second part of the proof, assume that $\vert \triangle(x,y) \vert < d-1$ and assume that $\psi(y) = x$. Since $\vert \triangle(x,y) \vert < d-1$, there exists $i \in S_{1}(x) \setminus(\triangle(x,y) \cup \{y\})$. 
    Define a new assignment $\psi'$ between $S_{1}(x)$ and $S_{1}(y)$ by
    \begin{equation*}
        \psi'(z) = \begin{cases}
            \psi(i), & \text{if $z = y$;}\\
            x, & \text{if $z = i$;}\\
            \psi(z), & \text{otherwise.}
        \end{cases}
    \end{equation*}
    Since $i \not\in \triangle(x,y)$, we have $d(i, \psi(i)) \geq 1$, leading to
    \begin{equation*}
        d(y, \psi'(y)) + d(i, \psi'(i)) = 2 \leq d(y, \psi(y)) + d(i, \psi(i)). 
    \end{equation*}
    Hence, the new assignment $\psi'$ is still optimal. Finally, note that $\psi'(i) = \psi(i) = i$ for all $i \in \triangle(x,y)$. 
    This concludes the proof.
\end{proof}

We are now in a position to derive a precise expression for $\kappa_{\alpha}(x,y)$ for $\alpha = 0$. The subsequent theorem, which is an immediate consequence of the first part of the previous lemma, states this expression.

\begin{theorem}\thlabel{kappa_null}
    Let $G=(V,E)$ be a locally finite graph. Let $x,y \in V$ be of equal degree $d$ with $x\sim y$. Then 
    \begin{equation*}
        \kappa_{0}(x,y) = \frac{1}{d}\Biggl(d - \inf_{\phi} \mathlarger{\sum}_{z \in S_{1}(x)\setminus\triangle(x,y)}d(z, \phi(z))\Biggr),
    \end{equation*}
    where the infimum is taken over all assignments $\phi$ between $S_{1}(x)\setminus\triangle(x,y)$ and $S_{1}(y)\setminus\triangle(x,y)$.
\end{theorem}

\begin{proof}
    Let $\phi$ be an optimal assignment between $S_{1}(x)$ and $S_{1}(y)$. We may assume that $\phi(i) = i$ for all $i\in \triangle(x,y)$, according to \thref{assumptions_optimal_matching}. As per \thref{Wasserstein_distance_easy}, the transport plan $\pi$ transporting $\mu_{x}^{0}$ to $\mu_{y}^{0}$ defined by
    \begin{equation*}
        \pi(i,j) = \begin{cases}
            \frac{1}{d}, & \text{if $i \in S_{1}(x)$ and $j = \phi(i)$;}\\
            0, & \text{otherwise.}
        \end{cases}
    \end{equation*}
    is optimal. Therefore,
    \begin{equation*}
        W_{1}(\mu_{x}^{0}, \mu_{y}^{0}) = \frac{1}{d}\sum_{z \in S_{1}(x)\setminus\triangle(x,y)} d(z, \phi(z)),
    \end{equation*}
    and hence
    \begin{equation*}
        \kappa_{0}(x,y) = \frac{1}{d} \Biggl(d- \sum_{z \in S_{1}(x) \setminus \triangle(x,y)}d(z, \phi(z))\Biggr).
    \end{equation*}
    Note that $\phi{\big|}_{S_{1}(x)\setminus\triangle(x,y)}$ is an optimal assignment between $S_{1}(x)\setminus\triangle(x,y)$ and $S_{1}(y)\setminus\triangle(x,y)$. 
    This concludes the proof.
\end{proof}

An immediate consequence of \thref{kappa_null} is the following Corollary.

\begin{corollary}\thlabel{kappa_null_in_Z}
    Let $G=(V,E)$ be a locally finite graph. Let $x,y\in V$ be of equal degree $d$ with $x \sim y$. Then $\kappa_{0}(x,y) \in \mathbb{Z}/d$. Furthermore, if $\vert \triangle(x,y) \vert < d-1$, then
    \begin{equation}\label{ineq: kappa_null}
        - 2 + \frac{4}{d} + 3 \frac{\vert \triangle(x,y) \vert}{d} \leq \kappa_{0}(x,y) \leq \frac{\vert\triangle(x,y)\vert}{d}.
    \end{equation}
    If $\vert \triangle(x,y) \vert=d-1$, then $\kappa_{0}(x,y) = \frac{\vert \triangle(x,y) \vert}{d}$.
\end{corollary}

\begin{remark}
    Observe that the upper bound in \ref{ineq: kappa_null} is attained if and only if there exists a perfect matching between $S_{1}(x)\setminus\triangle(x,y)$ and $S_{1}(y)\setminus\triangle(x,y)$. 
    
\end{remark}

Note that so far, we have only used the first assumption of \thref{assumptions_optimal_matching}. We now use the second assumption to study the relationship between $\kappa_{0}(x,y)$ and $\kappa(x,y)$.

\begin{theorem}\thlabel{kappa_vergleich}
    Let $G=(V,E)$ be a locally finite graph. Let $x,y \in V$ be of equal degree $d$ with $x\sim y$. If $\vert\triangle(x,y)\vert < d-1$, then 
    \begin{equation*}
        \kappa_{0}(x,y) = \kappa(x,y) - \frac{1}{d}\Biggl(3 - \sup_{\phi \in \mathcal{O}_{xy}} \sup_{z \in R_{x}(x,y)} d(z, \phi(z))\Biggr)
    \end{equation*}
    where $\mathcal{O}_{xy}$ denotes the set of all optimal assignments between $R_{x}(x,y)$ and $R_{y}(x,y)$.
    If $\vert \triangle(x,y) \vert=d-1$, then 
    \begin{equation*}
        \kappa_{0}(x,y) = \kappa(x,y)-\frac{2}{d}.
    \end{equation*}
\end{theorem}

\begin{proof}
    \emph{Case 1: $\vert \triangle(x,y) \vert = d-1$.} Straightforward calculations show that
    $\kappa(x,y) = \frac{d+1}{d}$, and $\kappa_{0}(x,y) = \frac{d-1}{d}$. Hence,
    \begin{equation*}
        \kappa_{0}(x,y) = \kappa(x,y) - \frac{2}{d}.
    \end{equation*}

    \emph{Case 2: $\vert \triangle(x,y) \vert < d-1$.} First, observe that $\vert R_{x}(x,y) \vert < \infty$ and hence $\vert\mathcal{O}_{xy}\vert < \infty$, as $G$ is locally finite.
    Therefore, there exists $\phi \in \mathcal{O}_{xy}$ and $j \in R_{x}(x,y)$ such that
    \begin{equation*}
        d(j, \phi(j)) = \sup_{\phi \in \mathcal{O}_{xy}} \sup_{z \in R_{x}(x,y)} d(z, \phi(z)).
    \end{equation*}
    Define an assignment $\psi$ between $S_{1}(x)\setminus \triangle(x,y)$ and $S_{1}(y)\setminus \triangle(x,y)$ by
    \begin{equation*}
        \psi(z) = \begin{cases}
            x, & \text{if $z = j$;}\\
            \phi(j), & \text{if $z = y$;} \\
            \phi(z), & \text{otherwise.}
        \end{cases}
    \end{equation*}
    We claim that $\psi$ is an optimal assignment between $S_{1}(x)\setminus \triangle(x,y)$ and $S_{1}(y)\setminus \triangle(x,y)$. Assume this is not the case. Then let $\psi'$ be an optimal assignment satisfying the conditions outlined in \thref{assumptions_optimal_matching}. As $\psi$ is not optimal,
    \begin{equation}\label{eq: 6}
        \sum_{z \in S_{1}(x)\setminus \triangle(x,y)} d(z, \psi'(z)) < \sum_{z \in S_{1}(x)\setminus \triangle(x,y)} d(z, \psi(z))
    \end{equation}
    must hold. Using $d(j,\psi(j)) = d(y,\psi(y)) = 1$, we obtain
    \begin{equation}\label{eq: 7}
        \sum_{z \in S_{1}(x)\setminus \triangle(x,y)} d(z, \psi(z)) = \sum_{z \in R_{x}(x,y)\setminus \{j\}}d(z,\phi(z)) + 2.
    \end{equation}
    On the other hand, we know $\psi'(y) \neq x$. Denote by $k$ the preimage of $x$ under $\psi'$. Then 
    \begin{equation}\label{eq: 8}
        \sum_{z \in S_{1}(x)\setminus \triangle(x,y)} d(z, \psi'(z)) = \sum_{z \in R_{x}(x,y)\setminus\{k\}} d(z,\psi'(z)) + 2.
    \end{equation}
    Combining equations \ref{eq: 6}, \ref{eq: 7} and \ref{eq: 8} leads to
    \begin{equation}\label{eq: 9}
        \sum_{z \in R_{x}(x,y)\setminus\{k\}} d(z,\psi'(z)) < \sum_{z \in R_{x}(x,y)\setminus \{j\}}d(z,\phi(z)).
    \end{equation}
    As $\psi'(y) \in R_{y}(x,y)$ and $k \in R_{x}(x,y)$ and $\psi'$ is bijective, the map \newline $\phi': R_{x}(x,y) \to R_{y}(x,y)$ defined by
    \begin{equation*}
        \phi'(z) = \begin{cases}
            \psi'(y), & \text{if $z = k$;}\\
            \psi'(z), & \text{otherwise.}
        \end{cases}
    \end{equation*}
     is an assignment between $R_{x}(x,y)$ and $R_{y}(x,y)$. Due to the optimality of $\phi$, we have
     \begin{equation*}
        \sum_{z \in R_{x}(x,y)} d(z, \phi'(z)) \geq \sum_{z \in R_{x}(x,y)} d(z, \phi(z)). 
     \end{equation*}

    We now distinguish the following two cases:

    \emph{Case 1:} $\sum_{z \in R_{x}(x,y)} d(z, \phi'(z)) = \sum_{z \in R_{x}(x,y)} d(z, \phi(z))$. In this case, $\phi'$ is an optimal assignment between $R_{x}(x,y)$ and $R_{y}(x,y)$ as well. 
    By the choice of $\phi$ and $j$, we have
    \begin{equation*}
        d(k, \phi'(k)) \leq d(j, \phi(j)),
    \end{equation*}
    leading to
    \begin{equation*}
        \sum_{z \in R_{x}(x,y)\setminus\{k\}} d(z,\psi'(z)) \geq \sum_{z \in R_{x}(x,y)\setminus \{j\}}d(z,\phi(z)),
    \end{equation*}
    contradicting equation \ref{eq: 9}.

    \emph{Case 2:} $\sum_{z \in R_{x}(x,y)} d(z, \phi'(z)) > \sum_{z \in R_{x}(x,y)} d(z, \phi(z))$.
    Due to equation \ref{eq: 9}, we have 
    \begin{equation*}
        2 \leq d(k, \phi'(k))) - d(j,\phi(j)).
    \end{equation*} 
    As $3 \geq d(k, \phi'(k))$ and $1 \leq d(j, \phi(j))$, we must have $d(j, \phi(j))= 1$ and therefore 
    \begin{equation*}
        1 = d(j,\phi(j)) \geq d(z,\phi(z)) \geq 1
    \end{equation*} 
    holds for any $z \in R_{x}(x,y)$, by the choice of $j$ and $\phi$. This contradicts equation \ref{eq: 9}, as 
    \begin{equation*}
        d(z, \psi'(z)) \geq 1, \; \forall z \in R_{x}(x,y).
    \end{equation*}
    Therefore our assumption was wrong and $\psi$ is an optimal assignment between $S_{1}(x)\setminus \triangle(x,y)$ and $S_{1}(y)\setminus \triangle(x,y)$. Using \thref{kappa_null}, we obtain
    \begin{align*}
        \kappa_{0}(x,y) &= \frac{1}{d}\Biggl(d-\sum_{z \in S_{1}(x)\setminus \triangle(x,y)} d(z, \psi(z))\Biggr) \\
        & = \frac{1}{d}\Biggl(d - \sum_{z \in R_{x}(x,y)\setminus\{j\}}d(z,\phi(z)) - 2\Biggr) \\
        & = \frac{1}{d}\Biggl(d +1 - \sum_{z \in R_{x}(x,y)}d(z,\phi(z)) \Biggr) - \frac{1}{d}\Biggl(3 - d(j, \phi(j))\Biggr),
    \end{align*}
    where we used equation \ref{eq: 7} for the second equality. Using the optimality of $\phi$ and \thref{Lin_Lu_Yau_curvature}, we obtain
    \begin{equation*}
        \kappa_{0}(x,y) = \kappa(x,y) - \frac{1}{d}\Biggl(3 - d(j, \phi(j))\Biggr).
    \end{equation*}
    The choice of $\phi$ and $j$ conclude the proof.
\end{proof}

We will employ this result in the subsequent chapter to examine the relation between $\kappa(x,y)$ and $\kappa_{0}(x,y)$ further.
However, before doing so, we will present an exact expression for the Ollivier-Ricci curvature on the first linear part using the formulas derived in this section.

\subsection{Ollivier-Ricci curvature on the first linear part}

The objective of this section is to establish a precise expression for the $\alpha$-Ollivier-Ricci curvature on the first linear part, that is, for $\alpha \in \interval[scaled]{0}{\frac{1}{d+1}}$. 

In \cite{bourne2018ollivier}, Bourne et al. provided a formula for $\kappa_{\alpha}(x,y)$ in terms of $\kappa_{0}(x,y)$  and $\kappa(x,y)$ for an edge $x \sim y$ where $x$ and $y$ have equal degree $d$. This formula is stated in the subsequent theorem.

\begin{theorem}[\cite{bourne2018ollivier}, Theorem 5.3]\thlabel{linear_formula}
    Let $G=(V,E)$ be a locally finite graph. Let $x,y\in V$ be of equal degree $d$ with $x\sim y$. Let $\alpha \in \interval[scaled]{0}{\frac{1}{d+1}}$. Then the $\alpha$-Ollivier-Ricci curvature
    \begin{equation*}
        \kappa_{\alpha}(x,y) = (1-\alpha)\kappa_{0}(x,y) + \alpha \cdot d \cdot (\kappa(x,y) - \kappa_{0}(x,y)).
    \end{equation*}
\end{theorem}

Using \thref{kappa_vergleich}, we obtain a precise expression for $\kappa(x,y) - \kappa_{0}(x,y)$ in terms of an optimal assignment between $R_{x}(x,y)$ and $R_{y}(x,y)$. Furthermore, \thref{kappa_null} provides a precise formula for $\kappa_{0}(x,y)$. Combining these results leads to the following theorem.

\begin{theorem}\thlabel{formula_first_linear_part}
    Let $G=(V,E)$ be a locally finite graph. Let $x,y\in V$ be of equal degree $d$ with $x\sim y$. Let $\alpha \in \interval[scaled]{0}{\frac{1}{d+1}}$. Then, if $\vert \triangle(x,y) \vert < d-1$, the $\alpha$-Ollivier-Ricci curvature
    \begin{equation*}
        \kappa_{\alpha}(x,y) = \frac{1-\alpha}{d} \Biggl(d - \inf_{\psi} \mathlarger{\sum}_{z \in S_{1}(x)\setminus\triangle(x,y)}d(z, \psi(z))\Biggr) + \alpha \Biggl(3 - \sup_{\phi \in \mathcal{O}_{xy}} \sup_{z \in R_{x}(x,y)} d(z, \phi(z))\Biggr),
    \end{equation*}
    where the infimum is taken over all assignment $\psi$ between $S_{1}(x)\setminus\triangle(x,y)$ and $S_{1}(y)\setminus\triangle(x,y)$. 
    Furthermore, if $\vert \triangle(x,y) \vert = d-1$, then
    \begin{equation*}
        \kappa_{\alpha}(x,y) = (1-\alpha)\frac{d-1}{d} + 2\alpha.
    \end{equation*}
\end{theorem}

A consequence of \thref{formula_first_linear_part} is the following Corollary.

\begin{corollary}
    Let $G=(V,E)$ be a locally finite graph. Let $x,y\in V$ be of equal degree $d$ with $x\sim y$. Let $\alpha \in \interval[scaled]{0}{\frac{1}{d+1}}$. Then 
    \begin{equation*}
        \kappa_{\alpha}(x,y) - 2\alpha  \leq (1-\alpha)\kappa_{0}(x,y) \leq \kappa_{\alpha}(x,y).
    \end{equation*}
\end{corollary}

Furthermore, \thref{linear_formula} implies that the idleness function is globally linear, namely
\begin{equation*}
    \kappa_{\alpha}(x,y) = (1-\alpha) \kappa(x,y) \quad \forall \alpha \in [0,1]
\end{equation*}
if $\kappa(x,y) = \kappa_{0}(x,y)$,
The next chapter is devoted to a further investigation of this case and, more generally, the relation between $\kappa(x,y)$ and $\kappa_{0}(x,y)$.

\section{Relation between two important curvature notions}

In this chapter, we shall examine the relation between the Lin-Lu-Yau curvature and the $\alpha$-Ollivier-Ricci curvature for vanishing idleness $\alpha=0$. By doing so, we will provide answers to open questions proposed by Bourne et al. in \cite{bourne2018ollivier}.

The foundation for the subsequent discussion lies in \thref{kappa_vergleich} from the previous chapter. For ease of reading, we restate this important theorem here. 

\begingroup
\def\thetheorem{\ref{kappa_vergleich}}
\begin{theorem}
    Let $G=(V,E)$ be a locally finite graph. Let $x,y \in V$ be of equal degree $d$ with $x\sim y$. If $\vert\triangle(x,y)\vert < d-1$, then 
    \begin{equation*}
        \kappa_{0}(x,y) = \kappa(x,y) - \frac{1}{d}\Biggl(3 - \sup_{\phi \in \mathcal{O}_{xy}} \sup_{z \in R_{x}(x,y)} d(z, \phi(z))\Biggr).
    \end{equation*}
    If $\vert \triangle(x,y) \vert=d-1$, then 
    \begin{equation*}
        \kappa_{0}(x,y) = \kappa(x,y)-\frac{2}{d}.
    \end{equation*}
\end{theorem}
\addtocounter{theorem}{-1}
\endgroup

\subsection{Equality condition}

In this section, we present our initial observations regarding the relationship between $\kappa_{0}(x,y)$ and $\kappa(x,y)$ for an edge $x \sim y$ where $d_{x} = d_{y}$. The main result will be a sufficient and necessary condition for equality between the two curvature notions.

\begin{corollary}\thlabel{vergleich_einfach}
    Let $G=(V,E)$ be a locally finite graph. Let $x,y\in V$ be of equal degree $d$ with $x \sim y$. Then,
    \begin{equation*}
        \kappa_{0}(x,y) = \kappa(x,y) - \frac{c}{d},
    \end{equation*}
    where $c \in \{0,1,2\}$.
\end{corollary} 

\begin{proof}
    We distinguish the following two cases:

    \emph{Case 1: $\vert \triangle(x,y) \vert = d-1$.} According to \thref{kappa_vergleich}, we have
    \begin{equation*}
        \kappa_{0}(x,y) = \kappa(x,y)-\frac{2}{d}.
    \end{equation*}

    \emph{Case 2: $\vert \triangle(x,y) \vert < d-1$.} This case is an immediate consequence of \thref{kappa_vergleich}, using $1 \leq d(i,j) \leq 3$ for any $i \in R_{x}(x,y)$ and any $j \in R_{y}(x,y)$.
\end{proof}

The main result of this section is the following necessary and sufficient condition for $\kappa_{0}(x,y) = \kappa(x,y)$ on regular graphs. 

\begin{corollary}\thlabel{equality_condition}
    Let $G=(V,E)$ be a locally finite graph. Let $x,y \in V$ be of equal degree $d$ with $x \sim y$. Then 
    \begin{equation*}
        \kappa_{0}(x,y) = \kappa(x,y)
    \end{equation*} 
    if and only if there exists an optimal assignment $\phi \in \mathcal{O}_{xy}$ between $R_{x}(x,y)$ and $R_{y}(x,y)$ such that
    \begin{equation*}
        \exists z \in R_{x}(x,y): \; d(z, \phi(z)) = 3.
    \end{equation*}
\end{corollary}

The corollary is an immediate consequence of \thref{kappa_vergleich}. 

\begin{remark}
Recall that every strongly regular graph $G=(V,E)$ is of diameter less than or equal to two. Thus, according to the previous corollary, every edge $x \sim y$ satisfies:
\begin{equation*}
    \kappa(x,y) > \kappa_{0}(x,y).
\end{equation*}
The same holds for arbitrary regular graphs $G=(V,E)$ with a diameter less than or equal to two, e.g., for $d$-regular graphs with $d\geq \frac{\vert V \vert-1}{2}$.
\end{remark}

We now use \thref{equality_condition} to determine intervals where $\kappa$ and $\kappa_{0}$ always coincide and intervals where they are never equal.

\begin{corollary}
    Let $G=(V,E)$ be a locally finite graph. Let $x,y \in V$ be of equal degree $d$ with $x \sim y$. If 
    \begin{equation*}
        \kappa(x,y) < -1 + \frac{3}{d} + 2 \frac{\vert \triangle(x,y) \vert}{d},
    \end{equation*}
    then $\kappa(x,y) = \kappa_{0}(x,y)$.
\end{corollary}

\begin{proof}
    Assume $\kappa(x,y) \neq \kappa_{0}(x,y)$ and 
    \begin{equation*}
        \kappa(x,y) < -1 + \frac{3}{d} + 2 \frac{\vert \triangle(x,y) \vert}{d}.
    \end{equation*}
    Let $\phi \in \mathcal{O}_{xy}$ be an optimal assignment between $R_{x}(x,y)$ and $R_{y}(x,y)$.
    According to \thref{equality_condition}, we have $d(z, \phi(z)) \leq 2$ for any $z \in R_{x}(x,y)$. Hence,
    \begin{align*}
        \kappa(x,y) &= \frac{1}{d} \Biggl(d+1 - \sum_{z \in R_{x}(x,y)}d(z, \phi(z))\Biggr) \\
                    &\geq \frac{1}{d} \Biggl(d + 1 - 2 \vert R_{x}(x,y)\vert\Biggr) \\
                    &= \frac{1}{d} \Biggl(-d + 3 + 2 \vert \triangle(x,y) \vert\Biggr),
    \end{align*}
    contradicting our assumptions.
\end{proof}

\begin{corollary}
    Let $G=(V,E)$ be a locally finite graph. Let $x,y \in V$ be of equal degree $d$ with $x \sim y$. If 
    \begin{equation*}
        \kappa(x,y) > \frac{\vert \triangle(x,y) \vert}{d},
    \end{equation*}
    then $\kappa(x,y) \neq \kappa_{0}(x,y)$.
\end{corollary}

\begin{proof}
    This is an immediate consequence of \thref{kappa_null_in_Z}.
\end{proof}

\begin{remark}
    Thus, $\kappa(x,y)$ and $\kappa_{0}(x,y)$ coincide for small values but differ for large values.
\end{remark}

\subsection{Opposite signs}

We have seen that $\kappa\geq \kappa_{0}$ holds for regular graphs. In \cite[Theorem 5.5]{bourne2018ollivier}, the authors extend this result to arbitrary locally finite graphs.  It is known from the literature that graphs with $\kappa(x,y) = 0$ and $\kappa_{0}(x,y) < 0$, for any edge $x \sim y$, exist. The Petersen graph is one example of a graph satisfying this property. The objective of this section is to answer the question of whether there exists a graph $G$ with an edge $x\sim y$ satisfying $\kappa(x,y) > 0$ and $\kappa_{0}(x,y) <0$. 

According to the subsequent theorem, this cannot be the case for regular graphs. This result was first obtained in \cite{cushing2021curvatures}. Using our earlier results, we can provide a significantly simplified proof.

\begin{theorem}\thlabel{larger_implies_equal_zero}
    Let $G=(V,E)$ be a locally finite graph. Let $x,y \in V$ be of equal degree $d$ with $x \sim y$. If $\kappa(x,y) > 0$, then $\kappa_{0}(x,y) \geq 0$.
\end{theorem}

\begin{proof}
    Let $G=(V,E)$ be a locally finite graph and let $x,y \in V$ be of equal degree $d$ with $x\sim y$. We argue by contradiction. 
    
    Assume that $\kappa(x,y) > 0$ and $\kappa_{0}(x,y) < 0$. Recall that $\kappa(x,y) \in \mathbb{Z}/d$ and $\kappa_{0}(x,y) \in \mathbb{Z}/d$. Using \thref{vergleich_einfach}, we deduce that 
    \[
    \kappa(x,y) = \frac{1}{d} \quad \text{and} \quad \kappa_{0}(x,y) = - \frac{1}{d},
    \]
    as $\kappa(x,y) - \kappa_{0}(x,y) \leq \frac{2}{d}$. Furthermore, observe that $\vert \triangle(x,y) \vert < d-1$, as $\kappa_{0}(x,y) < 0$. Hence, we can apply \thref{kappa_vergleich}, leading to the existence of a $\phi \in O_{xy}$ such that
    \begin{equation*}
        \sup_{z \in R_{x}(x,y)}d(z,\phi(z)) = 1.
    \end{equation*}
    Using our precise expression for the Lin-Lu-Yau curvature, we obtain
    \begin{align*}
        \kappa(x,y) &= \frac{1}{d}\Biggl(d+1-\sum_{z \in R_{x}(x,y)}d(z, \phi(z))\Biggr)\\
                    &= \frac{1}{d}\Biggl(d+1- \vert R_{x}(x,y) \vert\Biggr)\\
                    &= \frac{2 + \vert \triangle(x,y) \vert}{d} \\
                    &\geq \frac{2}{d},
    \end{align*}
    contradicting $\kappa(x,y) = \frac{1}{d}$. 
\end{proof}

The subsequent graphs show that the condition $d_{x} = d_{y}$ is necessary. Namely, there exist non-regular graphs $G$ with an edge $x \sim y$ satisfying 
\begin{equation*}
    \kappa(x,y) > 0 >\kappa_{0}(x,y).
\end{equation*}

\emph{Example 1:} 
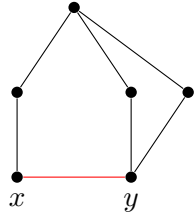
\begin{center}
    \begin{tikzpicture}[x=1.5cm, y=1.5cm,
        vertex/.style={
            shape=circle, fill=black, inner sep=1.5pt	
        }
    ]
    
    \node[vertex, label=below:$x$] (1) at (0, 0) {};
    \node[vertex, label=below:$y$] (2) at (1, 0) {};
    \node[vertex] (3) at (0, 0.75) {};
    \node[vertex] (4) at (1, 0.75) {};
    \node[vertex] (5) at (1.5, 0.75) {};
    \node[vertex] (6) at (0.5, 1.5) {};
    
    \draw[red] (1) -- (2);
    \draw (1) -- (3);
    \draw (2) -- (4);
    \draw (2) -- (5);
    \draw (6) -- (3);
    \draw (6) -- (4);
    \draw (6) -- (5);
    \end{tikzpicture}
\end{center}
The edge $x\sim y$, with $d_{x}=2 < 3= d_{y}$, satisfies 
\begin{equation*}
    \kappa_{0}(x,y) = -\frac{1}{6} < 0 < \frac{1}{6} = \kappa(x,y).
\end{equation*}

\emph{Example 2:}
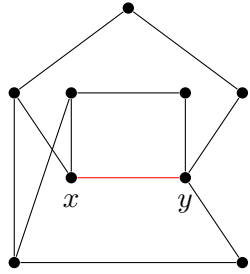
\begin{center}
    \begin{tikzpicture}[x=1.5cm, y=1.5cm,
        vertex/.style={
            shape=circle, fill=black, inner sep=1.5pt	
        }
    ]
    
    \node[vertex, label=below:$x$] (1) at (0, 0) {};
    \node[vertex, label=below:$y$] (2) at (1, 0) {};
    \node[vertex] (3) at (0, 0.75) {};
    \node[vertex] (4) at (1, 0.75) {};
    \node[vertex] (5) at (-0.5, 0.75) {};
    \node[vertex] (6) at (1.5, 0.75) {};
    \node[vertex] (7) at (0.5, 1.5) {};
    \node[vertex] (8) at (-0.5, -0.75) {};
    \node[vertex] (9) at (1.5, -0.75) {};
    
    \draw[red] (1) -- (2);
    \draw (1) -- (3);
    \draw (1) -- (5);
    \draw (2) -- (9);
    \draw (2) -- (4);
    \draw (3) -- (4);
    \draw (2) -- (6);
    \draw (5) -- (7);
    \draw (6) -- (7);
    \draw (3) -- (8);
    \draw (5) -- (8);
    \draw (8) -- (9);
    \end{tikzpicture}
\end{center}
The edge $x\sim y$, with $d_{x}=3 < 4= d_{y}$, satisfies 
\begin{equation*}    
    \kappa_{0}(x,y) = -\frac{1}{6} < 0 < \frac{1}{12} = \kappa(x,y).
\end{equation*}

\subsection{Bone idleness}\label{section_bone_idle}

We introduce the notion of bone idle, initially proposed by Bourne et al. in \cite{bourne2018ollivier}.

\begin{definition}[Bone idle]
    Let $G=(V,E)$ be a locally finite graph. We say an edge $x \sim y$ is \textit{bone idle} if $\kappa_{\alpha}(x,y) = 0$ for every $\alpha \in [0,1]$. We say that $G$ is \textit{bone idle} if every edge is bone idle.
\end{definition}

The notion of Ricci-flatness is weaker but strongly related to bone idleness.

\begin{definition}[Ricci-flat]
    Let $G=(V,E)$ be a locally finite graph. We call $G$ \textit{Ricci-flat} if $\kappa(x,y) = 0$ for every edge. We call $G$ \textit{$\alpha$-Ricci-flat} if $\kappa_{\alpha}(x,y) = 0$ for every edge.
\end{definition}

\begin{lemma}
Let $G=(V,E)$ be a locally finite graph. Let $x,y \in V$ with $x\sim y$. Then the following are equivalent:
\begin{enumerate}
    \item[$(i)$] $\kappa_{\alpha}(x,y) = 0$ for all $\alpha \in [0,1]$.
    \item[$(ii)$] $\kappa_{0}(x,y) = \kappa(x,y) = 0$.
\end{enumerate}
\end{lemma}

\begin{proof}
\emph{(ii)$\implies$(i)}. Let $\alpha \in (0,1)$ be arbitrary. Assume $\kappa_{0}(x,y) = \kappa(x,y) = 0$. Recall that the idleness function is concave, and also note that $\kappa_{1}(x,y)=0$. Hence,
\begin{equation*}
    \kappa_{\alpha}(x,y) \geq \alpha \kappa_{1}(x,y) + (1-\alpha) \kappa_{0}(x,y) = 0.
\end{equation*}
The other inequality follows from the fact that the graph of a concave function lies below its tangent line at each point and that $\kappa_{1}' = -\kappa$:
\begin{equation*}
    \kappa_{\alpha}(x,y) \leq \kappa_{1}(x,y) + \kappa_{1}'(x,y)\cdot(\alpha - 1) = \kappa(x,y) \cdot (1-\alpha) =  0.
\end{equation*}

\emph{(i) $\implies$ (ii)}. This is an immediate consequence of the definition of $\kappa$.
\end{proof}

Therefore, a graph is bone idle if it is Ricci-flat and 0-Ricci-flat.
Previous works have addressed the classification of Ricci-flat graphs. Bhattacharya et al. \cite{bhattacharya2015exact} classified all graphs that are 0-Ricci-flat and have girth at least five.

\begin{theorem}[\cite{bhattacharya2015exact}, Corollary 4.1]\thlabel{0_ricci_flat}
    Let $G=(V,E)$ be a locally finite graph with girth at least five. Suppose that $G$ is 0-Ricci-flat. Then $G$ is isomorphic to one of the following graphs:
    \begin{enumerate}
        \item[$(i)$] The infinite path,
        \item[$(ii)$] the infinite ray, 
        \item[$(iii)$] the path $P_{n}$ for $n \geq 2$,
        \item[$(iv)$] the cycle graph $C_{n}$ for $n \geq 5$,
        \item[$(v)$] the star graph $T_{n}$ for $n\geq 3$.   
    \end{enumerate}
\end{theorem}

On the other hand, Cushing et al. classified all Ricci-flat graphs with girth at least five \cite{cushing2018erratum}.

\begin{theorem}[\cite{cushing2018erratum}, Theorem 1]\thlabel{ricci_flat}
    Let $G=(V,E)$ be a locally finite graph with girth at least five. Suppose that $G$ is Ricci-flat. Then $G$ is isomorphic to one of the following graphs:
    \begin{enumerate}
        \item[$(i)$] The infinite path,
        \item[$(ii)$] the cycle graph $C_{n}$ for $n \geq 6$,
        \item[$(iii)$] the dodecahedral graph,
        \item[$(iv)$] the Petersen graph,
        \item[$(v)$] the half-dodecahedral graph,
        \item[$(vi)$] the Triplex graph.
    \end{enumerate}
\end{theorem}

Combining \thref{0_ricci_flat} and \thref{ricci_flat} yields the following result:

\begin{corollary}\thlabel{bone_idle_girth_5}
    Let $G=(V,E)$ be a locally finite graph with girth at least five. Suppose that $G$ is bone idle. Then $G$ is isomorphic to one of the following graphs:
    \begin{enumerate}
        \item[$(i)$] The infinite path,
        \item[$(ii)$] the cycle graph $C_{n}$ for $n \geq 6$. 
    \end{enumerate}
\end{corollary}

\begin{remark}
    Hence, there are no bone idle graphs with girth equal to five.
\end{remark}

The full classification of Ricci-flat and bone idle graphs appears to be a difficult graph theory problem, which is still open. 
In what follows, we will give necessary and sufficient conditions on the structure of Ricci-flat and bone idle regular graphs of girth 4.

\begin{theorem}\thlabel{ricci_flat_condition}
    Let $G=(V,E)$ be a locally finite graph. Let $x,y\in V$ be of equal degree $d$ with $x\sim y$. Furthermore, assume that $\triangle(x,y) = \emptyset$. Then $\kappa(x,y) = 0$ if and only if one of the following holds:
    \begin{enumerate}
        \item[$(i)$] There exists an optimal assignment $\phi \in \mathcal{O}_{xy}$ such that $\square(\phi) = d-2$ and $\pentago(\phi) = 0$.
        \item[$(ii)$] There exists an optimal assignment $\phi \in \mathcal{O}_{xy}$ such that $\square(\phi) = d-3$ and $\pentago(\phi) = 2$.
    \end{enumerate}
\end{theorem}

\begin{proof}
    Assume $\kappa(x,y) = 0$. Let $\phi \in \mathcal{O}_{xy}$ be an optimal assignment between $R_{x}(x,y)$ and $R_{y}(x,y)$. Then,
    \begin{equation*}
        \kappa(x,y) = \frac{1}{d}\Biggl(-2d + 4 + 2\vert \square(\phi) \vert + \vert \pentago(\phi) \vert\Biggr) = 0.
    \end{equation*}
    Using that $\vert \square(\phi) \vert + \vert \pentago(\phi) \vert \leq d-1$,  we conclude that one of the following cases must hold:
    \begin{enumerate}
        \item[$(i)$] $\square(\phi) = d-2$ and $\pentago(\phi) = 0$, or 
        \item[$(ii)$] $\square(\phi) = d-3$ and $\pentago(\phi) = 2$.
    \end{enumerate}
    Conversely, suppose that property $(i)$ or $(ii)$ is satisfied. Then 
    \begin{equation*}
        2 \vert \square(\phi) \vert + \vert \pentago(\phi) \vert = 2d -4,
    \end{equation*}
    and hence 
    \begin{equation*}
        \kappa(x,y) = \frac{1}{d}\Biggl(-2d + 4 + 2d - 4\Biggr) = 0.
    \end{equation*}
\end{proof}

Next, we present a necessary and sufficient condition for an edge $x\sim y$ to have $\kappa_{0}(x,y)=0$. This condition was already established by Bhattacharya et al. \cite{bhattacharya2015exact}.

\begin{theorem}
    Let $G=(V,E)$ be a locally finite graph. Let $x,y\in V$ be of equal degree $d$ with $x\sim y$. Furthermore, assume that $\triangle(x,y) = \emptyset$. Then $\kappa_{0}(x,y) = 0$ if and only if there exists a perfect matching between $S_{1}(x)$ and $S_{1}(y)$.
\end{theorem}

\begin{proof}
    Assume $\kappa_{0}(x,y) = 0$. Let $\phi$ be an optimal assignment between $S_{1}(x)$ and $S_{1}(y)$. Using the formula for $\kappa_{0}(x,y)$ derived in \thref{kappa_null}, we obtain
    \begin{equation}\label{eq: 10}
        \mathlarger{\sum}_{z \in S_{1}(x)}d(z, \phi(z)) = d.
    \end{equation}
    Since $S_{1}(x) \cap S_{1}(y) = \triangle(x,y) = \emptyset$, we must have $d(z,\phi(z)) \geq 1$ for all $z \in S_{1}(x)$. Combining this with \ref{eq: 10}, we conclude that 
    $d(z,\phi(z)) = 1$ for all $z \in S_{1}(x)$. Hence, $\phi$ is a perfect matching.\par
    Conversely, suppose there exists a perfect matching $\phi$. Since $S_{1}(x) \cap S_{1}(y) = \triangle(x,y) = \emptyset$, this is an optimal assignment. Then by \thref{kappa_null},
    \begin{equation*}
        \kappa_{0}(x,y) = \frac{1}{d}\Biggl(d - \mathlarger{\sum}_{z \in S_{1}(x)}d(z, \phi(z))\Biggr) = 0.
    \end{equation*}
    This concludes the proof.
\end{proof}

Therefore, a $0$-Ricci-flat, regular graph of girth four must have a perfect matching between the neighborhoods $S_{1}(x)$ and $S_{1}(y)$ of every edge $x\sim y$. Examples are the $n$-dimensional hypercube $\mathcal{Q}_{n}$, the $n$-dimensional integer lattice $\mathbb{Z}^{n}$ and the complete bipartite graph $K_{n,n}$.

The subsequent theorem provides a necessary and sufficient condition for an edge in a graph of girth four to be bone idle.

\begin{theorem}\thlabel{bone_idle}
    Let $G=(V,E)$ be a locally finite graph. Let $x,y\in V$ be of equal degree $d$ with $x\sim y$. Furthermore, assume that $\triangle(x,y) = \emptyset$. Then the edge $x\sim y$ is bone idle if and only if there exists an optimal assignment $\phi \in \mathcal{O}_{xy}$ such that $\vert \square(\phi)\vert  = d-2$ and $\vert \pentago(\phi)\vert  = 0$.
\end{theorem}

\begin{proof}
    Assume the edge $x \sim y$ is bone idle, i.e., $\kappa(x,y) = \kappa_{0}(x,y) = 0$. 
    According to \thref{equality_condition} there exists an optimal assignment $\phi \in \mathcal{O}_{xy}$ between $R_{x}(x,y)$ and $R_{y}(x,y)$ and an $z \in R_{x}(x,y)$, such that $d(z,\phi(z)) = 3$. Therefore, 
    \begin{equation}\label{eq: 300}
        \vert\square(\phi) \vert+ \vert\pentago(\phi)\vert \leq d-2.
    \end{equation}
    As $\kappa(x,y) = 0$ we can apply \thref{ricci_flat_condition}. Using equation \ref{eq: 300}, we conclude that $\vert \square(\phi) \vert = d-2$ and $\vert \pentago(\phi)\vert = 0$ must hold.
    \par
    Conversely, assume that there exists an optimal assignment $\phi \in \mathcal{O}_{xy}$ that satisfies $\vert \square(\phi) \vert = d-2$ and $\vert \pentago(\phi)\vert = 0$. Since 
    \begin{equation*}
        \vert\square(\phi) \vert+ \vert\pentago(\phi)\vert = d-2 < \vert R_{x}(x,y) \vert,
    \end{equation*} 
    there exists a $z \in R_{x}(x,y)$ such that $d(z, \phi(z)) = 3$. Hence, according to \thref{equality_condition}, we have $\kappa(x,y) = \kappa_{0}(x,y)$.
    According to \thref{ricci_flat_condition}, we have $\kappa(x,y) = 0$, which closes the proof.
\end{proof}

Both the complete bipartite graph $K_{n,n}$ and the $n$-dimensional hypercube $\mathcal{Q}_{n}$ are $0$-Ricci-flat, regular graphs of girth four. Note that there exists a perfect matching between $R_{x}(x,y)$ and $R_{y}(x,y)$ for any edge $x \sim y$ in the complete bipartite graph $K_{n,n}$. The same holds for the $n$-dimensional hypercube $\mathcal{Q}_{n}$. Therefore, according to the previous theorem, neither of the graphs is bone idle. Using our precise expression for the Lin-Lu-Yau curvature, we obtain that for both the complete bipartite graph $K_{n,n}$ and the $n$-dimensional hypercube $\mathcal{Q}_{n}$
\begin{equation*}
    \kappa(x,y) = \frac{2}{n},
\end{equation*}
for every edge $x \sim y$. Thus, the graphs $K_{n,n}$ and $\mathcal{Q}_{n}$ satisfy $\kappa(x,y) >0$ and $\kappa_{0}(x,y)=0$ for all edges $x \sim y$.

A direct implication of \thref{bone_idle} is that the $n$-dimensional integer lattice $\mathbb{Z}^{n}$ is bone idle. 

Now, we shift our focus to constructing finite regular graphs of girth four that are bone idle. The following theorem demonstrates that this is not possible for $3$-regular graphs.

\begin{theorem}\thlabel{bone_idle_grith_3}
    Let $G=(V,E)$ be a locally finite graph. Suppose that $G$ is bone idle, then $G$ is not $3$-regular.
\end{theorem}

\begin{figure}
    \center 
    \begin{subfigure}{0.4\textwidth}
        \centering
        \begin{tikzpicture}[x=1.5cm, y=1.5cm,
            vertex/.style={
                shape=circle, fill=black, inner sep=1.5pt	
            }
        ]
        
        \node[vertex, label=below:$x$] (1) at (0, 0) {};
        \node[vertex, label=below:$y$] (2) at (1, 0) {};
        \node[vertex, label=above:$x_{1}$] (3) at (0, 1) {};
        \node[vertex, label=above: $y_{1}$] (4) at (1, 1) {};
        \node[vertex, label=below:$x_{2}$] (5) at (-1, 0) {};
        \node[vertex, label=above:$z$] (6) at (-1, 1) {};
        
        \draw (1) -- (2);
        \draw (1) -- (5);
        \draw[red] (1) -- node[midway,right] {$\kappa>0$} ++ (3);
        \draw (2) -- (4);
        \draw (3) -- (4);
        \draw (5) -- (6);
        \draw (3) -- (6);    
        \end{tikzpicture}
      \end{subfigure}
      \begin{subfigure}{0.4\textwidth}
        \centering
        \begin{tikzpicture}[x=1.5cm, y=1.5cm,
            vertex/.style={
                shape=circle, fill=black, inner sep=1.5pt	
            }
        ]
        
        \node[vertex, label=left:$x$] (1) at (0, 0) {};
        \node[vertex, label=right:$y$] (2) at (1, 0) {};
        \node[vertex, label=above:$x_{1}$] (3) at (0, 1) {};
        \node[vertex, label=above:$y_{1}$] (4) at (1, 1) {};
        \node[vertex, label=below:$x_{2}$] (5) at (0, -1) {};
        \node[vertex, label=below:$y_{2}$] (6) at (1, -1) {};
        
        \draw[red] (1) -- node[midway,below] {$\kappa>0$} ++ (2);
        \draw (1) -- (3);
        \draw (1) -- (5);
        \draw (2) -- (4);
        \draw (2) -- (6);
        \draw (3) -- (4);
        \draw (5) -- (6);

        \end{tikzpicture}
      \end{subfigure}
    \caption{Illustration of the two possible cases in \thref{bone_idle_grith_3}.}
\end{figure}

\begin{proof}
    We argue by contradiction. Assume $G$ is a $3$-regular graph that is bone idle. According to \thref{bone_idle_girth_5}, the girth of $G$ must be less than five. Furthermore, there exists no edge $x \sim y \in E$ with $\vert \triangle(x,y) \vert > 0$, as otherwise, using \thref{Lin_Lu_Yau_in_Z},
    \begin{equation*}
        \kappa(x,y) \geq -2 + \frac{4}{3} + 3 \frac{\vert \triangle(x,y) \vert}{3} > 0.
    \end{equation*}
    Hence, the girth of $G$ must be equal to four. Let $x\sim y$ be an edge contained in a $4$-cycle. Denote by $x_{1}, x_{2}$ and $y_{1}, y_{2}$ the other two neighbors of $x$ and $y$, respectively. Without loss of generality, we may assume that $x_{1} \sim y_{1}$, as $x \sim y$ is contained in a $4$-cycle. Observe that $x_{2} \sim y_{1}$ and $y_{2} \sim x_{1}$ cannot hold true at the same time. Otherwise, there exists a perfect matching between $S_{1}(x)\setminus\{y\}$ and $S_{1}(y)\setminus\{x\}$, leading to $\kappa(x,y) = \frac{2}{3} > 0$. 
    Therefore, we may assume without loss of generality that $x_{2} \not\sim y_{1}$.

    According to \thref{bone_idle}, there must be a $4$-cycle based on the edge $x \sim x_{2}$. Using that $x_{2} \not\sim y_{1}$, one of the following cases must be true:

    \emph{Case 1:} There is a $z \in S_{1}(x_{2})\setminus\{x\}$ such that $z \sim x_{1}$. In this case we have $\kappa(x, x_{1})>0$, contradicting the bone idleness of $G$. 

    \emph{Case 2:} $x_{2} \sim y_{2}$. In this case, we have $\kappa(x,y) >0$, contradicting the bone idleness of $G$. 
    
    This concludes the proof.
\end{proof}

Hence, there are no bone idle graphs that are 3-regular. However, the subsequent construction demonstrates the existence of 4-regular, bone idle graphs. 

According to \thref{bone_idle}, a 4-regular graph is bone idle if, for every edge $x \sim y$, there exists an optimal assignment $\phi \in \mathcal{O}_{xy}$ such that $\vert \square(\phi) \vert = 2$ and $\vert \pentago(\phi) \vert = 0$.

\begin{figure}
    \center 
    \begin{subfigure}{0.4\textwidth}
        \centering
        \begin{tikzpicture}[x=1.5cm, y=1.5cm,
            vertex/.style={
                shape=circle, fill=black, inner sep=1.5pt	
            }
        ]
        
        \node[vertex] (1) at (0, 0) {};
        \node[vertex] (7) at (-0.5, 0) {};
        \node[vertex] (2) at (0.5, 0.5) {};
        \node[vertex] (8) at (0.5, 1) {};
        \node[vertex] (3) at (1, 0.5) {};
        \node[vertex] (9) at (1, 1) {};
        \node[vertex] (4) at (1.5, 0) {};
        \node[vertex] (10) at (2, 0) {};
        \node[vertex] (5) at (1, -0.5) {};
        \node[vertex] (11) at (1, -1) {};
        \node[vertex] (6) at (0.5, -0.5) {};
        \node[vertex] (12) at (0.5, -1) {};

        \draw (1) -- (2);
        \draw (2) -- (3);
        \draw (3) -- (4);
        \draw (4) -- (5);
        \draw (5) -- (6);
        \draw (6) -- (1);
        \draw (7) -- (8);
        \draw (8) -- (9);
        \draw (9) -- (10);
        \draw (10) -- (11);
        \draw (11) -- (12);
        \draw (12) -- (7);
        \draw (7) -- (2);
        \draw (7) -- (6);
        \draw (8) -- (1);
        \draw (8) -- (3);
        \draw (9) -- (2);
        \draw (9) -- (4);
        \draw (10) -- (3);
        \draw (10) -- (5);
        \draw (11) -- (4);
        \draw (11) -- (6);
        \draw (12) -- (5);
        \draw (12) -- (1);

        \end{tikzpicture}
      \end{subfigure}
      \begin{subfigure}{0.4\textwidth}
        \centering
        \begin{tikzpicture}[x=1.5cm, y=1.5cm,
            vertex/.style={
                shape=circle, fill=black, inner sep=1.5pt	
            }
        ]
        \node[vertex] (1) at (0, 0.25) {};
        \node[vertex] (2) at (0.5, 0.5) {};
        \node[vertex] (3) at (1, 0.5) {};
        \node[vertex] (4) at (1.5, 0) {};
        \node[vertex] (5) at (1, -0.5) {};
        \node[vertex] (6) at (0.5, -0.5) {};
        \node[vertex] (7) at (0, -0.25) {};
        \node[vertex] (8) at (-0.5, 0.25) {};
        \node[vertex] (9) at (0.5, 1) {};
        \node[vertex] (10) at (1, 1) {};
        \node[vertex] (11) at (2, 0) {};
        \node[vertex] (12) at (1, -1) {};
        \node[vertex] (13) at (0.5, -1) {};
        \node[vertex] (14) at (-0.5, -0.25) {};

        \draw (1) -- (2);
        \draw (2) -- (3);
        \draw (3) -- (4);
        \draw (4) -- (5);
        \draw (5) -- (6);
        \draw (6) -- (7);
        \draw (7) -- (1);

        \draw (8) -- (9);
        \draw (9) -- (10);
        \draw (10) -- (11);
        \draw (11) -- (12);
        \draw (12) -- (13);
        \draw (13) -- (14);
        \draw (14) -- (8);

        \draw (8) -- (7);
        \draw (8) -- (2);
        \draw (9) -- (1);
        \draw (9) -- (3);
        \draw (10) -- (2);
        \draw (10) -- (4);
        \draw (11) -- (3);
        \draw (11) -- (5);
        \draw (12) -- (4);
        \draw (12) -- (6);
        \draw (13) -- (5);
        \draw (13) -- (7);
        \draw (14) -- (1);
        \draw (14) -- (6);

    \end{tikzpicture}
    \end{subfigure}
      \begin{subfigure}{0.4\textwidth}
        \centering
        \begin{tikzpicture}[x=1.5cm, y=1.5cm,
            vertex/.style={
                shape=circle, fill=black, inner sep=1.5pt	
            }
        ]
        \node[vertex] (1) at (0, 0.25) {};
        \node[vertex] (9) at (-0.5, 0.25) {};
        \node[vertex] (8) at (0, -0.25) {};
        \node[vertex] (16) at (-0.5, -0.25) {};
        \node[vertex] (2) at (0.5, 0.5) {};
        \node[vertex] (10) at (0.5, 1) {};
        \node[vertex] (3) at (1, 0.5) {};
        \node[vertex] (11) at (1, 1) {};
        \node[vertex] (4) at (1.5, 0.25) {};
        \node[vertex] (5) at (1.5, -0.25) {};
        \node[vertex] (12) at (2, 0.25) {};
        \node[vertex] (13) at (2, -0.25) {};
        \node[vertex] (6) at (1, -0.5) {};
        \node[vertex] (14) at (1, -1) {};
        \node[vertex] (7) at (0.5, -0.5) {};
        \node[vertex] (15) at (0.5, -1) {};

        \draw (1) -- (2);
        \draw (2) -- (3);
        \draw (3) -- (4);
        \draw (4) -- (5);
        \draw (5) -- (6);
        \draw (6) -- (7);
        \draw (7) -- (8);
        \draw (8) -- (1);

        \draw (9) -- (10);
        \draw (10) -- (11);
        \draw (11) -- (12);
        \draw (12) -- (13);
        \draw (13) -- (14);
        \draw (14) -- (15);
        \draw (15) -- (16);
        \draw (16) -- (9);

        \draw (9) -- (2);
        \draw (9) -- (8);
        \draw (10) -- (1);
        \draw (10) -- (3);
        \draw (11) -- (2);
        \draw (11) -- (4);
        \draw (12) -- (3);
        \draw (12) -- (5);
        \draw (13) -- (4);
        \draw (13) -- (6);
        \draw (14) -- (5);
        \draw (14) -- (7);
        \draw (15) -- (6);
        \draw (15) -- (8);
        \draw (16) -- (7);
        \draw (16) -- (1);

    \end{tikzpicture}
    \end{subfigure}
    \caption{The bone idle graphs $BI_{6}, BI_{7}$ and $BI_{8}$.}
    \label{bone_idle_graphs}
\end{figure}

\begin{figure}
    \centering
   
    \begin{tikzpicture}[x=1.5cm, y=1.5cm,
            vertex/.style={
                shape=circle, fill=black, inner sep=1.5pt	
            }
        ]
        \node[vertex] (1) at (0, 0) {};
        \node[vertex] (2) at (1, 0) {};
        \node[vertex] (3) at (2, 0) {};
        \node[vertex] (4) at (3, 0) {};
        \node[vertex] (5) at (4, 0) {};
        \node[vertex] (6) at (5, 0) {};
        \node[vertex] (7) at (6, 0) {};

        \node[vertex] (8) at (0, 1) {};
        \node[vertex] (9) at (1, 1) {};
        \node[vertex] (10) at (2, 1) {};
        \node[vertex] (11) at (3, 1) {};
        \node[vertex] (12) at (4, 1) {};
        \node[vertex] (13) at (5, 1) {};
        \node[vertex] (14) at (6, 1) {};

        \draw (1) -- (2);
        \draw (2) -- (3);
        \draw (3) -- (4);
        \draw (4) -- (5);
        \draw (5) -- (6);
        \draw (6) -- (7);

        \draw (8) -- (9);
        \draw (9) -- (10);
        \draw (10) -- (11);
        \draw (11) -- (12);
        \draw (12) -- (13);
        \draw (13) -- (14);

        \draw (1) -- (9);
        \draw (2) -- (8);
        \draw (2) -- (10);
        \draw (3) -- (9);
        \draw (3) -- (11);
        \draw (4) -- (10);
        \draw (4) -- (12);
        \draw (5) -- (11);
        \draw (5) -- (13);
        \draw (6) -- (12);
        \draw (6) -- (14);
        \draw (7) -- (13);

        \draw[dashed] (1) -- (-1,0);
        \draw[dashed] (1) -- (-1,1);
        \draw[dashed] (8) -- (-1,1);
        \draw[dashed] (8) -- (-1,0);

        \draw[dashed] (7) -- (7,0);
        \draw[dashed] (7) -- (7,1);
        \draw[dashed] (14) -- (7,1);
        \draw[dashed] (14) -- (7,0);
    \end{tikzpicture}
    \caption{Bone idle, $4$-regular graph with an infinite number of vertices.}
    \label{bone_idle_infinite}
\end{figure}

To construct a graph satisfying this property, we place an $n$-cycle containing the vertices $x_{0}, \dots, x_{n-1}$ inside another $n$-cycle containing the vertices $y_{0}, \dots, y_{n-1}$. We then add edges between the vertices $y_{k}$ and $x_{(k-1)\mod n}$ as well as between $y_{k}$ and $x_{(k+1)\mod n}$. We denote these graphs by $BI_{n}$. It is evident that the graph $BI_{n}$ is $4$-regular and bone idle if $n \geq 6$. Refer to Figure \ref{bone_idle_graphs} for an illustration of the graphs $BI_{6}, BI_{7}$ and $BI_{8}$.

Finally, it is noteworthy that a similar construction yields a bone idle, 4-regular graph with a girth of 4 and an infinite number of vertices. Please refer to Figure \ref{bone_idle_infinite} for an illustration.

\begin{remark}
    In the subsequent chapter, we will see that the graph $BI_{6}$ is the smallest bone idle, regular graph of girth less than five.
\end{remark}

\section{Implementation}

In this chapter, we introduce an efficient implementation of the Ollivier-Ricci curvature on regular graphs. Additionally, we explore an interesting problem in graph theory: the enumeration of all regular graphs up to isomorphism. 
Leveraging these findings, we classify all regular, bone idle graphs with a small number of vertices and establish a rigidity theorem for cocktail party graphs. Finally, we present a condition that ensures a regular graph $G$ satisfies $\textit{Ric}(G) >0$.

\subsection{Linear assignment problem}

As demonstrated in preceding chapters, the computation of the Ollivier-Ricci curvature on regular graphs simplifies to identifying an optimal assignment between two distinct sets of vertices of equal cardinality, denoted by $n$.

This problem referred to as the \textit{linear assignment problem}, stands as a fundamental task in linear optimization. With a rich history dating back to the early 20th century, the linear assignment problem 
has garnered significant attention in the literature. Thereby leading to the development of efficient algorithms for its solution. The earliest algorithms, proposed in the 1940s, exhibited a complexity of $O(2^n n^2)$. 
More recently, a wide range of algorithms with polynomial time complexity have been proposed. The most famous among these is the \textit{Hungarian algorithm}, developed by Kuhn in 1955 \cite{Kuhn1955Hungarian}. The time complexity $O(n^4)$ of the original algorithm was later improved to $O(n^3)$ by Edmonds and Karp \cite{edmonds1972}.

In this work, we use the \textit{auction algorithm} to solve the linear assignment problem. The algorithm was developed by Bertsekas in 1979 \cite{bertsekas1979}. In our setting, it exhibits the same time complexity as the Hungarian algorithm, namely $O(n^3)$. 
It is parallelizable and often exhibits scalability, making it suitable for large-scale assignment problems. Furthermore, it offers a unique perspective on the optimization problem by simulating a bidding process among the vertices to achieve an optimal assignment. 

Because of the analogy to the auction setting, we call the vertices from one set bidders and the vertices from the other set items. Then,
the \textit{naive auction algorithm} operates in successive bidding rounds. We denote by $U$ the set of all unassigned bidders and by $p_{i}$ the price of the item $i$, which is initialized to zero.
The algorithm proceeds as follows:

\begin{itemize}
    \item[$(i)$] Pick any unassigned bidder $i \in U$ and choose an item $j_{i}$ with the lowest cost, that is, 
    \begin{equation*}
        j_{i} \in \argmin_{k=1, \dots ,n} \{d(i,k) + p_{k}\}.
    \end{equation*}

    \item[$(ii)$] Assign bidder $i$ to item $j_{i}$. If item $j_{i}$ was previously assigned to another bidder $k$, add $k$ back to $U$.
    
    \item[$(iii)$] Update the price of item $j_{i}$ to the level where bidder $i$ is indifferent between $j_{i}$ and the second best item, that is,
    \begin{equation*}
        p_{j_{i}} \leftarrow p_{j_{i}} + \gamma_{i},
    \end{equation*}
    where 
    \begin{equation*}
        \gamma_{i} = \min_{j \neq j_{i}}\{d(i,j) + p_{j}\} - (d(i,j_{i}) + p_{j_{i}}).
    \end{equation*}
    \item[$(iv)$] If $U$ is empty, terminate; otherwise go back to step $(i)$.
\end{itemize}

We called this process the naive auction algorithm as it is seriously flawed. The bidding increment $\gamma_{i}$ can be zero if more than one item minimizes the cost. In this case, it can happen that several bidders bid for a smaller number of items without raising their prices, thereby creating a never-ending cycle.
To solve this, we introduce a positive scalar $ \epsilon > 0$ and change the updating rule of the prices to 
\begin{equation*}
    p_{j_{i}} \leftarrow p_{j_{i}} + \gamma_{i} + \epsilon.
\end{equation*}
In other words, each bid must raise the price by a minimum positive increment $\epsilon$. This guarantees the convergence of the algorithm. The cost of the final assignment differs from the optimal solution by no more than $n\epsilon$. For a formal proof of these two statements, we refer the reader to Bertsekas's paper \cite{bertsekas1979}.
As the distances $d(i,j)$ are all integers, we conclude that the assignment obtained upon termination is indeed optimal if \begin{equation*}
    \epsilon < \frac{1}{n}.
\end{equation*}
It is noteworthy that the final prices lie within an $n\epsilon$ range of the optimal solution of the dual problem. Therefore, the auction algorithm can be interpreted as a dual algorithm.

Using the auction algorithm, we implemented a Python program calculating the Lin-Lu-Yau curvature and the $0$-Ollivier-Ricci curvature of regular graphs. The source code is available in a public repository\footnote{\url{https://github.com/moritzmath/Regular-Graph-Curvature}}.

\subsection{Generation of regular graphs}

As we wish to study the distribution of the Ollivier-Ricci curvature on regular graphs, generating exhaustive lists of all $d$-regular graphs with a given number of vertices, up to isomorphism, is of interest. This problem is one of the oldest in the field of constructive combinatorics \cite{meringer1999}. 

\begin{table}
    \centering
    \scriptsize
    \begin{tabular}{|l|c|c|c|c|c|c|c|c|c|c|c|} 
        \hline
        \diagbox{n}{d}
                     & 3 & 4 & 5 & 6 & 7 & 8 & 9 & 10 & 11 & 12 & 13 \\
        \hline
        4              & 1 & 0 & 0 & 0 & 0 & 0 & 0 & 0 & 0 & 0 & 0 \\
        \hline
        5              & 0 & 1 & 0 & 0 & 0 & 0 & 0 & 0 & 0 & 0 & 0\\
        \hline
        6              & 2 & 1 & 1 & 0 & 0 & 0 & 0 & 0 & 0 & 0 & 0\\
        \hline
        7              & 0 & 2 & 0 & 1 & 0 & 0 & 0 & 0 & 0 & 0 & 0\\
        \hline
        8              & 5 & 6 & 3 & 1 & 1 & 0 & 0 & 0 & 0 & 0 & 0\\
        \hline
        9              & 0 & 16 & 0 & 4 & 0  & 1 & 0  & 0 & 0 & 0 & 0\\
        \hline
        10              & 19 & 59 & 60 & 21 & 5 & 1 & 1 & 0 & 0 & 0 & 0\\
        \hline
        11             &  0 & 256 & 0 & 266 & 0 & 6 & 0 & 1 & 0 & 0 & 0\\
        \hline
        12            &  85 & 1544 & 7848 & 7849 & 1547 & 94 & 9 & 1 & 1 & 0 & 0\\
        \hline
        13             &  0 & 10778 & 0 & 367860 & 0 & 10786 & 0 & 10 & 0 & 1 & 0\\
        \hline
        14            &  509 & 88168 & 3459383 & 21609300 & 21609301 & 3459386 & 88193 & 540 & 13 & 1 & 1\\
        \hline
    \end{tabular}
    \caption{Number of connected regular graphs with given number of vertices $n$ and degree $d$.}
    \label{number_regular_graphs}
\end{table}

Already determining the number of $d$-regular, labeled graphs with $n$ vertices, denoted by $\vert \mathcal{G}_{n,d} \vert$, proves to be an exceedingly difficult problem for which no closed formula is currently known. But there are some interesting asymptotic results. Wormald and McKay \cite{MCKAY1990565} conjectured that asymptotically
\begin{equation*}
    \vert \mathcal{G}_{n, d(n)} \vert \sim \sqrt{2} \binom{n-1}{d(n)}^{n}(\lambda^{\lambda}(1-\lambda)^{(1-\lambda)})^{\binom{n}{2}}e^{\frac{1}{4}},
\end{equation*}
where $\lambda = \frac{d(n)}{n-1}$ and $d(n)$ satisfying $1 \leq d(n) \leq n-2$ with $nd(n)$ always an even integer. A complete proof of this conjecture was only presented quite recently.
McKay and Wormald proved the conjecture for $d(n)$ between $cn/\log n$ and $n/2$, for a constant $c >2/3$ in 1990 \cite{MCKAY1990565}, and for $d(n) \in o(\sqrt{n})$ in 1991 \cite{MCKAY1991}. The gap was closed by Liebenau and Wormald in 2017 \cite{liebenau2017asymptotic}. Note that the complementary cases for $d(n)$ larger than $n/2$ follow trivially.

For the asymptotic number of regular graphs up to isomorphism, the study of the automorphism groups of regular graphs is necessary. We refer the interested reader to \cite{MCKAY1984} for more details.

Although no closed formula exists, the number of $d$-regular graphs with $n$ vertices for small values of $n$ has been determined using algorithms designed to generate all regular graphs with a specified degree and number of vertices up to isomorphism. In the following, we use the algorithm \textit{GENREG}, proposed by Meringer in \cite{meringer1999}. 
It not only computes the number of regular graphs for the chosen parameters but also constructs the desired graphs. 
Refer to Table \ref{number_regular_graphs} for the number of connected regular graphs with a given number of vertices $n$ and degree $d$, up to isomorphism, for $n \leq 14$.

\subsection{Results}

In this section, we present some of the results obtained by applying our implementation of the Ollivier-Ricci curvature to the exhaustive lists of regular graphs generated with \textit{GENREG}.

The first result concerns bone idle graphs.

\begin{theorem}
    Let $G=(V,E)$ be a regular graph that satisfies $\vert V \vert < 15$. If $G$ is bone idle, then $G$ is isomorphic to one of the following graphs:
    \begin{itemize}
        \item[$(i)$] The cycle graph $C_{n}$ for $6 \leq n < 15$,
        \item[$(ii)$] the bone idle graphs $BI_{n}$ for $n \in \{6,7\}$ defined in Section \ref{section_bone_idle}.
    \end{itemize}
\end{theorem}

We obtain this result by computing the Lin-Lu-Yau curvature and the 0-Ollivier-Ricci curvature of all regular graphs with fewer than 15 vertices. Due to the rapidly increasing number of regular graphs, this calculation becomes exceedingly complex for regular graphs with more vertices.

The second result classifies all regular graphs with Lin-Lu-Yau curvature equal to one.

\begin{theorem}\thlabel{rigidity_one}
    Assume $G=(V,E)$ is a locally finite, regular graph. Then $\textit{Ric}(G) = 1$ if and only if $G$ is a cocktail party graph, meaning it is $d$-regular with $d$ = $\vert V \vert -2$.
\end{theorem}

For the proof of the theorem, we need the following Bonnet-Myers type theorem on graphs by Lin, Lu, and Yau.
\begin{theorem}[Discrete Bonnet-Myers Theorem \cite{lin2011ricci}]\thlabel{Bonnet_Myers}
    Let $G=(V,E)$ be a locally finite graph. If $\kappa(x,y) \geq k > 0$ for any edge $x \sim y$, then the diameter of graph $G$ is bounded as follows:
    \begin{equation*}
        diam(G) \leq \frac{2}{k}.
    \end{equation*}
\end{theorem}

We are now able to prove \thref{rigidity_one}.

\begin{proof}
    Assume $G$ is $d$-regular with $\textit{Ric}(G) = 1$. Let $x \sim y$ be an arbitrary edge in $G$. Using \thref{Lin_Lu_Yau_in_Z}, we obtain
    \begin{equation*}
        1 = \kappa(x,y) \leq \frac{2 + \vert \triangle(x,y) \vert}{d},
    \end{equation*}
    or equivalently
    \begin{equation*}
        d-2 \leq \vert \triangle(x,y) \vert.
    \end{equation*}
    If $\vert \triangle(x,y) \vert = d-1$, then according to \thref{Lin_Lu_Yau_in_Z}, $\kappa(x,y) = \frac{d+1}{d} > 1$. Hence, 
    $\vert \triangle(x,y) \vert = d-2$ must hold.
    Furthermore, according to the Discrete Bonnet-Myers Theorem, we have 
    $diam(G) \leq 2.$

    We now show that $\vert V \vert = d+2$. To this end, let $x$ be an arbitrary vertex in $G$. Denote by 
    $x_{1}, \dots, x_{d}$ the neighbors of $x$. We have seen $\vert \triangle(x, x_{1}) \vert = d-2$. Hence, without loss of generality, 
    we assume that $x_{1} \sim x_{i}$ for $i = 2, \dots, d-1$. Furthermore, it exists an $z \in V$ such that $x_{1} \sim z$ and  $x \not\sim z$.
    As $\kappa(x,x_{1}) = 1$, there must be a perfect matching between $R_{x}(x,x_{1})= \{x_{d}\}$ and $R_{x_{1}}(x,x_{1})= \{z\}$. Hence,
    $x_{d} \sim z$ must hold. 

    Using $\vert \triangle(x,x_{d}) \vert = d-2$ and $x_{1} \not\sim x_{d}$, we have $x_{d} \sim x_{i}$ for $i = 2,\dots, d-1$. Combining that with
    $\vert \triangle(z,x_{d}) \vert = d-2$, $z \not\sim x$ and $x_{d} \not\sim x_{1}$, we conclude that $x_{i} \sim z$ for $i = 2,\dots, d-1$.

    Now, assume there exists $y \in V \setminus \{x,z, x_{1}, \dots, x_{d}\}$. As $diam(G) \leq 2$ holds, we have $d(x,y)=2$. Hence, there exists an $i \in \{1, \dots, d\}$ such that $x \sim x_{i} \sim y$.
    Using $\vert \triangle(x,x_{i}) \vert = d-2$ and $x_{i} \sim z$, we conclude $\vert S_{1}(x_{i}) \vert\geq d+1$, contradicting the regularity of $G$. Thus, $V = \{x,z, x_{1}, \dots, x_{d}\}$ and therefore
    \begin{equation*}
        \vert V \vert = d+2.
    \end{equation*}

    Conversely, suppose $G$ is $d$-regular with $d= \vert V \vert - 2$. Let $x \sim y$ be an arbitrary edge in $G$. There is no $z \in V$ such that $z \not\sim x$ and $z \not\sim y$, as this would imply $ d_{z} \leq \vert V \vert - 3$. Hence,
    \begin{equation*}
        \vert V \vert = \vert S_{1}(x) \cup S_{1}(y) \vert = \vert S_{1}(x) \vert + \vert S_{1}(y) \vert - \vert \triangle(x,y) \vert,
    \end{equation*}
    which implies
    \begin{equation*}
        \vert \triangle(x,y) \vert = \vert V \vert-4 = d-2.
    \end{equation*}
    Hence, $\vert R_{x}(x,y) \vert = \vert R_{y}(x,y) \vert = 1$. 
    Denote by $z_{1}$ the vertex in $R_{x}(x,y)$ and by $z_{2}$ the vertex in $R_{y}(x,y)$. As $d_{z_{1}}= \vert V \vert - 2$ and $z_{1} \not\sim y$, we conclude $z_{1} \sim z_{2}$.
    Thus, using our precise formula for the Lin-Lu-Yau curvature, we obtain
    \begin{equation*}
        \kappa(x,y) = 1.
    \end{equation*}
    Since the edge $x \sim y$ was chosen arbitrarily, we conclude that 
    \begin{equation*}
        \textit{Ric}(G) = 1.
    \end{equation*}
\end{proof}

\thref{rigidity_one} can be seen as an extension of the following rigidity theorem, which was established by Bonini et al. in \cite{bonini2020condensed}.

\begin{theorem}[\cite{bonini2020condensed}, Theorem 3.3]
    Let $G=(V,E)$ be a locally finite graph. Then $\textit{Ric}(G) > 1$ if and only if $G$ is complete. 
\end{theorem}

Next, we present a condition on the degree of a regular graph $G$, guaranteeing that it satisfies $\textit{Ric}(G)> 0$.
Using the exhaustive lists of regular graphs generated with \textit{GENREG} and our implementation, we calculate the number of 
connected, regular graphs that satisfy $\textit{Ric(G)} > 0$, up to isomorphism. Refer to Table \ref{number_ric_positive_graphs} for the results.

\begin{table}
    \centering
    \scriptsize
    \begin{tabular}{|l|c|c|c|c|c|c|c|c|c|c|c|} 
        \hline
        \diagbox{n}{d}
                     & 3 & 4 & 5 & 6 & 7 & 8 & 9 & 10 & 11 & 12 & 13 \\
        \hline
        4              & 1 & 0 & 0 & 0 & 0 & 0 & 0 & 0 & 0 & 0 & 0\\
        \hline
        5              & 0 & 1 & 0 & 0 & 0 & 0 & 0 & 0 & 0 & 0 & 0\\
        \hline
        6              & 2 & 1 & 1 & 0 & 0 & 0 & 0 & 0 & 0 & 0 & 0\\
        \hline
        7              & 0 & 2 & 0 & 1 & 0 & 0 & 0 & 0 & 0 & 0 & 0\\
        \hline
        8              & 2 & 6 & 3 & 1 & 1 & 0 & 0 & 0 & 0 & 0 & 0\\
        \hline
        9              & 0 & 15 & 0 & 4 & 0  & 1 & 0  & 0 & 0 & 0 & 0 \\
        \hline
        10              & 1 & 32 & 60 & 21 & 5 & 1 & 1 & 0 & 0 & 0 & 0\\
        \hline
        11             &  0 & 68 & 0 & 266 & 0 & 6 & 0 & 1 & 0 & 0 & 0\\
        \hline
        12            &  0 & 44 & 6569 & 7849 & 1547 & 94 & 9 & 1 & 1 & 0 & 0 \\
        \hline
        13             &  0 & 23 & 0 & 366804 & 0 & 10786 & 0 & 10 & 0 & 1 & 0\\
        \hline
        14            &  0 & 9 & 719285 & 20998083 & 21609251 & 3459386 & 88193 & 540 & 13 & 1 & 1\\
        \hline
    \end{tabular}
    \caption{Number of connected regular graphs with given number of vertices $n$ and degree $d$, that satisfy $\textit{Ric}(G)>0$.}
    \label{number_ric_positive_graphs}
\end{table}

A comparison with the total number of connected, regular graphs in Table \ref{number_regular_graphs} leads to the conjecture that $d$-regular graphs, for sufficiently large $d$, must satisfy $\textit{Ric}(G) > 0$. The following lemma provides an initial lower bound, ensuring that this property holds.

\begin{lemma}\thlabel{initial_lower_bound}
    Let $G=(V,E)$ be a finite graph. If $G$ is $d$-regular with
    \begin{equation*}
        d > \frac{2\vert V \vert}{3}-1,
    \end{equation*}
    then $\textit{Ric(G)}>0$.
\end{lemma}

To prove this, we need the following lemma.

\begin{lemma}\thlabel{bound_diameter}
    Let $G=(V,E)$ be a finite graph. If $G$ is $d$-regular with 
    \begin{equation*}
        d \geq \frac{\vert V \vert-1}{2},
    \end{equation*}
    then $diam(G) \leq 2$.
\end{lemma}

\begin{proof}
    Assume that $diam(G) > 2$. Hence, there exist vertices $x,y \in V$ such that $d(x,y) >2$, and therefore $S_{1}(x) \cap S_{1}(y) = \emptyset$. As $x,y \notin S_{1}(x) \cup S_{1}(y)$, we have 
    \begin{equation*}
        \vert S_{1}(x) \cup S_{1}(y) \vert \leq \vert V \vert - 2.
    \end{equation*}
    This contradicts 
    \begin{equation*}
        \vert S_{1}(x) \cup S_{1}(y)\vert = \vert S_{1}(x) \vert + \vert S_{1}(y)\vert = 2d \geq \vert V \vert - 1.
    \end{equation*}
    Therefore, our assumption was incorrect, and it follows that $diam(G) \leq 2$ must hold.
\end{proof}

We are now in a position to prove \thref{initial_lower_bound}.

\begin{proof}
    According to \thref{bound_diameter}, we have $diam(G) \leq 2$.
    Now let $x \sim y$ be an arbitrary edge in $G$. Then,
    \begin{equation*}
        \vert V \vert \geq \vert S_{1}(x) \cup S_{1}(y) \vert = 2d - \vert \triangle(x,y) \vert,
    \end{equation*}
    or equivalently, $\vert \triangle(x,y) \vert \geq 2d - \vert V \vert$. Thus, we obtain 
    \begin{align*}
        \kappa(x,y) &= \frac{1}{d}\Biggl(d+1 - \inf_{\phi \in \mathcal{A}_{xy}} \mathlarger{\sum}_{z \in R_{x}(x,y)}d(z, \phi(z))\Biggr) \\
                    &\geq \frac{1}{d}\Biggl(d+1 - 2 \vert R_{x}(x,y) \vert \Biggr) \\
                    &= \frac{1}{d}\Biggl(d+1 - 2 (d-1-\vert \triangle(x,y) \vert) \Biggr) \\
                    &\geq \frac{1}{d}\Biggl(3d + 3 - 2\vert V \vert \Biggr). 
    \end{align*}
    Utilizing the lower bound on $d$, we conclude that $\kappa(x,y) > 0$.
\end{proof}

This initial lower bound can be improved by further investigating the structure of optimal assignments. To this end, we need the following technical lemmata.

\begin{lemma}\thlabel{inequality_triangles}
    Let $G=(V,E)$ be a finite graph. Assume that $G$ is $d$-regular with
    \begin{equation*}
        \frac{2\vert V \vert}{3} - 2 <  d \leq  \frac{2\vert V \vert}{3} - 1.
    \end{equation*}
    If an edge $x \sim y$ satisfies $\vert \triangle(x,y) \vert > 2d - \vert V \vert$, then $\kappa(x,y) > 0$. 
\end{lemma}

\begin{proof}
    According to \thref{bound_diameter}, we have $diam(G) \leq 2$. If $\vert \triangle(x,y) \vert \geq 2d - \vert V \vert +2$, we obtain
    \begin{align*}
        \kappa(x,y) &\geq \frac{1}{d}\Biggl(d+1 - 2 \vert R_{x}(x,y) \vert \Biggr) \\
                    &\geq \frac{1}{d}\Biggl(3d +7 -2\vert V \vert\Biggr).
    \end{align*}
    Utilizing the lower bound on $d$, we conclude that $\kappa(x,y) > 0$.

    Now, assume $\vert \triangle(x,y) \vert = 2d - \vert V \vert + 1$. If $R_{x}(x,y) = \emptyset$, then 
    \begin{equation*}
        d = 1 + \vert R_{x}(x,y) \vert + \vert \triangle(x,y) \vert = \vert V \vert -2.
    \end{equation*}
    Using \thref{rigidity_one}, we conclude $\kappa(x,y) = 1 > 0$. If $R_{x}(x,y) \neq \emptyset$, we have 
    \begin{equation}\label{eq: 2002}
        \vert V \setminus (S_{1}(x) \cup S_{1}(y)) \vert = 1,
    \end{equation}
    and 
    \begin{equation*}
        \vert R_{x}(x,y) \vert + \vert R_{y}(x,y) \vert = 2(\vert V \vert - d - 2).
    \end{equation*}
    Using our assumptions on $d$, we must have
    \begin{equation*}
        \vert R_{x}(x,y) \vert + \vert R_{y}(x,y) \vert > d - 2.
    \end{equation*}
    Thus, without loss of generality, there exists an $j\in \triangle(x,y)$ and $i\in R_{x}(x,y)$ such that $i \not\sim j$. The reason for this is that $j$ has a degree of $d$ and is adjacent to both $x$ and $y$. Hence, $i$ has at most $d-3$ neighbors in $S_{1}(x)$, as it is not adjacent to $y$, $j$ and $i$. Using equality \ref{eq: 2002}, we conclude that there exists an $k\in R_{x}(x,y)$ such that $i \sim k$. Thus, an optimal assignment must satisfy $\vert \square(\phi) \vert \geq 1$, leading to
    \begin{align*}
        \kappa(x,y) &= \frac{1}{d}\Biggl(d+1 - \inf_{\phi \in \mathcal{A}_{xy}} \mathlarger{\sum}_{z \in R_{x}(x,y)}d(z, \phi(z))\Biggr) \\
                    &\geq \frac{1}{d}\Biggl(d+1 - \Big(2(\vert R_{x}(x,y) \vert - 1) + 1\Big) \Biggr) \\
                    &= \frac{1}{d}\Biggl(3d +6 - 2\vert V \vert\Biggr).
    \end{align*}
    Utilizing the lower bound on $d$, we conclude that $\kappa(x,y)> 0$.
\end{proof}

\begin{lemma}\thlabel{equality_triangles}
    Let $G=(V,E)$ be a finite graph. Assume that $G$ is $d$-regular with
    \begin{equation*}
        \frac{2\vert V \vert}{3} - 2 <  d \leq  \frac{2\vert V \vert}{3} - 1.
    \end{equation*}
    If an edge $x \sim y$ satisfies $\vert \triangle(x,y) \vert = 2d - \vert V \vert$, then $\kappa(x,y) > 0$. 
\end{lemma}

The proof of this result relies on the famous Hall's marriage theorem.

\begin{theorem}[Hall's marriage theorem]
    Let $G=(A,B,E)$ be a bipartite graph. Then there exists a perfect matching between $A$ and a subset of $B$ if and only if $\vert N(X) \vert \geq \vert X \vert$ for every $X \subseteq A$, where $N(X) = \bigcup_{x\in X}S_{1}(x)$.
\end{theorem}

As a consequence of Hall's marriage theorem, we obtain the following Corollary.

\begin{corollary}\thlabel{regular_bipartite_graph}
    Let $G=(A,B,E)$ be a bipartite graph. If $G$ is regular, then there exists a perfect matching between $A$ and $B$.
\end{corollary}

Utilizing these results, we present the following proof for \thref{equality_triangles}.

\begin{proof}
    According to \thref{bound_diameter}, we have $diam(G) \leq 2$. Same arguments as in the proof of \thref{inequality_triangles} show that $\kappa(x,y) > 0$ if $\vert R_{x}(x,y) \vert < 2$. Therefore, we now assume $\vert R_{x}(x,y) \vert \geq 2$. 
    
    Let $i\in R_{x}(x,y)$ be arbitrary. Then $i$ is adjacent to at most $d-2$ vertices in $S_{1}(x)$, because $i \not\sim y$. As $V = S_{1}(x) \cup S_{1}(y)$, we conclude that there exists an $j\in R_{y}(x,y)$ such that $i\sim j$. Similarly, every vertex in $R_{y}(x,y)$ is adjacent to a vertex in $R_{x}(x,y)$. 

    Now, denote by $H$ the induced bipartite subgraph consisting of all edges connecting a vertex in $R_{x}(x,y)$ with a vertex in $R_{y}(x,y)$. According to the preceding argument, every vertex in $H$ is of degree at least one in $H$. Denote by $\Delta(G)$ the maximal degree of the induced bipartite subgraph $H$. We proceed with a proof by cases.

    \emph{Case 1: $\Delta(G) = 1$.} The induced bipartite subgraph $H$ is $1$-regular. According to \thref{regular_bipartite_graph}, there exists a perfect matching between $R_{x}(x,y)$ and $R_{y}(x,y)$. Hence, we obtain
    \begin{equation*}
        \kappa(x,y) = \frac{2 + \vert \triangle(x,y) \vert}{d} > 0.
    \end{equation*}

    \emph{Case 2: $1 < \Delta(G) < \vert R_{y}(x,y) \vert -1$.} Let $z_{1}\in R_{x}(x,y)$ be, without loss of generality, a vertex of degree $\Delta(G)$ in $H$. There exist distinct vertices $i,j \in R_{y}(x,y)$ such that $z_{1}$ is not adjacent to both $i$ and $j$. Choose $z_{2}\in R_{x}(x,y)$ such that $z_{2} \sim i$. If the degree of $z_{2}$ in $H$ is greater than one, choose $z_{3} \in R_{x}(x,y) \setminus \{z_{1}, z_{2}\}$ arbitrary. Otherwise, choose $z_{3}\in R_{x}(x,y)$ such that $z_{3} \sim j$. By construction, every $X \subseteq \{z_{1}, z_{2}, z_{3}\}$ satisfies $\vert N(X) \vert \geq \vert X \vert$ in $H$. Hall's marriage theorem implies that an optimal matching $\phi \in \mathcal{O}_{xy}$ between $R_{x}(x,y)$ and $R_{y}(x,y)$ must satisfy $\vert \square(\phi) \vert \geq 3$ and thus
    \begin{align*}
        \kappa(x,y) &=  \frac{1}{d}\Biggl(d+1 - \inf_{\phi \in \mathcal{A}_{xy}} \mathlarger{\sum}_{z \in R_{x}(x,y)}d(z, \phi(z))\Biggr) \\ 
                    &\geq \frac{1}{d}\Biggl(d + 1 - \Big(2(\vert R_{x}(x,y) \vert-3) + 3\Big)\Biggr) \\
                    &= \frac{1}{d}\Biggl(3d + 6 - 2\vert V \vert\Biggr).
    \end{align*}
    Utilizing the lower bound on $d$, we conclude that $\kappa(x,y)> 0$.

    \emph{Case 3: $\Delta(G) \geq \vert R_{y}(x,y) \vert -1$}. Let $z_{1}\in R_{x}(x,y)$ be, without loss of generality, a vertex of degree $\Delta(G)$. If $d = (2 \vert V \vert - 3)/3$, straightforward calculations show that $\kappa(x,y) > 0$. If $d < (2 \vert V \vert - 3)/3$, we obtain
    \begin{equation}\label{eq: 20000}
        \vert R_{x}(x,y)\vert + \vert R_{y}(x,y) \vert -1 > d.
    \end{equation}
    Thus, there exists an $z_{2} \in R_{x}(x,y)$ such that $z_{1} \not\sim z_{2}$. Otherwise, according to inequality \ref{eq: 20000}, the degree of $z_{1}$ would be greater than $d$ in $G$. Hence, the degree of $z_{2}$ is greater than or equal to two in $H$. Choose $z_{3} \in R_{x}(x,y)\setminus\{z_{1},z_{2}\}$ arbitrary. Then, by construction, every $X\subseteq \{z_{1}, z_{2}, z_{3}\}$ satisfies $\vert N(X) \vert \geq \vert X \vert$ in $H$. Hall's marriage theorem implies that an optimal matching $\phi \in \mathcal{O}_{xy}$ between $R_{x}(x,y)$ and $R_{y}(x,y)$ satisfies $\vert \square(\phi) \vert \geq 3$, leading to $\kappa(x,y)> 0$.

    This concludes the proof.
\end{proof}

Utilizing these lemmata, we obtain the following main theorem.

\begin{theorem}\thlabel{lower_bound_d}
    Let $G=(V,E)$ be a finite graph. If $G$ is $d$-regular with
    \begin{equation*}
        d > \frac{2\vert V \vert}{3}-2,
    \end{equation*}
    then $\textit{Ric(G)}>0$.
\end{theorem}

\begin{proof}
    Since $S_{1}(x) \cup S_{1}(y) \subseteq V$, the inequality 
    \begin{equation*}
        2d - \vert V \vert \leq \vert  \triangle(x,y) \vert
    \end{equation*}
    always holds true. If $2d - \vert V \vert < \vert  \triangle(x,y) \vert$, the claim follows from \thref{inequality_triangles}. If $2d - \vert  V \vert = \vert  \triangle(x,y) \vert$, the claim follows from \thref{equality_triangles}. This concludes the proof of \thref{lower_bound_d}.
\end{proof}

\begin{figure}
    \centering   
    \begin{tikzpicture}[x=1.5cm, y=1.5cm,
            vertex/.style={
                shape=circle, fill=black, inner sep=1.5pt	
            }
        ]
        \node[vertex, label=below:$x$] (1) at (0, 0) {};
        \node[vertex, label=below:$y$] (2) at (2, 0) {};
        \node[vertex] (3) at (-0.5, 1) {};
        \node[vertex] (4) at (0, 1) {};
        \node[vertex] (5) at (0.5, 0.5) {};
        \node[vertex] (6) at (2, 1) {};
        \node[vertex] (7) at (2.5, 1) {};
        \node[vertex] (8) at (1.5, 0.5) {};
        \node[vertex] (9) at (1, 2) {};

        \draw[red] (1)  -- node[midway,below] {$\kappa=0$} ++ (2);
        \draw (1) -- (3);
        \draw (1) -- (4);
        \draw (1) -- (5);

        \draw (2) -- (6);
        \draw (2) -- (7);
        \draw (2) -- (8);

        \draw (3) -- (4);
        \draw (3) -- (5);
        \draw (4) -- (5);

        \draw (6) -- (7);
        \draw (6) -- (8);
        \draw (7) -- (8);

        \draw (9) -- (3);
        \draw (9) -- (4);
        \draw (9) -- (7);
        \draw (9) -- (6);
        \draw (5) -- (8);

    \end{tikzpicture}
    \caption{Illustration of a $4$-regular graph $G$ with 9 vertices that does not satisfy $\textit{Ric}(G)>0$.}
    \label{counter_example_bound_d}
\end{figure}

\begin{remark}
    This property does not hold for $d = \frac{2 \vert V \vert}{3}-2$, as the $d$-regular graph $G=(V,E)$ with $\vert V \vert = 9$ and $d = 4$, illustrated in Figure \ref{counter_example_bound_d}, shows.
\end{remark}

The following corollary is an immediate consequence of \thref{larger_implies_equal_zero}.

\begin{corollary}
    Let $G=(V,E)$ be a finite graph. If $G$ is $d$-regular with 
    \begin{equation*}
        d > \frac{2\vert V \vert}{3}-2,
    \end{equation*}
    then $\kappa_{0}(x,y)\geq0$, for every edge $x\sim y$.
\end{corollary}

Next, we discuss the sharpness of the lower bound in \thref{lower_bound_d}.

\begin{theorem}\thlabel{asymptotic_sharpness}
    Let $n \in \mathbb{N}$ such that
    \begin{equation*}
        d = \frac{2n-8}{3} \in 2\mathbb{N}.
    \end{equation*}
    and $d \geq 12$. Then, there exists a $d$-regular graph on $n$ vertices such that $\kappa(x,y) =0$ and $\kappa_{0}(x,y) < 0$ for an edge $x \sim y$.
\end{theorem}

\begin{figure}
    \centering  
    \scalebox{0.8}{

\tikzset{every picture/.style={line width=0.75pt}} 

\begin{tikzpicture}[x=0.75pt,y=0.75pt,yscale=-1,xscale=1]

\draw   (172,150.75) .. controls (172,132.66) and (186.66,118) .. (204.75,118) .. controls (222.84,118) and (237.5,132.66) .. (237.5,150.75) .. controls (237.5,168.84) and (222.84,183.5) .. (204.75,183.5) .. controls (186.66,183.5) and (172,168.84) .. (172,150.75) -- cycle ;
\draw [color={rgb, 255:red, 208; green, 2; blue, 27 }  ,draw opacity=1 ]   (271.44,41.06) .. controls (48.03,149.54) and (49.03,150.54) .. (271.5,259.33) ;
\draw [shift={(271.5,259.33)}, rotate = 26.06] [color={rgb, 255:red, 208; green, 2; blue, 27 }  ,draw opacity=1 ][fill={rgb, 255:red, 208; green, 2; blue, 27 }  ,fill opacity=1 ][line width=0.75]      (0, 0) circle [x radius= 3.35, y radius= 3.35]   ;
\draw [shift={(271.44,41.06)}, rotate = 154.1] [color={rgb, 255:red, 208; green, 2; blue, 27 }  ,draw opacity=1 ][fill={rgb, 255:red, 208; green, 2; blue, 27 }  ,fill opacity=1 ][line width=0.75]      (0, 0) circle [x radius= 3.35, y radius= 3.35]   ;
\draw    (220.63,179.78) -- (271.5,259.33) ;
\draw    (220,121) -- (271.44,41.06) ;
\draw   (270,113) .. controls (270,101.95) and (301,93) .. (339.25,93) .. controls (377.5,93) and (408.5,101.95) .. (408.5,113) .. controls (408.5,124.05) and (377.5,133) .. (339.25,133) .. controls (301,133) and (270,124.05) .. (270,113) -- cycle ;
\draw   (269,199) .. controls (269,187.95) and (300.9,179) .. (340.25,179) .. controls (379.6,179) and (411.5,187.95) .. (411.5,199) .. controls (411.5,210.05) and (379.6,219) .. (340.25,219) .. controls (300.9,219) and (269,210.05) .. (269,199) -- cycle ;
\draw    (271.44,41.06) -- (339.25,93) ;
\draw    (271.5,259.33) -- (340.25,219) ;
\draw    (237.5,150.75) -- (269,199) ;
\draw    (237.5,150.75) -- (270,113) ;
\draw    (271.44,41.06) -- (480,117) ;
\draw [shift={(480,117)}, rotate = 20.01] [color={rgb, 255:red, 0; green, 0; blue, 0 }  ][fill={rgb, 255:red, 0; green, 0; blue, 0 }  ][line width=0.75]      (0, 0) circle [x radius= 3.35, y radius= 3.35]   ;
\draw    (340.25,179) -- (480,117) ;
\draw    (408.5,113) -- (511,141) ;
\draw [shift={(511,141)}, rotate = 15.28] [color={rgb, 255:red, 0; green, 0; blue, 0 }  ][fill={rgb, 255:red, 0; green, 0; blue, 0 }  ][line width=0.75]      (0, 0) circle [x radius= 3.35, y radius= 3.35]   ;
\draw    (411.5,199) -- (511,182) ;
\draw [shift={(511,182)}, rotate = 350.3] [color={rgb, 255:red, 0; green, 0; blue, 0 }  ][fill={rgb, 255:red, 0; green, 0; blue, 0 }  ][line width=0.75]      (0, 0) circle [x radius= 3.35, y radius= 3.35]   ;
\draw    (511,182) -- (511,141) ;
\draw    (408.5,113) -- (480,117) ;
\draw    (511,141) -- (411.5,199) ;
\draw    (408.5,113) -- (511,182) ;

\draw (175,141.4) node [anchor=north west][inner sep=0.75pt]    {${\textstyle \triangle \mathnormal{( x,y})}$};
\draw (266,19.4) node [anchor=north west][inner sep=0.75pt]    {$x$};
\draw (266,271.4) node [anchor=north west][inner sep=0.75pt]    {$y$};
\draw (292,105.4) node [anchor=north west][inner sep=0.75pt]    {$R_{x}( x,y) \setminus \{x_{l}\}$};
\draw (313,190.4) node [anchor=north west][inner sep=0.75pt]    {$R_{y}( x,y)$};
\draw (15,139.4) node [anchor=north west][inner sep=0.75pt]  [color={rgb, 255:red, 208; green, 2; blue, 27 }  ,opacity=1 ]  {$\kappa ( x,y) \ =\ 0$};
\draw (488,112.4) node [anchor=north west][inner sep=0.75pt]    {$x_{l}$};
\draw (521,178.4) node [anchor=north west][inner sep=0.75pt]    {$v_{1}$};
\draw (522,138.4) node [anchor=north west][inner sep=0.75pt]    {$v_{2}$};

\end{tikzpicture}                                  
}
\caption{Illustration of the graph constructed in \thref{asymptotic_sharpness}}
\end{figure}

\begin{proof}
    Let the assumptions of \thref{asymptotic_sharpness} hold.
    We construct a $d$-regular graph on $n$ vertices that contains an edge $x \sim y$ such that $\kappa(x,y) = 0$. To this end, define the vertex set
    \begin{equation*}
        V = \{x,y, z_{0}, \dots, z_{l-3}, x_{0}, \dots, x_{l}, y_{0}, \dots, y_{l}, v_{1}, v_{2}\},
    \end{equation*}
    where $l = \frac{d}{2}$. Next, we add the edges
    \begin{itemize}
        \item $x\sim y$,
        \item $x \sim z_{i}$ for $i = 0,\dots, l-3$,
        \item $x \sim x_{i}$ for $i = 0,\dots, l$
        \item $y \sim z_{i}$ for $i = 0,\dots, l-3$,
        \item $y \sim y_{i}$ for $i = 0,\dots, l$.
    \end{itemize}
    to the set of edges $E$. Therefore, $x$ and $y$ are of degree $d$, and
    \begin{equation*}
        \triangle(x,y) = \{z_{0}, \dots, z_{l-3}\}, \quad R_{x}(x,y) = \{x_{0}, \dots, x_{l}\}, \quad R_{y}(x,y)= \{y_{0}, \dots, y_{l}\}.
    \end{equation*}
    Next, we add for every $z_{i} \in \triangle(x,y)$ the edges $z_{i} \sim x_{j}$ for $j = 0, \dots, l-1$ if $j \not= i$
    as well as the edges $z_{i} \sim y_{j}$ for $j = 0, \dots, l$ if 
    \begin{equation*}
        j\not\in\{2i \pmod{l+1}, 2i+1 \pmod{l+1}\}.
    \end{equation*}

    Thus, every $z_{j} \in  \triangle(x,y)$ is of degree $d$, as it has $d-2$ neighbors in $R_{x}(x,y) \cup R_{y}(x,y)$ and is adjacent to both $x$ and $y$. 
    
    On the other hand, every $x_{i}$ is adjacent to $\vert \triangle(x,y) \vert - 1$ vertices in $\triangle(x,y)$ if $i \leq l-3$, and to all vertices in $\triangle(x,y)$ if $i \in \{l-2, l-1\}$. Every $y_{i}$ is adjacent to $\vert \triangle(x,y) \vert -2$ vertices in $\triangle(x,y)$ if $i \leq l-6$ and to $\vert \triangle(x,y) \vert - 1$ otherwise.

    Next, we add the edges $x_{i} \sim x_{j}$ for $i \not=j$, except $x_{l-2} \sim x_{l-1}$. Analogously, we add the edges $y_{i} \sim y_{j}$ for $i \not= j$, as well as $x_{l} \sim y_{j}$ for $j = 0, \dots, l-2$.

    Thus, every $x_{i} \in R_{x}(x,y)\setminus\{x_{l}\}$ has degree $d-2$ and $x_{l}$ has degree $d$. On the other hand, $y_{i}$ has degree $d-1$ if $i \in \{l-5, l-4, l-3, l-2\}$ and $d-2$ otherwise.

    Finally, we add the edges
    \begin{itemize}
        \item $v_{1} \sim x_{i}$ and $v_{2} \sim x_{i}$ for $i = 0, \dots, l-1$.
        \item $v_{1} \sim y_{i}$ and $v_{2} \sim y_{i}$ for $i = 0, \dots, l-6$ and $i =l-1, l$.
        \item $v_{1} \sim y_{l-5}, v_{1} \sim y_{l-4}, v_{2} \sim y_{l-3}, v_{2} \sim y_{l-2}$ and $v_{1} \sim v_{2}$.
    \end{itemize} 
    The constructed graph $G=(V,E)$ is $d$-regular. Now, denote by $H$ the induced bipartite subgraph consisting of all edges connecting a vertex in $R_{x}(x,y)$ with a vertex in $R_{y}(x,y)$. By construction, every $x_{j}$ for $0 \leq j \leq l-1$ is of degree zero in $H$. Therefore, according to Hall's marriage theorem, an optimal assignment $\phi \in \mathcal{O}_{xy}$ must satisfy $\vert \{z \in R_{x}(x,y): d(z, \phi(z)) = 1\} \vert = 1$. Thus, we obtain

    \begin{align*}
        \kappa(x,y) &= \frac{1}{d}\Biggl(d+1 - \inf_{\phi \in \mathcal{A}_{xy}} \mathlarger{\sum}_{z \in R_{x}(x,y)}d(z, \phi(z))\Biggr) \\ 
        &= \frac{1}{d}\Biggl(d+1 - 2 (\vert R_{x}(x,y) \vert - 1) - 1\Biggr) \\
        &= 0.
    \end{align*}
    Furthermore, according to \thref{kappa_vergleich}, we obtain
    \begin{align*}
        \kappa_{0}(x,y) &= \kappa(x,y) - \frac{1}{d}\Big(3-\sup_{\phi\in \mathcal{O}_{xy}}\sup_{z\in R_x(x,y)}d(z,\phi(z))\Big)\\
        &= -\frac{1}{d}.
    \end{align*}
\end{proof}

\subsection{Maximal regular graphs with positive Ricci curvature}

We finish this article with an investigation of the maximal number of vertices that an $d$-regular graph $G$ satisfying $Ric(G) > 0$ can have,
\begin{equation*}
    \mathcal{M}_{d} = \max_{\substack{\text{$G$ $d$-regular},\\ Ric(G) > 0}} \vert V(G) \vert.
\end{equation*}

The following Proposition shows that $\mathcal{M}_{d}$ is always finite.

\begin{proposition}\thlabel{bound_Md}
    Let $d\in \mathbb{N}$. Then
    \begin{equation*}
        \mathcal{M}_{d} \leq 1 + d \sum_{i=0}^{2d-1}(d-1)^{i}.
    \end{equation*}
\end{proposition}

The proof requires the following Moore Bound, which gives an upper bound on the number of vertices in a regular graph with a specified diameter.

\begin{theorem}[Moore Bound]
    Let $G=(V,E)$ be a $d$-regular graph with diameter $k$. Then 
    \begin{equation*}
        \vert V \vert \leq 1 + d \sum_{i=0}^{k-1}(d-1)^{i}.
    \end{equation*}
\end{theorem}

\begin{proof}[Proof of \thref{bound_Md}]
    Let $G$ be an arbitrary $d$-regular graph, satisfying $Ric(G) >0$. According to \thref{Lin_Lu_Yau_in_Z}, we have $\kappa(x,y) \in \mathbb{Z}/d$ for every edge $x \sim y$. Therefore, $Ric(G) \geq \frac{1}{d}$ must hold. Using \thref{Bonnet_Myers}, we obtain 
    \begin{equation*}
        diam(G) \leq 2d.
    \end{equation*}
    The proposition is now an immediate consequence of the Moore Bound.
\end{proof}

The following theorem quantifies $\mathcal{M}_{d}$ for $d$ equal to two and three.

\begin{theorem}\thlabel{3_reg}
    Let $G=(V,E)$ be a $d$-regular graph, satisfying $Ric(G) >0$.
     If $d =2$, then $\vert V \vert \leq 5$. If $d=3$, then $\vert V \vert \leq 10$. Consequently, we have $\mathcal{M}_{2} = 5$ and $\mathcal{M}_{3} = 10$.
\end{theorem}

For the proof of \thref{3_reg}, we need the following classification result by Cushing et al \cite{cushing2022graph}.

\begin{theorem}\thlabel{classification}
    Let $G=(V,E)$ be a 3-regular graph. Then $\kappa_{0}(x,y) \geq 0$ for all edges $x \sim y$ if and only if $G$ is a prism graph $Y_{n}$ for some $n \geq 3$ or a M{\"o}bius ladder $M_{n}$ for some $n \geq 2$.
\end{theorem}

\begin{figure}
    \centering 
    \scalebox{0.8}{  
    \tikzset{every picture/.style={line width=0.75pt}} 

        \begin{tikzpicture}[x=0.75pt,y=0.75pt,yscale=-1,xscale=1]

        \draw   (180,121) -- (230,121) -- (230,171) -- (180,171) -- cycle ;
        \draw   (230,121) -- (280,121) -- (280,171) -- (230,171) -- cycle ;
        \draw  [dash pattern={on 0.84pt off 2.51pt}]  (280,121) -- (330,121.33) ;
        \draw [shift={(280,121)}, rotate = 0.38] [color={rgb, 255:red, 0; green, 0; blue, 0 }  ][fill={rgb, 255:red, 0; green, 0; blue, 0 }  ][line width=0.75]      (0, 0) circle [x radius= 3.35, y radius= 3.35]   ;
        \draw  [dash pattern={on 0.84pt off 2.51pt}]  (280,171) -- (331,171.33) ;
        \draw [shift={(280,171)}, rotate = 0.37] [color={rgb, 255:red, 0; green, 0; blue, 0 }  ][fill={rgb, 255:red, 0; green, 0; blue, 0 }  ][line width=0.75]      (0, 0) circle [x radius= 3.35, y radius= 3.35]   ;
        \draw  [dash pattern={on 0.84pt off 2.51pt}]  (381,171) -- (430.11,171.01) ;
        \draw [shift={(430.11,171.01)}, rotate = 0.01] [color={rgb, 255:red, 0; green, 0; blue, 0 }  ][fill={rgb, 255:red, 0; green, 0; blue, 0 }  ][line width=0.75]      (0, 0) circle [x radius= 3.35, y radius= 3.35]   ;
        \draw  [dash pattern={on 0.84pt off 2.51pt}]  (381,121) -- (430.11,121.01) ;
        \draw [shift={(430.11,121.01)}, rotate = 0.01] [color={rgb, 255:red, 0; green, 0; blue, 0 }  ][fill={rgb, 255:red, 0; green, 0; blue, 0 }  ][line width=0.75]      (0, 0) circle [x radius= 3.35, y radius= 3.35]   ;
        \draw   (430.11,121.01) -- (480.11,121.01) -- (480.11,171.01) -- (430.11,171.01) -- cycle ;
        \draw    (180,121) .. controls (330.11,60.31) and (330.11,60.31) .. (481.67,121) ;
        \draw [shift={(481.67,121)}, rotate = 25.01] [color={rgb, 255:red, 0; green, 0; blue, 0 }  ][fill={rgb, 255:red, 0; green, 0; blue, 0 }  ][line width=0.75]      (0, 0) circle [x radius= 3.35, y radius= 3.35]   ;
        \draw [shift={(180,121)}, rotate = 334.78] [color={rgb, 255:red, 0; green, 0; blue, 0 }  ][fill={rgb, 255:red, 0; green, 0; blue, 0 }  ][line width=0.75]      (0, 0) circle [x radius= 3.35, y radius= 3.35]   ;
        \draw    (180,171) .. controls (330.11,230.01) and (330.11,230.01) .. (481.67,171) ;
        \draw [shift={(481.67,171)}, rotate = 335.52] [color={rgb, 255:red, 0; green, 0; blue, 0 }  ][fill={rgb, 255:red, 0; green, 0; blue, 0 }  ][line width=0.75]      (0, 0) circle [x radius= 3.35, y radius= 3.35]   ;
        \draw [shift={(180,171)}, rotate = 24.69] [color={rgb, 255:red, 0; green, 0; blue, 0 }  ][fill={rgb, 255:red, 0; green, 0; blue, 0 }  ][line width=0.75]      (0, 0) circle [x radius= 3.35, y radius= 3.35]   ;
        \draw    (280,121) -- (230,121) ;
        \draw [shift={(230,121)}, rotate = 180] [color={rgb, 255:red, 0; green, 0; blue, 0 }  ][fill={rgb, 255:red, 0; green, 0; blue, 0 }  ][line width=0.75]      (0, 0) circle [x radius= 3.35, y radius= 3.35]   ;
        \draw    (280,171) -- (230,171) ;
        \draw [shift={(230,171)}, rotate = 180] [color={rgb, 255:red, 0; green, 0; blue, 0 }  ][fill={rgb, 255:red, 0; green, 0; blue, 0 }  ][line width=0.75]      (0, 0) circle [x radius= 3.35, y radius= 3.35]   ;

        \draw (162,105) node [anchor=north west][inner sep=0.75pt]  [font=\footnotesize]  {$x_{1}$};
        \draw (225,105) node [anchor=north west][inner sep=0.75pt]  [font=\footnotesize]  {$x_{2}$};
        \draw (274,105) node [anchor=north west][inner sep=0.75pt]  [font=\footnotesize]  {$x_{3}$};
        \draw (415,105) node [anchor=north west][inner sep=0.75pt]  [font=\footnotesize]  {$x_{n-1}$};
        \draw (485,105) node [anchor=north west][inner sep=0.75pt]  [font=\footnotesize]  {$x_{n}$};
        \draw (147,175) node [anchor=north west][inner sep=0.75pt]  [font=\footnotesize]  {$x_{n+1}$};
        \draw (219,175) node [anchor=north west][inner sep=0.75pt]  [font=\footnotesize]  {$x_{n+2}$};
        \draw (266,175) node [anchor=north west][inner sep=0.75pt]  [font=\footnotesize]  {$x_{n+3}$};
        \draw (413,175) node [anchor=north west][inner sep=0.75pt]  [font=\footnotesize]  {$x_{2n-1}$};
        \draw (483.67,175) node [anchor=north west][inner sep=0.75pt]  [font=\footnotesize]  {$x_{2n}$};
    \end{tikzpicture}
    }
    \caption{The prism graph $Y_{n}$}
    \label{prism_graph}
\end{figure}

\begin{figure}
    \centering 
    \scalebox{0.8}{  
    \tikzset{every picture/.style={line width=0.75pt}} 

        \begin{tikzpicture}[x=0.75pt,y=0.75pt,yscale=-1,xscale=1]

        \draw   (180,121) -- (230,121) -- (230,171) -- (180,171) -- cycle ;
        \draw   (230,121) -- (280,121) -- (280,171) -- (230,171) -- cycle ;
        \draw  [dash pattern={on 0.84pt off 2.51pt}]  (280,121) -- (330,121.33) ;
        \draw [shift={(280,121)}, rotate = 0.38] [color={rgb, 255:red, 0; green, 0; blue, 0 }  ][fill={rgb, 255:red, 0; green, 0; blue, 0 }  ][line width=0.75]      (0, 0) circle [x radius= 3.35, y radius= 3.35]   ;
        \draw  [dash pattern={on 0.84pt off 2.51pt}]  (280,171) -- (331,171.33) ;
        \draw [shift={(280,171)}, rotate = 0.37] [color={rgb, 255:red, 0; green, 0; blue, 0 }  ][fill={rgb, 255:red, 0; green, 0; blue, 0 }  ][line width=0.75]      (0, 0) circle [x radius= 3.35, y radius= 3.35]   ;
        \draw  [dash pattern={on 0.84pt off 2.51pt}]  (381,171) -- (430.11,171.01) ;
        \draw [shift={(430.11,171.01)}, rotate = 0.01] [color={rgb, 255:red, 0; green, 0; blue, 0 }  ][fill={rgb, 255:red, 0; green, 0; blue, 0 }  ][line width=0.75]      (0, 0) circle [x radius= 3.35, y radius= 3.35]   ;
        \draw  [dash pattern={on 0.84pt off 2.51pt}]  (381,121) -- (430.11,121.01) ;
        \draw [shift={(430.11,121.01)}, rotate = 0.01] [color={rgb, 255:red, 0; green, 0; blue, 0 }  ][fill={rgb, 255:red, 0; green, 0; blue, 0 }  ][line width=0.75]      (0, 0) circle [x radius= 3.35, y radius= 3.35]   ;
        \draw   (430.11,121.01) -- (480.11,121.01) -- (480.11,171.01) -- (430.11,171.01) -- cycle ;
        \draw    (180,121) -- (480.11,171.01) ;
        \draw [shift={(480.11,171.01)}, rotate = 9.46] [color={rgb, 255:red, 0; green, 0; blue, 0 }  ][fill={rgb, 255:red, 0; green, 0; blue, 0 }  ][line width=0.75]      (0, 0) circle [x radius= 3.35, y radius= 3.35]   ;
        \draw [shift={(180,121)}, rotate = 9.46] [color={rgb, 255:red, 0; green, 0; blue, 0 }  ][fill={rgb, 255:red, 0; green, 0; blue, 0 }  ][line width=0.75]      (0, 0) circle [x radius= 3.35, y radius= 3.35]   ;

        \draw    (180,171) -- (481.67,121) ;
        \draw [shift={(480.11,121.01)}, rotate = 350.54] [color={rgb, 255:red, 0; green, 0; blue, 0 }  ][fill={rgb, 255:red, 0; green, 0; blue, 0 }  ][line width=0.75]      (0, 0) circle [x radius= 3.35, y radius= 3.35]   ;
        \draw [shift={(180,171)}, rotate = 350.54] [color={rgb, 255:red, 0; green, 0; blue, 0 }  ][fill={rgb, 255:red, 0; green, 0; blue, 0 }  ][line width=0.75]      (0, 0) circle [x radius= 3.35, y radius= 3.35]   ;
        \draw    (280,121) -- (230,121) ;
        \draw [shift={(230,121)}, rotate = 180] [color={rgb, 255:red, 0; green, 0; blue, 0 }  ][fill={rgb, 255:red, 0; green, 0; blue, 0 }  ][line width=0.75]      (0, 0) circle [x radius= 3.35, y radius= 3.35]   ;
        \draw    (280,171) -- (230,171) ;
        \draw [shift={(230,171)}, rotate = 180] [color={rgb, 255:red, 0; green, 0; blue, 0 }  ][fill={rgb, 255:red, 0; green, 0; blue, 0 }  ][line width=0.75]      (0, 0) circle [x radius= 3.35, y radius= 3.35]   ;

        \draw (162,105) node [anchor=north west][inner sep=0.75pt]  [font=\footnotesize]  {$x_{1}$};
        \draw (225,105) node [anchor=north west][inner sep=0.75pt]  [font=\footnotesize]  {$x_{2}$};
        \draw (274,105) node [anchor=north west][inner sep=0.75pt]  [font=\footnotesize]  {$x_{3}$};
        \draw (415,105) node [anchor=north west][inner sep=0.75pt]  [font=\footnotesize]  {$x_{n-1}$};
        \draw (485,105) node [anchor=north west][inner sep=0.75pt]  [font=\footnotesize]  {$x_{n}$};
        \draw (147,175) node [anchor=north west][inner sep=0.75pt]  [font=\footnotesize]  {$x_{n+1}$};
        \draw (219,175) node [anchor=north west][inner sep=0.75pt]  [font=\footnotesize]  {$x_{n+2}$};
        \draw (266,175) node [anchor=north west][inner sep=0.75pt]  [font=\footnotesize]  {$x_{n+3}$};
        \draw (413,175) node [anchor=north west][inner sep=0.75pt]  [font=\footnotesize]  {$x_{2n-1}$};
        \draw (483.67,175) node [anchor=north west][inner sep=0.75pt]  [font=\footnotesize]  {$x_{2n}$};
    \end{tikzpicture}
    }
    \caption{The M{\"o}bius Ladder $M_{n}$}
    \label{moebius_graph}
\end{figure}

Figure \ref{prism_graph} illustrates the prism graph $Y_{n}$, while Figure \ref{moebius_graph} provides a visualization of the M{\"o}bius ladder $M_{n}$. With this classification result established, we are now prepared to prove \thref{3_reg}.

\begin{proof}[Proof of \thref{3_reg}]
    The only $2$-regular, connected graphs are the cycle graphs $C_{n}$. It is straightforward to verify that $Ric(C_{n}) = 0$ for $n \geq 6$. Hence, if $G=(V,E)$ is $2$-regular satisfying $Ric(G) >0$, then $G=C_{n}$ for some $3 \leq n \leq5$ and therefore $\vert V \vert \leq 5$ holds true.

    We now turn to the case where $d=3$. Let $x\sim y$ be an arbitrary edge in $G$. By assumption, we have $\kappa(x,y) > 0$. Since $G$ is regular, it follows from \thref{larger_implies_equal_zero} that $\kappa_{0}(x,y) \geq 0$. Thus, according to \thref{classification}, $G$ is a prism graph $Y_{n}$ for some $n \geq 3$ or a M{\"o}bius ladder $M_{n}$ for some $n \geq 2$. The prism graph $Y_{n}$ does not satisfy $Ric(Y_{n}) > 0$ for $n \geq 6$ and the M{\"o}bius ladder $M_{n}$ does not satisfy $Ric(M_{n})>0$ for $n \geq 5$. Therefore, $G$ has at most 10 vertices.
\end{proof}

The values $\mathcal{M}_{d}$ for $d\geq 4$ are still unknown. The Cartesian product of $n$ 5-cycles $(C_{5})^{\square n}$ is $2n$-regular, has $5^{n}$ vertices and satisfies 
\begin{equation*}
    Ric((C_{5})^{\square n})= \frac{1}{2n} > 0.
\end{equation*}
Thus, $\mathcal{M}_{d}$ exhibits exponential growth, satisfying
\begin{equation*}
    \mathcal{M}_{d} \geq \sqrt{5}^{d}
\end{equation*}
for $d$ even.

Recall that $\kappa(x,y) \in \mathbb{Z}/d$ for every edge $x\sim y$ of an $d$-regular graph. Hence $(C_{5})^{\square n}$ exhibits constant, minimal positive Lin-Lu-Yau curvature for an $2n$-regular graph. This led us to formulate the following conjecture.

\begin{conjecture}
    For $d\in 2\mathbb{N}$, 
    \begin{equation*}
        \mathcal{M}_{d} = \sqrt{5}^{d}.
    \end{equation*}
\end{conjecture}

Moreover, for odd $d$, $\mathcal{M}_{d}$ also exhibits exponential growth, as e.g. the $n$-dimensional hypercube $\mathcal{Q}_{n}$ for odd $n$ shows.

\TOCstop

\subsection*{Acknowledgments}

I would like to express my deepest gratitude to Prof. Renesse and Dr. Münch for their invaluable support and insightful feedback throughout the course of this work. Their guidance and inspiration have been instrumental in the completion of this work.

\printbibliography

@article{MCKAY1984,
title = {Automorphisms of random graphs with specified vertices.},
journal = {Combinatorica},
volume = {4},
pages = {325–338},
year = {1984},
author = {Brendan D. McKay and Nicholas C. Wormald}
}

@article{liebenau2017asymptotic,
  title={Asymptotic enumeration of graphs by degree sequence, and the degree sequence of a random graph},
  author={Liebenau, Anita and Wormald, Nick},
  journal={arXiv preprint arXiv:1702.08373},
  year={2017}
}

@article{MCKAY1991,
title = {Asymptotic Enumeration by Degree Sequence of Graphs with degrees $o(n^{1/2})$},
journal = {Combinatorica},
volume = {11},
pages = {369-382},
year = {1991},
author = {Brendan D. McKay and Nicholas C. Wormald}
}

@article{cushing2021curvatures,
  title={Curvatures, graph products and Ricci flatness},
  author={Cushing, David and Kamtue, Supanat and Kangaslampi, Riikka and Liu, Shiping and Peyerimhoff, Norbert},
  journal={Journal of Graph Theory},
  volume={96},
  number={4},
  pages={522--553},
  year={2021},
  publisher={Wiley Online Library}
}

@article{MCKAY1990565,
title = {Asymptotic Enumeration by Degree Sequence of Graphs of High Degree},
journal = {European Journal of Combinatorics},
volume = {11},
number = {6},
pages = {565-580},
year = {1990},
author = {Brendan D. McKay and Nicholas C. Wormald}
}

@article{meringer1999,
author = {Markus Meringer},
year = {1999},
title = { Fast Generation of Regular Graphs and Construction of Cages},
journal = {Journal of Graph Theory 30},
pages = {137-146}
}

@article{bertsekas1979,
author = {Dimitri P. Bertsekas},
year = {1979},
month = {03},
title = {A Distributed Algorithm for the Assignment Problem},
journal = {Lab. for Information and Decision Systems Working Paper}
}

@article{edmonds1972,
author = {Edmonds, Jack and Karp, Richard},
year = {1972},
month = {04},
pages = {248-264},
title = {Theoretical Improvement in Algorithmic Efficiency for Network Flow Problems},
volume = {19},
journal = {Journal of the ACM}
}

@article{Kuhn1955Hungarian,
  author = {Kuhn, Harold W.},
  journal = {Naval Research Logistics Quarterly},
  month = {03},
  number = {1--2},
  pages = {83--97},
  title = {The Hungarian Method for the Assignment Problem},
  volume = 2,
  year = 1955
}

@article{McSHANE1934ExtensionOR,
  title={Extension of range of functions},
  author={BY E. J. McShane},
  journal={Bulletin of the American Mathematical Society},
  year={1934},
  volume={40},
  pages={837-842}
}

@article{cushing2018erratum,
  title={Erratum for Ricci-flat graphs with girth at least five},
  author={Cushing, David and Kangaslampi, Riikka and Lin, Yong and Liu, Shiping and Lu, Linyuan and Yau, Shing-Tung},
  journal={arXiv preprint arXiv:1802.02979},
  year={2018}
}

@article{bhattacharya2015exact,
  title={Exact and asymptotic results on coarse Ricci curvature of graphs},
  author={Bhattacharya, Bhaswar B and Mukherjee, Sumit},
  journal={Discrete Mathematics},
  volume={338},
  number={1},
  pages={23--42},
  year={2015},
  publisher={Elsevier}
}

@article{jost2021characterizations,
  title={Characterizations of Forman curvature},
  author={Jost, J{\"u}rgen and M{\"u}nch, Florentin},
  journal={arXiv preprint arXiv:2110.04554},
  year={2021}
}

@article{samal2018comparative,
  title={Comparative analysis of two discretizations of Ricci curvature for complex networks},
  author={Samal, Areejit and Sreejith, RP and Gu, Jiao and Liu, Shiping and Saucan, Emil and Jost, J{\"u}rgen},
  journal={Scientific reports},
  volume={8},
  number={1},
  year={2018},
  publisher={Nature Publishing Group UK London}
}

@article{munch2019ollivier,
  title={Ollivier Ricci curvature for general graph Laplacians: heat equation, Laplacian comparison, non-explosion and diameter bounds},
  author={M{\"u}nch, Florentin and Wojciechowski, Rados{\l}aw K},
  journal={Advances in Mathematics},
  volume={356},
  pages={106759},
  year={2019},
  publisher={Elsevier}
}

@article{bonini2020condensed,
  title={Condensed Ricci curvature of complete and strongly regular graphs},
  author={Bonini, Vincent and Carroll, Conor and Dinh, Uyen and Dye, Sydney and Frederick, Joshua and Pearse, Erin},
  journal={Involve, a Journal of Mathematics},
  volume={13},
  number={4},
  pages={559--576},
  year={2020},
  publisher={Mathematical Sciences Publishers}
}

@inproceedings{Ye2020Curvature,
title={Curvature Graph Network},
author={Ze Ye and Kin Sum Liu and Tengfei Ma and Jie Gao and Chao Chen},
booktitle={International Conference on Learning Representations},
year={2020}
}

@article{li2022curvature,
  title={Curvature graph neural network},
  author={Li, Haifeng and Cao, Jun and Zhu, Jiawei and Liu, Yu and Zhu, Qing and Wu, Guohua},
  journal={Information Sciences},
  volume={592},
  pages={50--66},
  year={2022},
  publisher={Elsevier}
}

@inproceedings{nguyen2023revisiting,
  title={Revisiting over-smoothing and over-squashing using ollivier-ricci curvature},
  author={Nguyen, Khang and Hieu, Nong Minh and Nguyen, Vinh Duc and Ho, Nhat and Osher, Stanley and Nguyen, Tan Minh},
  booktitle={International Conference on Machine Learning},
  pages={25956--25979},
  year={2023},
  organization={PMLR}
}

@article{Sia2019community,
author = {Sia, Jayson and Jonckheere, Edmond and Bogdan, Paul},
year = {2019},
month = {07},
title = {Ollivier-Ricci Curvature-Based Method to Community Detection in Complex Networks},
volume = {9},
journal = {Scientific Reports}
}

@article{ni2019community,
  title={Community detection on networks with Ricci flow},
  author={Ni, Chien-Chun and Lin, Yu-Yao and Luo, Feng and Gao, Jie},
  journal={Scientific reports},
  volume={9},
  number={1},
  pages={9984},
  year={2019},
  publisher={Nature Publishing Group UK London}
}

@article{Farooq2019hallmark,
author={Farooq, Hamza
and Chen, Yongxin
and Georgiou, Tryphon T.
and Tannenbaum, Allen
and Lenglet, Christophe},
title={Network curvature as a hallmark of brain structural connectivity},
journal={Nature Communications},
year={2019},
month={10},
day={30},
volume={10},
number={1}
}

@article{Sandhu2015cancer,
author = {Sandhu, Romeil and Georgiou, Tryphon and Reznik, Ed and Zhu, Liangjia and Kolesov, Ivan and Şenbabaoğlu, Yasin and Tannenbaum, Allen},
year = {2015},
month = {07},
pages = {12323},
title = {Graph Curvature for Differentiating Cancer Networks},
volume = {5},
journal = {Scientific reports},
}

@article{Sandhu2016systematic,
author = {Romeil S. Sandhu  and Tryphon T. Georgiou  and Allen R. Tannenbaum },
title = {Ricci curvature: An economic indicator for market fragility and systemic risk},
journal = {Science Advances},
volume = {2},
number = {5},
pages = {e1501495},
year = {2016}
}

@inproceedings{ni2015ricciinternet,
  title={Ricci curvature of the internet topology},
  author={Ni, Chien-Chun and Lin, Yu-Yao and Gao, Jie and Gu, Xianfeng David and Saucan, Emil},
  booktitle={2015 IEEE conference on computer communications (INFOCOM)},
  pages={2758--2766},
  year={2015},
  organization={IEEE}
}

@incollection{ollivier2010survey,
  title={A survey of Ricci curvature for metric spaces and Markov chains},
  author={Ollivier, Yann},
  booktitle={Probabilistic approach to geometry},
  volume={57},
  pages={343--382},
  year={2010},
  publisher={Mathematical Society of Japan}
}

@article{boguna2021network,
  title={Network geometry},
  author={Boguna, Marian and Bonamassa, Ivan and De Domenico, Manlio and Havlin, Shlomo and Krioukov, Dmitri and Serrano, M {\'A}ngeles},
  journal={Nature Reviews Physics},
  volume={3},
  number={2},
  pages={114--135},
  year={2021},
  publisher={Nature Publishing Group UK London}
}

@article{vonRenesseMax-K.2005Tige,
author = {von Renesse, Max-K. and Sturm, Karl-Theodor},
address = {Hoboken},
copyright = {Copyright © 2004 Wiley Periodicals, Inc.},
journal = {Communications on pure and applied mathematics},
number = {7},
pages = {923-940},
publisher = {Wiley Subscription Services, Inc., A Wiley Company},
title = {Transport inequalities, gradient estimates, entropy and Ricci curvature},
volume = {58},
year = {2005}
}

@book{villani2003topics,
  title={Topics in Optimal Transportation},
  author={Villani, C.},
  isbn={9780821833124},
  series={Graduate studies in mathematics},
  year={2003},
  publisher={American Mathematical Society}
}

@book{diestel2017graph,
  title={Graph Theory},
  author={Diestel, R.},
  isbn={978-3-662-53621-6},
  series={Graduate Texts in Mathematics},
  year={2017},
  publisher={Springer Berlin, Heidelberg}
}

@article{sreejith2016forman,
  title={Forman curvature for complex networks},
  author={Sreejith, RP and Mohanraj, Karthikeyan and Jost, J{\"u}rgen and Saucan, Emil and Samal, Areejit},
  journal={Journal of Statistical Mechanics: Theory and Experiment},
  volume={2016},
  number={6},
  pages={063206},
  year={2016},
  publisher={IOP Publishing}
}

@article{Forman2003Bochner,
  author = {Forman, Robin},
  year = {2003},
  month = {01},
  pages = {323-374},
  title = {Bochner's Method for Cell Complexes and Combinatorial Ricci Curvature},
  volume = {29},
  journal = {Discrete and Computational Geometry}
}

@article{Lin2010Ricci,
  author = {Lin, Yong and Yau, Shing-Tung},
  year = {2010},
  month = {03},
  pages = {343-356},
  title = {Ricci curvature and eigenvalue estimate on locally finite graphs},
  volume = {17},
  journal = {Mathematical research letters}
}

@article{Liu2018Bakry,
author = {Liu, Shiping and Münch, Florentin and Peyerimhoff, Norbert},
year = {2018},
month = {03},
pages = {},
title = {Bakry-Emery curvature and diameter bounds on graphs},
volume = {57},
journal = {Calculus of Variations and Partial Differential Equations}
}

@article{Bakry1985,
author = {Bakry and Dominique and Émery and Michel},
journal = {Séminaire de probabilités de Strasbourg},
pages = {177-206},
publisher = {Springer - Lecture Notes in Mathematics},
title = {Diffusions hypercontractives},
volume = {19},
year = {1985}
}

@article{erbar2012ricci,
  title={Ricci curvature of finite Markov chains via convexity of the entropy},
  author={Erbar, Matthias and Maas, Jan},
  journal={Archive for Rational Mechanics and Analysis},
  volume={206},
  pages={997--1038},
  year={2012},
  publisher={Springer}
}

@article{Mielke2013GeodesicCO,
  title={Geodesic convexity of the relative entropy in reversible Markov chains},
  author={Alexander Mielke},
  journal={Calculus of Variations and Partial Differential Equations},
  year={2013},
  volume={48},
  pages={1-31}
}

@article{hamilton1982three,
  title={Three-manifolds with positive Ricci curvature},
  author={Hamilton, Richard S},
  journal={Journal of Differential geometry},
  volume={17},
  number={2},
  pages={255--306},
  year={1982},
  publisher={Lehigh University}
}

@article{ollivier2009ricci,
  title={Ricci curvature of Markov chains on metric spaces},
  author={Ollivier, Yann},
  journal={Journal of Functional Analysis},
  volume={256},
  number={3},
  pages={810--864},
  year={2009},
  publisher={Elsevier}
}

@article{bourne2018ollivier,
  title={Ollivier-Ricci Idleness Functions of Graphs},
  author={Bourne, David P and Cushing, David and Liu, Shiping and Münch, F and Peyerimhoff, Norbert},
  journal={SIAM Journal on Discrete Mathematics},
  volume={32},
  number={2},
  pages={1408--1424},
  year={2018},
  publisher={SIAM}
}

@article{cushing2019long,
  title={Long-scale Ollivier Ricci curvature of graphs},
  author={Cushing, David and Kamtue, Supanat},
  journal={Analysis and Geometry in Metric Spaces},
  volume={7},
  number={1},
  pages={22--44},
  year={2019},
  publisher={De Gruyter Open}
}

@article{cushing2022graph,
  title={The Graph Curvature Calculator and the curvatures of cubic graphs},
  author={Cushing, David and Kangaslampi, Riikka and Lipi{\"a}inen, Valtteri and Liu, Shiping and Stagg, George W},
  journal={Experimental Mathematics},
  volume={31},
  number={2},
  pages={583--595},
  year={2022},
  publisher={Taylor \& Francis}
}

@article{lin2011ricci,
  title={Ricci curvature of graphs},
  author={Lin, Yong and Lu, Linyuan and Yau, Shing-Tung},
  journal={Tohoku Mathematical Journal, Second Series},
  volume={63},
  number={4},
  pages={605--627},
  year={2011},
  publisher={Mathematical Institute, Tohoku University}
}

@book{gromov1981structures,
  title={Structures m{\'e}triques pour les vari{\'e}t{\'e}s riemanniennes},
  author={Gromov, M. and Lafontaine, J. and Pansu, P.},
  isbn={9782712407148},
  series={Textes math{\'e}matiques. Recherche},
  year={1981},
  publisher={CEDIC/Fernand Nathan}
}

Moritz Hehl,\\
Department of Mathematics, University of Leipzig, Leipzig, Germany\\
\texttt{moritz.hehl@uni-leipzig.de}\\

\TOCstart

\end{document}